\documentclass[12pt]{amsart}
\usepackage{amssymb}
\usepackage{amsmath}
\usepackage{amsthm}
\usepackage{times}
\usepackage{changepage}
\usepackage{caption}
\usepackage{geometry}
\usepackage[curve,all]{xy}
\xyoption{curve}
\xyoption{line}
\xyoption{arc}
\xyoption{color}
\usepackage{array}
\usepackage{enumitem}
\usepackage[dvipdfmx]{graphicx}
\usepackage[dvipdfmx]{color}
\usepackage{color}
\usepackage{amsthm}
\newtheorem{theorem}{Theorem}[section]
\newtheorem{lemma}[theorem]{Lemma}
\newtheorem{corollary}[theorem]{Corollary}
\newtheorem{proposition}[theorem]{Proposition}
\newtheorem{prop}[theorem]{Proposition}

\usepackage{subfig}
 
\theoremstyle{definition}

\theoremstyle{remark}
\newtheorem{remark}[theorem]{Remark}

\numberwithin{equation}{section}

\newcommand{\calO}{\mathcal{O}}
\newcommand{\calP}{\mathcal{P}}

\newcommand{\calX}{\mathcal{X}}

\newcommand{\la}{\langle}
\newcommand{\ra}{\rangle}

\def\Ker{{\text{Ker}}}
\def\Im{{\text{Im}}}

\def\rank{{\text{rank}}}

\begin{document}
\title [Kummer surfaces]{Kummer surfaces and quadratic line complexes in characteristic two}

\author{Toshiyuki Katsura}
\address{Graduate School of Mathematical Sciences, The University of Tokyo,
Meguro-ku, Tokyo 153-8914, Japan}
\email{tkatsura@ms.u-tokyo.ac.jp}
\thanks{}

\author{Shigeyuki Kond\=o}
\address{Graduate School of Mathematics, Nagoya University, Nagoya,
464-8602, Japan}
\email{kondo@math.nagoya-u.ac.jp}
\thanks{Research of the first author is partially supported by JSPS Grant-in-Aid 
for Scientific Research (C) No.23K03066 and the second author by JSPS
Grant-in-Aid for Scientific Research (A) No.20H00112.}

\begin{abstract}
In this paper, we study the classical theory of quadratic line complexes and Kummer surfaces.
A quadratic line complex is the intersection of the Grassmannian $G(2,4)$ and a quadric hypersurface 
in ${\bf P}^5$, and a Kummer surface is the quotient of the Jacobian of a curve of genus 2 by the inversion.
F. Klein discovered a relationship between a quadratic line complex and a curve of genus 2, its Jacobian and the associated Kummer surface.  This theory holds in any characteristic not equal to two.  However the situation in
characteristic two is entirely different.  The purpose of this paper is to give 
an analogue in characteristic 2 of this classical theory.
\end{abstract}

\maketitle

\section{Introduction}\label{sec1}

Let $C$ be a compact Riemann surface of genus 2 and $J(C)$ its Jacobian.
Let $\iota$ be the inversion automorphism of $J(C)$ and
$\Theta$ the theta divisor on $J(C)$.  
The complete linear system $|2\Theta|$ gives a double covering from $J(C)$ to
a quartic surface $S$ in ${\bf P}^3$ with 16 nodes
defined by
$$x^4+y^4+z^4+t^4 + A\left(x^2y^2+z^2t^2\right)+B\left(x^2z^2+y^2t^2\right)+C\left(x^2t^2+y^2z^2\right)+Dxyzt=0,$$
where $A, B, C, D$ are constants satisfying a certain explicit equation (e.g. Hudson \cite[p.81]{H}).
The quartic surface $S$ is isomorphic to the quotient surface $J(C)/\la \iota\ra$ and is called a {\it Kummer quartic surface}.  It contains sixteen {\it tropes} (= double conics) which are the images of $\Theta$ 
and its translations by 
2-torsions.  The incidence relation between sixteen nodes and tropes is called 
a $(16_6)$-{\it configuration}, that is, each node is contained in six tropes and each trope 
contains six nodes.  Let $\Sigma$ be the minimal resolution of $S$ which is a $K3$ surface.  
The surface $\Sigma$ is realized as a complete intersection of three quadrics in ${\bf P}^5$
which is the image of the rational map from $J(C)$ defined by the linear system 
$|4\Theta - \sum p_i|$ where $\{p_i\}$ are sixteen 2-torsion points of $A$ 
(cf. Griffiths-Harris \cite[p.786]{GH}). 
Then $\Sigma$ contains 32 lines forming a $(16_6)$-configuration, that is, there are two sets of disjoint 16 lines on $\Sigma$ such that each member in one set meets exactly six members in another set.  The 32 lines consist of the proper 
transforms of sixteen tropes and sixteen exceptional curves over sixteen nodes.  We remark that
by contracting the proper transforms of sixteen tropes we obtain another quartic surface $S^\vee$
with 16 nodes which is the projective dual of $S$.

In the paper \cite{Klein}, Klein discovered a relationship between Kummer quartic surfaces 
and quadratic line complexes.  A {\it quadratic line complex} is a 3-dimensional family of 
lines in ${\bf P}^3$ which is defined as the intersection of the Grassmannian 
$G=G(2,4) \subset {\bf P}^5$ with a quadric $Q$ in ${\bf P}^5$.
In this paper we assume that $G\cap Q$ is non-singular.  Then we can diagonalize simultaneously
$G$ and $Q$ as
\begin{equation}\label{diagonal}
G= \left\{\sum_{i=1}^6 X_i^2 =0\right\}, \quad Q=\left\{\sum_{i=1}^6 a_iX_i^2 =0\right\},
\end{equation}
where $(X_1,\cdots, X_6)$ are homogeneous coordinates of ${\bf P}^5$ and 
$a_1,\ldots, a_6 \in {\bf C}$, $a_i\not=a_j (i\not=j)$.
Recall that a non-singular quadric in 
${\bf P}^5$ contains two irreducible families of planes and a singular quadric of rank 5
contains an irreducible family of planes.  
Since the pencil of quadrics in ${\bf P}^5$ defined by $G$ and $Q$ contains exactly six
singular quadrics of rank 5, we thus obtain a {\it non-singular curve $C$ of genus} 2 defined by
\begin{equation}\label{genus2}
y^2 = \prod_{i=1}^6 (x - a_i)
\end{equation}
parameterizing the irreducible families of planes contained in members of the pencil. 
The classical theory claims that the surface $\Sigma$ is isomorphic to the locus of special lines, 
called {\it singular lines}, and $S$ is isomorphic to the set of {\it foci} of these singular lines
(see (\ref{KummerQuartic1}), (\ref{Kummer}) for the definitions).  
The surface $\Sigma$ is given by the intersection of three quadrics: 
\begin{equation}\label{threeQuadrics}
\Sigma=\left\{\sum_{i=1}^6 X_i^2 = \sum_{i=1}^6 a_iX_i^2 =\sum_{i=1}^6 a_i^2X_i^2 =0\right\}.
\end{equation}
Moreover the {\it variety of lines} in $G\cap Q$ is an abelian surface which is 
isomorphic to the Jacobian $J(C)$ of $C$.  This is an outline of the classical theory.  
For quadratic line complexes we refer the reader to Jessop \cite{J}. 
In the last century, Narasimhan and Ramanan \cite{NR} and independently Newstead \cite{N} reproved this 
in connection with the theory of vector bundles over $C$.
For modern treatment of this theory, we refer the reader to Griffiths and Harris \cite{GH}, Cassels and Flynn
\cite{CF}, Dolgachev \cite{DoC} and for a history of Kummer surfaces to Dolgachev \cite{Do}.
The above theory holds in any characteristic different from 2.

Now assume that the ground field is an algebraically closed field of characteristic 2. 
Then the situation is entirely different.
There are mainly two differences between the case of characteristic $p=2$ and the case of $p\not= 2$.
First of all, a quadratic form in characteristic 2 is not determined by 
the associated alternating bilinear form.
Secondly the moduli space of curves of genus two and that of their Jacobians are stratified in terms
of $2$-{\it rank}. There are three types, that is, {\it ordinary} ($2$-rank 2), $2$-{\it rank} $1$ 
and {\it supersingular} ($2$-rank $0$).  
Shioda \cite{Sh} and the first author \cite{Ka} found that  
$J(C)/\la \iota\ra$ 
has, instead of sixteen nodes, four rational double points of type $D_4$ for $J(C)$ being ordinary, two rational double points of type $D_8$ for $J(C)$ with $2$-rank 1, 
an elliptic singularity of type 
$\raise0.2ex\hbox{\textcircled{\scriptsize{4}}}_{0,1}^1$ 
in the sense of Wagreich \cite{W} 
for $J(C)$ being supersingular.  See Table \ref{Table1} (see Figure \ref{ellipticSing} in Subsection \ref{Char2} for elliptic singularities).

\begin{table}[!htb]
\centering
\begin{tabular}{|>{\centering\arraybackslash}m{5cm}|>{\centering\arraybackslash}m{1cm}|>{\centering\arraybackslash}m{2.5cm}|>{\centering\arraybackslash}m{1.5cm}|>{\centering\arraybackslash}m{3cm}|}
\hline
{\rm char $p$} & $p\ne 2$ & $p=2$ &  $p=2$ & $p=2$\\  \hline \hline
{\rm $2$-rank of $J(C)$} & -- & $2$\ {\rm (ordinary)} & $1$ & $0$ \ {(supersingular)}\\ \hline
\rm{$\#$ of branches of $C\to {\bf P}^1$} & $6$ & $3$ & $2$  & $1$ \\  \hline
\rm{$\#$ of $2$-torsion points of $J(C)$} & $16$ & $4$ & $2$ & $1$  \\  \hline
\rm{Singularities of $J(C)/\la \iota \ra$} & $16$ $A_1$  & $4$ $D_4$  &  $2$ $D_8$   & $\raise0.2ex\hbox{\textcircled{\scriptsize{4}}}_{0,1}^1$   \\  \hline
\end{tabular}
\caption{Difference between $p\ne 2$ and $p=2$}
\label{Table1}
\end{table}

In characteristic 2, Bhosle \cite{B} studied pencils of quadrics in ${\bf P}^{2g+1}$ 
(For the recent progress on pencils of quadrics in ${\bf P}^{2g}$ we refer the reader to Dolgachev and Duncan \cite{DD}). 
She presented a canonical form of a pencil of quadrics in ${\bf P}^{2g+1}$, associated
a hyperelliptic curve $C$ of genus $g$ to the pencil in general case (that is, ordinary case for $g=2$) and showed that 
the Jacobian of $C$ is isomorphic to the variety of $(g-1)$-dimensional subspaces on the
base locus of the pencil.  Also Laszlo and Pauly \cite{LP}, \cite{LP2} studied the linear system $|2\Theta|$ and gave the equation of Kummer quartic surface
in ordinary case, Ducrohet \cite{Duc} in supersingular case and Duquesne \cite{Du} in arbitrary case. 

The main purpose of this paper is to present an analogue of the theory of Kummer surfaces 
and quadratic line complexes in characteristic 2.  We show that the stratification of the moduli space of
curves of genus 2 by the $2$-rank bijectively corresponds to the one by the canonical forms of pencils 
$\{\lambda G + \mu Q\}_{(\lambda ,\mu)\in {\bf P}^1}$
of quadrics.  Let $A, B$ be the associated alternating forms of $G, Q$, respectively.
Then there are three possibilities of the associated alternating forms
$\left(
    \begin{array}{cc}
       0 & \lambda A+\mu B \\
      ^t(\lambda A+ \mu B)& 0 \\
    \end{array}
  \right)$, 
where 
$A=
\left(
    \begin{array}{ccc}
       0 & 0 & 1 \\
      0 & 1 & 0 \\
      1 & 0 & 0 \\
    \end{array}
  \right)$
and $B$ is 
as in Table \ref{Table2}.  

\begin{table}[!htb]
\centering
\begin{tabular}{|>{\centering\arraybackslash}m{3cm}|>{\centering\arraybackslash}m{3cm}|>{\centering\arraybackslash}m{3cm}|>{\centering\arraybackslash}m{3cm}|}
\hline
{\rm $2$-rank of $J(C)$} &  $2$ & $1$ & $0$ \\ \hline
$B$ & 
$\left(
    \begin{array}{ccc}
       0 & 0 & a_1 \\
      0 & a_2 & 0 \\
      a_3 & 0 & 0 \\
    \end{array}
  \right)$
 & 
$\left(
    \begin{array}{ccc}
       0 & 0 & a_1 \\
      0 & a_2 & 0 \\
      a_2 & 1 & 0 \\
    \end{array}
  \right)$
  & 
$\left(
    \begin{array}{ccc}
       0 & 0 & a_1 \\
      0 & a_1 & 1 \\
      a_1 & 1 & 0 \\
    \end{array}
  \right)$  \\  \hline
\end{tabular}
\caption{Alternating forms}
\label{Table2}
\end{table}
\noindent
Here $a_i\ne a_j$ for $i\ne j$ (see Proposition \ref{CanAlt} and the equations (\ref{(a)}), (\ref{(b)}),
(\ref{(c)}) in Proposition \ref{CanFormPencilQuadrics}).

For each canonical form of a pencil of quadrics, we associate to quartic surfaces $S$ and $S^\vee$, 
an intersection of three quadrics $\Sigma$, a non-singular curve $C$ of genus 2 and its Jacobian $J(C)$ 
in terms of line geometry. 

We remark that in characteristic $p\ne 2$ a quadratic form is determined by the associated bilinear form.
Under the condition $G\cap Q$ being non-singular, the case of non-diagonalizable pairs $G, Q$ is excluded.
On the other hand, in the case $p=2$ a quadratic form is not determined by the associated alternating form.
This difference allows the possibility of the above three cases of alternating forms, 
and hence the exsistence of three types of curves of genus 2.

We now state the main theorems of this paper.  
We can take the following canonical forms of two quadrics $\{G, Q\}$ for ordinary case (a), $\{G, Q_1\}$ for
$p$-rank 1 case (b) and $\{G, Q_0\}$ for supersingular case (c) (see Proposition \ref{CanFormPencilQuadrics}):
\[
\left\{ \begin{array}{l}
G =G(2,4) : \sum_{i=1}^3 X_iY_i=0, \\
\\
{\rm (a)} \ \ Q : \sum_{i=1}^3 a_iX_iY_i + c_iX_i^2 + d_iY_i^2=0,\\
\\
{\rm (b)} \ \ Q_1 : \sum_{i=1}^3 \left(a_iX_iY_i + c_iX_i^2 + d_iY_i^2\right) + B =0 ,\\
\\
{\rm (c)} \ \ Q_0 :  \sum_{i=1}^3 \left(aX_iY_i + c_iX_i^2 + d_iY_i^2\right) + C = 0,\\
\end{array} \right.
\]
where $(X_1,X_2,X_3,Y_1,Y_2,Y_3)$ are homogeneous coordinates of ${\bf P}^5$,
\[
B = b_1X_2Y_3 + b_2X_3Y_2 + b_3X_2X_3 + b_4Y_2Y_3,
\]
\[
C= b_1X_2Y_3 + b_2X_3Y_2 + b_3X_2X_3 +b_4Y_2Y_3 + b_5X_1Y_2+b_6X_2Y_1+b_7X_1X_2+b_8Y_1Y_2,
\]
and $a_i, b_i, c_i, d_i, a$ are constants with $a_i\ne a_j (i\ne j)$ in case (a) and 
$a_1\ne a_2=a_3$ in case (b).
We denote by $\calX, \calX_1, \calX_0$ the intersection of two quadrics $G\cap Q, G\cap Q_1, G\cap Q_0$, 
respectively.
In this paper we assume that $\calX, \calX_1, \calX_0$ are non-singular.  We give the following necessary and sufficient condition
for being non-singular in Lemma \ref{3-fold}:

${\rm (a)}$ $\calX$ is non-singular iff $\prod_i c_id_i\ne 0$.

${\rm (b)}$ $\calX_1$ is non-singular iff 
$c_1d_1\ne 0$ and $b_2b_4c_2 + b_1b_4c_3 + b_1b_3d_2 + b_2b_3d_3 \ne 0$.

${\rm (c)}$ $\calX_0$ is non-singular iff 
$c_1(b_1b_8+b_4b_6)(b_2b_6+b_3b_8)+c_3(b_1b_5+b_4b_7)(b_1b_8+b_4b_6)
+d_1(b_1b_5+b_4b_7)(b_2b_7+b_3b_5)+d_3(b_2b_6+b_3b_8)(b_2b_7+b_3b_5) \ne 0.$

Now let $S, S_1$ or $S_0 \subset {\bf P}^3$ be the set of foci of singular lines of 
$\calX, \calX_1$ or $\calX_0$, respectively (see (\ref{KummerQuartic1}) for the definition), and
let $(x,y,z,t)$ be homogeneous coordinates of ${\bf P}^3$.

\begin{theorem}\label{KummerQuartic} \ 
{\rm (a)}\ The surface $S$ is given by the equation
$$(a_1+a_2)^2(c_3x^2y^2+d_3z^2t^2) + (a_1+a_3)^2(c_2x^2z^2+d_2y^2t^2)$$
$$+ (a_2+a_3)^2(c_1x^2t^2 + d_1y^2z^2)
+ (a_1+a_2)(a_2+a_3)(a_3+a_1)xyzt=0.
$$

\smallskip
{\rm (b)}\
The surface $S_1$ is given by the equation
$$b_3^2c_1x^4 + b_2^2d_1y^4 + b_1^2d_1z^4 + b_4^2c_1t^4$$ 
$$+(b_3^2d_2 + b_2^2c_2 + (a_1+a_2)^2c_3 + (a_1+a_2)b_2b_3)x^2y^2$$
$$+(b_3^2d_3 + b_1^2c_3 + (a_1+a_2)^2c_2 + (a_1+a_2)b_1b_3)x^2z^2$$
$$+ (b_2^2d_3 + b_4^2c_3 + (a_1+a_2)^2d_2 + (a_1+a_2)b_2b_4)y^2t^2$$
$$+ (b_1^2d_2 + b_4^2c_2 + (a_1+a_2)^2d_3 + (a_1+a_2)b_1b_4)z^2t^2$$
$$+ (a_1+a_2)^2(b_3x^2yz + b_2xy^2t + b_1xz^2t + b_4yzt^2) = 0.$$

\smallskip
{\rm (c)}\
The surface $S_0$ is given by the equation
$$(b_3^2c_1+b_7^2c_3)x^4 + (b_2^2d_1 +b_8^2c_3)y^4 + (b_1^2d_1+b_6^2d_3)z^4 + 
(b_4^2c_1+b_5^2d_3)t^4$$
$$+b_5(b_1b_5+b_4b_7)xt^3 + b_7(b_2b_7+b_3b_5)x^3t + b_2(b_2b_6 + b_3b_8)xy^3 
+ b_8(b_2b_6+b_3b_8)y^3z$$
$$+b_3(b_2b_7+b_3b_5)x^3y + b_4(b_1b_5+b_4b_7)zt^3 + b_6(b_1b_8 + b_4b_6)yz^3 
+ b_1(b_1b_8+b_4b_6)z^3t$$ 
$$+ (b_2^2c_2 + b_3^2d_2)x^2y^2
+(b_1^2c_3+b_3^2d_3 + b_6^2c_1+b_7^2d_1)x^2z^2+(b_5^2c_2+b_7^2d_2)x^2t^2$$
$$+(b_6^2d_2+b_8^2c_2)y^2z^2 + (b_2^2d_3 + b_4^2c_3 + b_5^2d_1 + b_8^2c_1)y^2t^2
+ (b_1^2d_2 + b_4^2c_2)z^2t^2$$
$$+ b_7(b_2b_6+b_3b_8)x^2yz + b_3(b_1b_5+b_4b_7)x^2zt + b_8(b_2b_7+b_3b_5)xy^2t + b_2(b_1b_8+b_4b_6)y^2zt$$
$$+b_1(b_2b_6+b_3b_8)xyz^2 + b_6(b_1b_5+b_4b_7)xz^2t + b_4(b_2b_7+b_3b_5)xyt^2
+ b_5(b_1b_8+b_4b_6)yzt^2=0.$$
\end{theorem}

We also give singularities of $S, S_1, S_0$ and tropes on them explicitly in Theorems \ref{KummerQuarticOrdinary}, \ref{KumQuarticRank1}, \ref{ssKummerQ}.

Next let $\Sigma, \Sigma_1$ or $\Sigma_0$ be the set of singular lines of 
$\calX, \calX_1$ or $\calX_0$, respectively (see (\ref{Kummer}) for the definition).

\begin{theorem}\label{Int3quadrics} \ 
{\rm (a)}\quad 
$\Sigma = \calX \cap \left\{\sum_{i=1}^3 a_i^2X_iY_i=0\right\}$.

\smallskip

{\rm (b)}\quad
$\Sigma_1= \calX_1 \cap \left\{\sum_{i=1}^3 a_i^2X_iY_i + b_1b_3X_2^2 + b_2b_3X_3^2+ b_2b_4Y_2^2+b_1b_4Y_3^2=0\right\}$.

\smallskip

{\rm (c)}\quad  
$\Sigma_0= \calX_0 \cap Q_0'$,

\smallskip
\noindent
where $Q_0'$ is the qudratic defined by the equation 
\[
a^2\sum_{i=1}^3X_iY_i + b_5b_7X_1^2 + (b_1b_3+b_6b_7)X_2^2 + b_2b_3X_3^2+ b_6b_8Y_1^2+(b_2b_4+b_5b_8)Y_2^2+b_1b_4Y_3^2\]
\[
+(b_1b_5+b_4b_7)X_1Y_3 + (b_2b_6+b_3b_8)X_3Y_1+(b_2b_7+b_3b_5)X_1X_3+(b_1b_8+b_4b_6)Y_1Y_3 =0.
\]
\end{theorem}

We give singularities of $\Sigma, \Sigma_1, \Sigma_0$ and lines on them explicitly in Theorems \ref{Int3quadricsOrdinary}, \ref{Int3quadricsP-rank1}, \ref{Int3quadricsSupersingular}.

Finally let $C, C_1$ or $C_0$ be a curve associated with 
$\calX, \calX_1$ or $\calX_0$, respectively (see (\ref{curveOFgenus2}) for the definition).

\begin{theorem}\label{Curve} \   The curves $C, C_1, C_0$ are
non-singular and of genus 2.  Moreover, 

{\rm (a)} \ the curve $C$ is given by
\[
 z^2 + (t + a_1)(t + a_2) (t + a_3)z
 \]
 \[
 = c_1d_1(t + a_2)^2(t + a_3)^2 + 
 c_2d_2(t + a_1)^2(t + a_3)^2 + c_3d_3(t + a_1)^2(t + a_2)^2.
\]

\smallskip
{\rm (b)} \ The curve $C_1$ is given by
$$
\begin{array}{l}
z^2 +(t + a_1)(t + a_2)^2z = (t +a_2)^4c_1d_1\\
 +(t+a_1)^2(t+a_2)^2\left(b_1b_2 + c_2d_2+c_3d_3 + b_1\sqrt{c_3d_2} + b_2\sqrt{c_2d_3} + b_3\sqrt{d_2d_3} + b_4\sqrt{c_2c_3}\right)\\
+(t + a_1)^2(t + a_2)\{b_3(b_1d_2 + b_2d_3) + b_4(b_2c_2 + b_1c_3)\}.
\end{array}
$$

\smallskip
{\rm (c)} \ The curve $C_0$ is given by
$$
\begin{array}{l}
 z^2 + (t + a)^3z  \\
 = (t + a)^4(c_1d_1 + c_2d_2 + c_3d_3 + b_1b_2 + b_5b_6)\\
 + (t + a)^3\{b_6b_8c_1 +(b_2b_4+b_5b_8)c_2 + b_1b_4c_3 +b_5b_7d_1 + (b_1b_3 +b_6b_7)d_2 +b_2b_3d_3\}\\
+ (t+a)^2\{b_8^2c_1c_2 +b_6^2c_1d_2 +b_5^2c_2d_1 +b_7^2d_1d_2 \\
+b_4^2c_2c_3 +b_2^2c_2d_3 +b_1^2c_3d_2 +b_3^2d_2d_3 + b_1b_3b_5b_8 +b_2b_4b_6b_7\}\\
+(t+a)\{(b_1b_3b_8^2 +b_2b_4b_6^2)c_1 +(b_1b_3b_5^2 +b_2b_4b_7^2)d_1 \\
+(b_1^2b_5b_8 +b_4^2b_6b_7)c_3 +(b_3^2b_5b_8 +b_2^2b_6b_7)d_3\}\\
+ (b_1b_8 +b_4b_6)^2c_1c_3 +(b_2b_6+b_3b_8)^2c_1d_3 +(b_1b_5 + b_4 b_7)^2c_3d_1 + (b_2b_7+b_3b_5)^2d_1d_3 .
\end{array}
$$
\end{theorem}
We also give Igusa's normal forms of $C, C_1, C_0$ in Remarks \ref{CurveOrdinary2}, \ref{CurveRank12}, \ref{CurveSupersingularIgusa}, respectively  (see (\ref{g=2}) for Igusa's normal forms).

Combining Theorems \ref{KummerQuartic}, \ref{Int3quadrics}, \ref{Curve} and a general theory of quadratic
line complexes (Propositions \ref{Basic}, \ref{Basic2}), we have the following Corollary.

\begin{corollary}\label{KummerQuarticCor}
The surface $S, S_1$ or $S_0$ is the Kummer quartic surface associated with the curve $C, C_1$ or $C_0$,
respectively, and $\Sigma, \Sigma_1$ or $\Sigma_0$ is a partial resolution of singularities of $S, S_1$ or
$S_0$, respectively.
\end{corollary}

The plan of this paper is as follows. 
In Section \ref{sec2}, we recall the classical theory of quadratic line complexes and Kummer surfaces,
the theory of quadratic forms in characteristic 2 and their pencils, and the theory of curves of genus 2 
and Kummer surfaces in characteristic 2.  We study singularities of 
quadratic line complexes in Section \ref{sec3-0},
quartic surfaces associated with quadratic line complexes in Section \ref{sec4},
the intersections of three quadrics in ${\bf P}^5$ in Section \ref{sec3}
and the curves of genus 2 in Section \ref{sec5}.  Finally in Section \ref{sec6} we
discuss the canonical map $|4\Theta - 2\sum_{i=1}^4 p_i|$ from the Jacobian of an ordinary curve of genus 2 
to the intersection of three quadrics, where $p_1,\ldots, p_4$ are 2-torsion points on the Jacobian.  
This is also an analogue of the map $|4\Theta - \sum_{i=1}^{16} p_i|$ mentioned above.

\smallskip

Very recently, Dolgachev \cite{Do24} has found an analogue of Kummer surfaces in characteristic two.
This is another direction, that is, $K3$ surfaces are supersingular and contain a 
$(16)_6$-configuration, but a curve of genus 2 does not appear there.   
On the other hand, in our case, $K3$ surfaces are not supersingular and contain no $(16)_6$-configurations.

\subsection*{Acknowlegement}
The authors are grateful to the anonymous referee for valuable comments.

\section{Preliminaries}\label{sec2}

In this section, we recall the classical theory of quadratic line complexes and Kummer surfaces in \ref{ClassicalQLC}.  Next we recall 
the theory of quadratic forms in characteristic 2 in \ref{QuadForm} and introduce a canonical forms of
pencils of quadrics in ${\bf P}^5$ in \ref{Pencils}.  Finally 
in \ref{Char2}, we recall the theory of curves of genus 2 and Kummer surfaces in characteristic 2.

\subsection{Classical theory of Quadratic line complexes and Kummer surfaces}\label{ClassicalQLC}

Let $k$ be an algebraically closed field of characteristic char$(k) = p \geq 0$ (when we assume $p\not= 2$,
we will mention it).
In the following we recall the classical theory of quadratic line complexes over the base field $k$.  
The main references are Griffiths and Harris \cite[Chapter 6]{GH}, Cassels and Flynn \cite[Chapter 17]{CF}.

Let $G=G(2,4)$ be the Grassmannian variety of lines in ${\bf P}^3$ which can be embedded into 
${\bf P}^5$ as a non-singular quadric hypersurface (Pl${\rm \ddot{u}}$cker embedding).
For a point 
$\mathfrak{p}$ and a hyperplane $\mathfrak{h}$ in ${\bf P}^3$, let 
$$
\sigma(\mathfrak{p}) = \{ \mathfrak{l} \in G : \ \mathfrak{p} \in \mathfrak{l}\}, \ \sigma(\mathfrak{h}) = \{ \mathfrak{l} \in G :\ \mathfrak{l} \subset \mathfrak{h}\},\ \sigma(\mathfrak{p},\mathfrak{h}) = \{ \mathfrak{l} \in G : \  \mathfrak{p} \in \mathfrak{l} \subset \mathfrak{h} \}
$$
be the Schubert varieties.  Both $\sigma(\mathfrak{p})$ and $\sigma(\mathfrak{h})$ are planes and 
$\sigma(\mathfrak{p},\mathfrak{h})$ is a line in ${\bf P}^5$.  
Conversely any plane in $G$ is of the form of either 
$\sigma(\mathfrak{p})$ or $\sigma(\mathfrak{h})$ for some
$\mathfrak{p}, \mathfrak{h}$ and any line in $G$ is of the form $\sigma(\mathfrak{p},\mathfrak{h})$.   
Any non-singular quadric in ${\bf P}^5$ contains two 3-dimensional irreducible families of planes 
and in case of $G$ they are nothing but $\{\sigma(\mathfrak{p})\}_{\mathfrak{p} \in {\bf P}^3}$ and 
$\{\sigma(\mathfrak{h})\}_{\mathfrak{h}\subset {\bf P}^3}$ 
(see Proposition \ref{maximalSubspaces} for $p= 2$).  

Let $Q$ be a quadric hypersurface in ${\bf P}^5$ such that $Q\cap G$ is non-singular.  
The variety $\calX=Q\cap G$ was classically called {\it a quadratic line complex} which parameterizes 
lines in ${\bf P}^3$.  
The condition $\calX$ being non-singular implies that
$\calX$ does not contain any plane  (see Lemma \ref{two-quadrics}) and hence 
the intersection $\sigma(\mathfrak{p}) \cap Q$ is a conic in the plane $\sigma(\mathfrak{p})$.
Define
\begin{equation}\label{KummerQuartic1}
S = \{ \mathfrak{p} \in {\bf P}^3 : \ \sigma(\mathfrak{p}) \cap Q \ {\rm  is\ a\ singular\ conic}\},
\end{equation}
\begin{equation}\label{KummerQuartic2}
R=\{\mathfrak{p} \in S : \ \sigma(\mathfrak{p}) \cap Q \ {\rm  is\ a\ double\ line}\}.
\end{equation}  
Similarly we define the dual version:
\begin{equation}\label{DualKummer}
S^\vee = \{ \mathfrak{h} \in ({\bf P}^3)^\vee : \ \sigma(\mathfrak{h}) \cap Q \ {\rm  is\ a\ singular\ conic}\},
\end{equation}
\begin{equation}\label{DualKummer2}
R^\vee = \{ \mathfrak{h} \in S^\vee : \ \sigma(\mathfrak{h}) \cap Q \ {\rm  is\ a\ double\ line}\}.
\end{equation}
For $x \in \calX$, we denote by $\mathfrak{l}_x$ the line in ${\bf P}^3$ corresponding to $x$.
A line $\mathfrak{l}_x$ is called {\it singular} if $x$ is a singular point of the conic 
$\sigma(\mathfrak{p})\cap Q$
(resp. $\sigma(\mathfrak{h})\cap Q$) for some $\mathfrak{p}$ (resp. $\mathfrak{h}$).  
The point $\mathfrak{p}$ (resp. the plane $\mathfrak{h}$) is uniquely 
determined by $x$ and is called the {\it focus} (resp. the {\it plane}) of $\mathfrak{l}_x$.
Let 
\begin{equation}\label{Kummer}
\Sigma =\{ x \in \calX :\ \mathfrak{l}_x\ {\rm  is\ a\ singular\ line }\}.
\end{equation}

\begin{prop}\label{singularlines}{\rm (Griffiths and Harris \cite[p.767]{GH}, Cassels and Flynn \cite[Lemma 17.2.1]{CF})}
Let $x \in \calX$.  
Then $x \in \Sigma$ if and only if the tangent space $T_x(Q) \subset {\bf P}^5$ of $Q$ at $x$ is tangent to
$G$ at some point.
\end{prop}

There exist canonical morphisms 
$$\pi: \Sigma \to S, \quad \pi^\vee: \Sigma \to S^\vee$$
by sending $x$ to the focus or the plane of $\mathfrak{l}_x$.  
Note that when $\sigma(\mathfrak{p})\cap Q$ is a double line,
it sends to the point $\mathfrak{p}$.

As mentioned above, a non-singular quadric in 
${\bf P}^5$ contains two irreducible families of planes and a singular quadric of rank 5
contains an irreducible family of planes for $p\ne 2$. 
The pencil of quadrics $\{t_0G+t_1Q\}_{(t_0:t_1)\in {\bf P}^1}$
defines a curve $C$ as follows.  Let $G(3,6)$ be the Grassmannian of planes in ${\bf P}^5$ and we consider
\[
Z=\{ ((t_0:t_1), P)\in {\bf P}^1 \times G(3,6) \ : \ P \subset t_0G+t_1Q\}.
\]
A general fiber of the first projection morphism $pr : Z \to {\bf P}^1$ has two connected components. 
Then, we define 
\begin{equation}\label{curveOFgenus2}
C = {\rm Spec}(pr_*\calO_Z)
\end{equation}
which is a double covering of ${\bf P}^1$,  
parameterizing irreducible components of families of planes contained in 
$t_0G+t_1Q$ $((t_0:t_1)\in {\bf P}^1)$.  For $p= 2$, see Proposition \ref{maximalSubspaces}.
\begin{remark}\label{KummerChar0}
In case $p\ne 2$, by the assumption $\calX=G\cap Q$ being non-singular,  
we can diagonalize $G$ and $Q$ simultaneously as in the equation (\ref{diagonal}) (cf. Klingenberg \cite[Satz 1]{Kli}).
Moreover it is known that $S$ is a quartic surface with sixteen nodes and is called a {\it Kummer quartic surface} and $S^\vee$ is isomorphic to $S$.  The surface $\Sigma$ is non-singular and hence is a $K3$ surface.
The both morphisms $\pi$ and $\pi^\vee$ are the minimal resolutions of singularities. In this case
the curve $C$ is given by the equation (\ref{genus2}).
\end{remark}

Let $A$ be the {\it variety of lines} in  $\calX$.  
For each $L\in A$, there exist a point $\mathfrak{p}_L$ and a plane $\mathfrak{h}_L$ with 
$L=\sigma(\mathfrak{p}_L,\mathfrak{h}_L)$.
Thus we have morphisms 
$$\varphi: A \to S, \quad \varphi^\vee: A \to S^\vee$$
defined by
$\varphi(L) = \mathfrak{p}_L, \ \varphi^\vee(L) = \mathfrak{h}_L$
when $L = \sigma(\mathfrak{p}_L, \mathfrak{h}_L).$  
Note that the conic $\sigma(\mathfrak{p}_L) \cap Q$ splits into two lines, that is, one is 
$\sigma(\mathfrak{p}_L,\mathfrak{h}_L)$ and
another is $\sigma(\mathfrak{p}_L, \mathfrak{h}')$ for some plane $\mathfrak{h}'$ 
containing $\mathfrak{p}_L$.  Similarly $\sigma(\mathfrak{h}_L)\cap Q 
= \sigma(\mathfrak{p}_L,\mathfrak{h}_L) + \sigma(\mathfrak{p}', \mathfrak{h}_L)$ 
for some point $\mathfrak{p}' \in \mathfrak{h}_L$.
Thus both morphisms $\varphi, \varphi^\vee$ have degree 2 branched at each point of $R$ and $R^\vee$.   

We will show that $A$ is isomorphic to the Jacobian $J(C)$ of $C$ and the maps $\varphi, \varphi^\vee$ are
the quotient map by inversions of $A$.  The following argument is given by 
Cassels and Flynn \cite[Chap. 17, \S 1]{CF}.
We denote by $\mathfrak{a}=(t_0,t_1, \mu)$ a point of $C$ over a point $(t_0:t_1)\in {\bf P}^1$, and
by $\bar{\mathfrak{a}}$ another point of $C$ over $(t_0:t_1)\in {\bf P}^1$, where $\mu$ denotes an irreducible family of planes in the quadric $t_0G+t_1Q$.
For $\mathfrak{a}\in C$ and $L\in A$, we define a line $\Upsilon(\mathfrak{a})L\in A$ as follows.
There exists a unique plane $\Pi$ in the family of planes in $t_0G+t_1Q$ corresponding to $\mu$ which
contains $L$ (Lemma \ref{two-quadrics}).  
Then the conic $\Pi\cap \calX$ in $\Pi$ splits into two lines $L + M$. We now 
define $\Upsilon(\mathfrak{a})L=M$.
Thus, if we denote by $A_L$ the set of lines $\{\Upsilon(\mathfrak{a})L\}_{\mathfrak{a}\in C}$, then
$A_L$ is the set of lines meeting with $L$ and is birational to $C$.

Now, for an effective divisor $\mathfrak{a}_1+\mathfrak{a}_2$ on $C$ and $L\in A$, define
$$L'= \Upsilon(\bar{\mathfrak{a}}_1)\Upsilon(\mathfrak{a}_2)L.$$
We remark that if $\mathfrak{a}_1+\mathfrak{a}_2$ is general, then $\Upsilon(\mathfrak{a}_1)L \cap \Upsilon(\mathfrak{a}_2)L=\emptyset$, $L'\cap L=\emptyset$ and 
$L'$ is the unique line meeting with $\Upsilon(\mathfrak{a}_1)L$ and
$\Upsilon(\mathfrak{a}_2)L$ (Cassels and Flynn \cite[Lemma 17.1.5]{CF}).  Moreover, if we denote by  
$\Lambda$ the 3-dimensional  linear space spanned by $L$, $\Upsilon(\mathfrak{a}_1)L$, 
$\Upsilon(\mathfrak{a}_2)L$, then
\begin{equation}\label{3-space}
\Lambda \cap \calX = L + \Upsilon(\mathfrak{a}_1)L+\Upsilon(\mathfrak{a}_2)L+L'.
\end{equation}
Finally note that for $\bar{\mathfrak{a}}+\mathfrak{a}$ we have 
$\Upsilon(\mathfrak{a})\Upsilon(\mathfrak{a})L = L$.
Thus $A$ has a natural structure of principal homogeneous space over $J(C)$.

\begin{prop}\label{Basic}
Let $k$ be an algebraically closed field in characteristic $p \geq 0$.  Assume that $C$ is a non-singular curve.
Then $A$ is isomorphic to the Jacobian $J(C)$ of $C$.  
\end{prop}

We remark that, by using the theory of vector bundles, Bhosle
studied pencils of quadrics in ${\bf P}^{2g+1}$ in characteristic 2.
She presented a canonical form of a pencil of quadrics in ${\bf P}^{2g+1}$, associated
a hyperelliptic curve $C$ of genus $g$ to a pencil for generic case and showed that 
the Jacobian of $C$ is isomorphic to the variety of $(g-1)$-dimensional subspaces on the
base locus of the pencil (Bhosle \cite[Theorem 4]{B}).

Let $\iota$ (resp. $\iota^\vee$) be the covering transformation of the morphism $\varphi: A \to S$ (resp. $\varphi^\vee : A \to S^\vee$).
Then $\iota$ (resp. $\iota^\vee$) has fixed points over $R$ (resp. $R^\vee$).  
The following Proposition is well known over the complex numbers (Griffiths and Harris \cite[p.780]{GH}).
 
\begin{prop}\label{Basic2}
Let $k$ be an algebraically closed field in characteristic $p \geq 0$.   Assume that $C$ is a non-singular curve.
The morphisms $\varphi : A \to S$ and $\varphi^\vee: A \to S^\vee$ are the quotient maps 
by the inversion.  In particular, $S$ and $S^\vee$ are isomorphic.
\end{prop}
\begin{proof}  
Let $\iota$ be the covering transformation of $\varphi$.  To prove that $\iota$ is an inversion of $A$, 
it is enough to show that $\iota$ has an isolated fixed point and no fixed curves 
(Katsura \cite[Lemma 3.5]{Ka2}).  Assume that $\iota$ has a fixed curve $D$.  Then, by definition of the map 
$\pi: \Sigma \to S$, $\pi^{-1}(\varphi(p)) \subset \Sigma$ is the line corresponding to $p$ for any 
$p\in D$. Thus $\pi^{-1}(\varphi(D))$ has dimension 2.  This implies that $\Sigma$ is not irreducible 
and  has singularities along a curve.  On the other hand, $\Sigma$ is non-singular in charcteristic 
$p\ne 2$.  In case $p=2$, we will show that $\Sigma$ is singular, but only isolated singularities 
(Theorems \ref{Int3quadricsOrdinary}, \ref{Int3quadricsP-rank1}, \ref{Int3quadricsSupersingular}).  
Thus $\iota$ has at most isolated fixed points.  
Assume now that $\iota$ has no fixed points.  Then $S$ is non-singular.  On the other hand, 
in characteristic $p\ne 2$, $S$ has 16 nodes.  In case $p=2$, we will show that $S$ has always 
singulatiries (Theorems \ref{KummerQuarticOrdinary}, \ref{KumQuarticRank1}, \ref{ssKummerQ}).
\end{proof}

It is known that $S^\vee$ is the projective dual of $S$ (cf. Cassels and Flynn \cite[p.181]{CF}).

\begin{remark}
Let $\Theta$ be the theta divisor of $A$.  Then, over the complex numbers, it is well known that 
the morphism $\varphi : A\to S$ is defined by the complete linear system $|2\Theta|$, and  the linear system
$|4\Theta - \sum_{i=1}^{16} p_i|$ also defines a rational map
$$\psi : A \to \Sigma$$
of degree 2, 
where $\{p_i\}$ is the set of 2-torsion points on $A$ (Griffiths and Harris \cite[p.786]{GH}).
In case $char (k)=2$, Laszlo and Pauly \cite[Proposition 4.1]{LP} 
studied the map defined by $|2\Theta|$ and found the equation of the Kummer quartic surface for
ordinary case, Ducrohet \cite{Duc} for supersingular case and Duquesne \cite{Du} for arbitrary case.
\end{remark}

\subsection{Quadratic forms in characteristic two}\label{QuadForm}

We recall fundamental facts on quadratic forms in characteristic two in 
Dieudonn\'e \cite{Die} and Bhosle \cite{B}.

Let $k$ be a field of any characteristic and 
let $E$ be a vector space over $k$.
A {\it quadratic form} $q$ on $E$ is a function $q: E \to k$ satisfying 
\begin{equation}
q(ax + by) = a^2q(x) + b^2q(y) + abA(x,y), \quad a, b \in k,
\end{equation}
where $A$ is a symmetric bilinear form on $E$.
Note that if $char(k)=2$, $A(x,x)=0$ for all $x\in E$ and hence $A$ is {\it alternating}.

In the following, we assume that $k$ is an algebraically closed field of characteristic two.
A subspace $V$ of $E$ is called {\it totally singular} if $q(x)=0$ for all $x\in V$.
A totally singular subspace is {\it totally isotropic}, that is, $A(x,y)=0$ for all $x, y\in V$.
The converse is not true.  The {\it index} $\nu$ of $q$ is defined as the dimension of a maximal
totally singular subspace.  Two totally singular subspaces of the same dimension are
equivalent under the action of the orthogonal group of $E$.

A quadratic form is called {\it non-defective} if the alternating form $A$ is non-degenerate and
{\it defective} if $A$ is degenerate.  
For a defective quadratic form $q$ we define the {\it radical} $N$ of $A$ by
$N=\{ x \in E : \ A(x,y)=0 \ {\rm for \ all} \ y \in E\}$.
The dimension of $N$ is called the {\it defect} of $q$.

\begin{prop}\label{quadraticForm}{\rm (Bhosle \cite[Lemma 2.5]{B})}
Let $q$ be a quadratic form on $k^{2m}$ with the associated alternating form $A$.

{\rm (1)} \ Assume that $q$ is non-defective.  
Let $W$ be a maximal totally singular subspace of $k^{2m}$ which is of dimension $m$ 
and let $e_1,\ldots, e_{m}$ be a basis of $W$.
Then there exists a basis $e_1,\ldots, e_m, f_1,\ldots, f_m$ of $k^{2m}$ such that
$f_1,\ldots, f_m$ span a totally singular subspace for $q$ and $A(e_i, f_j)=\delta_{ij}$
for $i,j=1,\ldots, m$.

{\rm (2)} Assume that the defect of $q$ is two and the radical $N$ contains a unique totally 
singular subspace $N_0$ of dimension one.  Let $W$ be a maximal totally singular subspace 
of $k^{2m}$.  Then there exists a basis $e_1,\ldots, e_m, f_1,\ldots, f_m$ of $k^{2m}$ 
such that $e_1,\ldots, e_m$ is a basis of $W$ and the span of $f_2,\ldots, f_m$ is totally
singular, $A(e_i, f_j)=\delta_{ij}$ for $i,j=2,\ldots, m$, $e_1$ spans $N_0$ and $e_1, f_1$
is a basis of $N$.
\end{prop}

\begin{prop}\label{maximalSubspaces}{\rm (Bhosle \cite[Lemma 2.6]{B})}
Let $q$ be a quadratic form on $k^{2m}$.

{\rm (1)} \ Assume that $q$ is non-defective.  
Then the space of maximal totally singular subspaces for $q$ has two connected components
each of which is a non-singular variety of dimension $m(m-1)/2$.

{\rm (2)} Assume that the defect of $q$ is two and the radical $N$ contains a unique totally 
singular subspace $N_0$ of dimension one.  Then the space of maximal totally singular subspaces for $q$ 
is a non-singular variety of dimension $m(m-1)/2$.
\end{prop}

\subsection{Pencils of quadratics in ${\bf P}^5$}\label{Pencils}

In this subsection, we introduce canonical forms of pencils of quadric hypersurfaces of ${\bf P}^5$ in characteristic two based on
the results due to Klingenberg \cite{Kli} and Bhosle \cite{B} (Proposition \ref{CanFormPencilQuadrics}). 

Let $q_1, q_2$ be quadratic forms on $k^{2m}$ and consider the pencil $\{xq_1 + yq_2\}_{(x:y)\in {\bf P}^1}$ of quadratic forms generated by $q_1, q_2$. 
We have a pencil of alternating forms $\{xA_1 + yA_2\}$ associated with $\{xq_1 + yq_2\}$.
By choosing a basis of $k^{2m}$, we denote the corresponding alternating matrix
by the same symbol $xA_1+yA_2$.
The Pfaffian of the pencil is defined as the square root of ${\rm det}(xA_1 + yA_2)$.

\begin{prop}\label{CanAlt}{\rm (Klingenberg \cite[Satz 2]{Kli}, Bhosle \cite[Proposition 2.8]{B})}
Let $A_1, A_2$ be alternating forms with $A_1$ non-degenerate.  
Let $\prod_i (t-a_i)^{2m_i}$ be the characteristic polynomial of $A_1^{-1}A_2$ where $a_i\ne a_j$ $(i\ne j)$.
Then the pencil 
can be written in the following form$:$
$$A_1= \sum_{i=1}^{m}\sum_{j=1}^{m_i} X_{ij}Y_{ij},\quad 
A_2= \sum_i\left\lbrack a_i\left(\sum_{j=1}^{m_i} X_{ij}Y_{ij}\right) + \sum_{j=2}^{m_i} X_{ij}Y_{i(j-1)}\right\rbrack,$$
where $X_{ij}Y_{ij}$ denote the alternating form
$$A((x_{ij},y_{ij}), (x'_{ij},y'_{ij})) = x_{ij}y'_{ij} - x'_{ij}y_{ij}.$$
\end{prop}
In \cite[Satz 2]{Kli}, the author assumed the charcteristic $p\ne 2$, however as Bhosle pointed out 
\cite[Proposition 2.8]{B}, the proof of Satz 2 works well in characteristic 2.

Let $(X_1,X_2,X_3,Y_1,Y_2,Y_3)$ be homogeneous coordinates of ${\bf P}^5$.
We consider a pencil of quadratic forms generated by $q_1, q_2$ with the associated alternating
forms $A_1, A_2$.  We assume that $q_1$ is non-defective and hence ${\rm det}(A_1)\not=0$.
We use the symbol $Q_i$ for the hypersurface in ${\bf P}^5$ defined by $q_i$ $(i=1, 2)$.
Let $\calP$ be the pencil of quadrics in ${\bf P}^5$ generated by $Q_1, Q_2$.
First we recall the following fact (e.g. Griffiths and Harris \cite[p.762, Lemma]{GH}).

\begin{lemma}\label{two-quadrics}
Assume that $Q_1\cap Q_2$ is non-singular.  Then $Q_1\cap Q_2$ does not contain a plane.
\end{lemma}
\begin{proof}
Assume that $Q_1\cap Q_2$ contains a plane $\Pi$.  By changing of coordinates (Proposition \ref{quadraticForm} (1)), we may assume that
$\Pi$ is defined by $X_1=X_2=X_3=0$, and 
$Q_1$, $Q_2$ are given by 
$$Q_1: A_1 + \sum_i c_iX_i^2 =0, \quad 
Q_2: \sum_{i,j} a_{ij}X_iY_j + \sum_{i,j}b_{ij}X_iX_j + \sum_i d_i X_i^2=0,$$
respectively. Here $A_1= \sum_{i=1}^3 X_iY_i$.  
By restricting the Jacobian of $Q_1\cap Q_2$ to $\Pi$, we obtain the three conics
$$Y_1\sum_{j=1}^3 a_{2j}Y_j+Y_2\sum_{j=1}^3 a_{1j}Y_j=0,\quad 
Y_2\sum_{j=1}^3 a_{3j}Y_j+Y_3\sum_{j=1}^3 a_{2j}Y_j=0,$$
$$Y_1\sum_{j=1}^3 a_{3j}Y_j+Y_3\sum_{j=1}^3 a_{1j}Y_j=0$$
which have a common zero.
\end{proof}

Let $\prod_{i=1}^3(t-a_i)^2$ be the characteristic polynomial of $A_1^{-1}A_2$.
The following three cases occur.

${\rm (a)} \ a_1, a_2, a_3$ are different;

${\rm (b)}\ a_1\not= a_2=a_3$;

${\rm (c)} \ a= a_1=a_2=a_3$.

\noindent
Now by using Proposition \ref{CanAlt}, 
$\calP$ is given as follows:
\begin{equation}\label{ordinaryPencil}
{\rm (a)} : \
\left\{ \begin{array}{l}
Q_1 : A_1 + \sum_{i=1}^3 (p_i^2X_i^2 + t_i^2Y_i^2)=0, \\
\\
Q_2: \sum_{i=1}^3 (a_iX_iY_i + r_i^2X_i^2 + s_i^2Y_i^2)=0.
\end{array} \right.
\end{equation}

\begin{equation}\label{p-rank1Pencil}
{\rm (b)} : \ 
\left\{ \begin{array}{l}
Q_1: A_1 + \sum_{i=1}^3 (p_i^2X_i^2 + t_i^2Y_i^2)=0,\\
\\ 
Q_2: \sum_{i=1}^3 (a_iX_iY_i + r_i^2X_i^2 + s_i^2Y_i^2) + X_3Y_2=0.
\end{array} \right.
\end{equation}

\begin{equation}\label{SupersingularPencil}
{\rm (c)} :\
\left\{ \begin{array}{l}
Q_1: A_1 + \sum_{i=1}^3 (p_i^2X_i^2 + t_i^2Y_i^2)=0,\\
\\ 
Q_2: \sum_{i=1}^3 (aX_iY_i + r_i^2X_i^2 + s_i^2Y_i^2) + X_2Y_1+X_3Y_2=0.
\end{array} \right.
\end{equation}

\smallskip
\noindent
As in the proof of Bhosle \cite[Corollary 2.10]{B}, by changing the coordinates
\begin{equation}\label{CC}
\left\{ \begin{array}{l}
x_i = \alpha_i X_i +\beta_i Y_i, \\
\\
y_i = \gamma_i X_i + \delta_i Y_i
\end{array} \right.
\end{equation}
with $\alpha_i\delta_i + \beta_i\gamma_i =1, \ \alpha_i\gamma_i =p_i^2, \beta_i\delta_i = t_i^2$, we may assume that $\calP$ is given by the equations in Proposition \ref{CanFormPencilQuadrics}.

\begin{prop}\label{CanFormPencilQuadrics}
\begin{equation}\label{(a)}
{\rm (a)} : \
\left\{ \begin{array}{l}
G : A_1=\sum_{i=1}^3 X_iY_i=0, \\
\\
Q : \sum_{i=1}^3 a_iX_iY_i + c_iX_i^2 + d_iY_i^2=0.
\end{array} \right.
\end{equation}
\begin{equation}\label{(b)}
{\rm (b)} : \ 
\left\{ \begin{array}{l}
G : \ A_1=0, \\
\\
Q : \ \sum_{i=1}^3 \left(a_iX_iY_i + c_iX_i^2 + d_iY_i^2\right) + B =0
\end{array} \right.
\end{equation}
where 
$
B = b_1X_2Y_3 + b_2X_3Y_2 + b_3X_2X_3 +
b_4Y_2Y_3
$
and
$b_1=\beta_3\gamma_2, \ b_2=\alpha_2\delta_3,\ b_3=\gamma_2\delta_3,\ b_4=\alpha_2\beta_3.$
\begin{equation}\label{(c)}
{\rm (c)} : \ 
\left\{ \begin{array}{l}
G : \ A_1=0, \\
\\
Q : \ \sum_{i=1}^3 \left(aX_iY_i + c_iX_i^2 + d_iY_i^2\right) + C = 0
\end{array} \right.
\end{equation}
where 
$C= b_1X_2Y_3 + b_2X_3Y_2 + b_3X_2X_3 +b_4Y_2Y_3 + b_5X_1Y_2+b_6X_2Y_1+b_7X_1X_2+
b_8Y_1Y_2,$ \ 
$b_1=\beta_3\gamma_2, \ b_2=\alpha_2\delta_3,\ b_3=\gamma_2\delta_3,\ b_4=\alpha_2\beta_3,\ 
b_5=\beta_2\gamma_1, \ b_6=\alpha_1\delta_2,\ b_7=\gamma_1\delta_2,\ b_8=\alpha_1\beta_2.$
\end{prop}

Let $\calP$, $\calP_1$, $\calP_0$ be the pencils of quadrics defined by 
(\ref{(a)}), (\ref{(b)}), (\ref{(c)}) in Proposition \ref{CanFormPencilQuadrics}, respectively, and call them
{\it canonical forms of pencil of quadrics}.  The subscripts of the second and third one correspond to the 2-ranks of the Jacobians.  Note that
$sG+tQ$ is non-defective for $(s,t)\ne (a_i,1)$ and $a_iG+Q$ has the defect 2 and the radical contains
a unique totally singular subspace of dimension one as stated in Proposition \ref{maximalSubspaces} (2).

\begin{remark}\label{coefficients2}
Note that in cases (b), (c), $b_1b_2=b_3b_4$, $b_5b_6=b_7b_8$.  Since $\alpha_i\delta_i + \beta_i\gamma_i =1$, 
$(b_1,b_2,b_3,b_4)\ne (0,0,0,0)$, and $(b_5,b_6,b_7,b_8)\ne (0,0,0,0)$.  Also
by the relations $b_1b_2=b_3b_4$ and $b_5b_6=b_7b_8$, we have
a relation
\begin{equation}\label{relation-9}
(b_1b_5+b_4b_7)(b_2b_6+b_3b_8)=(b_1b_8+b_4b_6)(b_2b_7+b_3b_5).
\end{equation}
Similarly we have
\begin{equation}\label{relation-91}
\begin{array}{l}
b_2(b_1b_8+b_4b_6)=b_4(b_2b_6+b_3b_8),\quad b_3(b_1b_8+b_4b_6)=b_1(b_2b_6+b_3b_8),\\
b_1(b_2b_7+b_3b_5)=b_3(b_1b_5+b_4b_7),\quad b_4(b_2b_7+b_3b_5)=b_2(b_1b_5+b_4b_7),\\
b_5(b_1b_8+b_4b_6)=b_8(b_1b_5+b_4b_7),\quad b_7(b_1b_8+b_4b_6)=b_6(b_1b_5+b_4b_7),\\
b_6(b_2b_7+b_3b_5)=b_7(b_2b_6+b_3b_8),\quad b_8(b_2b_7+b_3b_5)=b_5(b_2b_6+b_3b_8).
\end{array}
\end{equation}
\end{remark}

We use the following Lemma \ref{prank1translation} and Remark \ref{prank1translationRemark} to construct
some involutions of Kummer quartic surfaces (Propositions \ref{prank1translation2}, \ref{prank1translation3} and Lemma \ref{ssAut}).

\begin{lemma}\label{prank1translation}\ Assume that $b_1b_2b_3b_4\ne 0$ and $b_5b_6b_7b_8\ne 0$.

{\rm (1)}\ In the equation $(\ref{(b)})$, 
after changing the coordinates, we may assume that $b_2b_4c_2=b_1b_3d_2$.

\smallskip

{\rm (2)}\ In the equation $(\ref{(c)})$, 
after changing the coordinates, we may assume that $b_6b_8c_1=b_5b_7d_1$.
\end{lemma}
\begin{proof} (1)\ 
Recall that we start the pencil of quadrics given in (\ref{p-rank1Pencil}),
and by changing of coordinates (\ref{CC}), we have 
the equation (\ref{(b)}).
Here 
$b_1=\beta_3\gamma_2, b_2=\alpha_2\delta_3, b_3=\gamma_2\delta_3, b_4=\alpha_2\beta_3$
and
$c_i = a_i\gamma_i\delta_i + r_i^2\delta_i^2+s_i^2\gamma_i^2,\ d_i = a_i\alpha_i\beta_i +r_i^2\beta_i^2 + s_i^2\alpha_i^2.$
Since
$$b_2b_4c_2= \beta_3\delta_3(a_2\alpha_2^2\gamma_2\delta_2+r_2^2\alpha_2^2\delta_2^2 +s_2^2\alpha_2^2\gamma_2^2), \ b_1b_3d_2= \beta_3\delta_3(a_2\alpha_2\beta_2\gamma_2^2+r_2^2\beta_2^2\gamma_2^2 +s_2^2\alpha_2^2\gamma_2^2),$$
and $\alpha_2\delta_2+\beta_2\gamma_2=1$, 
the condition $c_2b_2b_4=d_2b_1b_3$ is equivalent to $a_2\alpha_2\gamma_2 + r_2^2=0$, that is,
$a_2p_2^2 + r_2^2=0$.
Now first by applying the changing of coordinates
\begin{equation}\label{CC3}
\left\{ \begin{array}{l}
Y_2 = \epsilon X_2' +Y_2' \quad (\epsilon \in k) \\
Y_1=Y_1',\ Y_3=Y_3',\ X_i=X_i'\quad (i=1,2,3)
\end{array} \right.
\end{equation}
to the equation (\ref{p-rank1Pencil}) and then by applying the changing of coordinates (\ref{CC}), 
the condition $a_2p_2^2 + r_2^2=0$ is replaced by $a_2p_2^2+r_2^2 +\varepsilon (a_2t_2^2 + s_2^2)+
\varepsilon^2(a_2t_2^2 + s_2^2)=0$.
Thus if $a_2t_2^2 + s_2^2\not= 0$, then we may assume $a_2p_2^2 + r_2^2=0$ and hence $c_2b_2b_4=d_2b_1b_3$.  
If $a_2t_2^2 + s_2^2=0$, then first apply the changing of coordinates
\begin{equation}\label{CC2}
\left\{ \begin{array}{l}
X_2 = X_2' +\varepsilon Y_2' \quad (\varepsilon \in k), \\
X_1=X_1',\ X_3=X_3',\ Y_i=Y_i' \quad (i=1,2,3)
\end{array} \right.
\end{equation}
to the equation (\ref{p-rank1Pencil}).
Then the condition $a_2t_2^2 + s_2^2=0$ is replaced by $a_2t_2^2 + s_2^2 +\epsilon^2(a_2p_2^2 + r_2^2)=0$, 
and hence we may assume $a_2t_2^2 + s_2^2\not= 0$.  Thus we have proved the assertion (1).

(2)\  The proof is the same as in the case (1).
\end{proof}

\begin{remark}\label{prank1translationRemark}
In case $b_1b_2b_3b_4=0$, Lemma \ref{prank1translation} holds after a modification.  
For example, if $b_1=b_4=0$ and $b_2b_3\ne 0$, then 
instead of $b_2b_4c_2=b_1b_3d_2$, we may assume $b_2^2c_2=b_3^2d_2$.

Also, by the same way, we may assume $b_1b_4c_3=b_2b_3d_3$.
We remark that the changing of coordinates (\ref{CC2}) does not change $\{ b_i\}$, but (\ref{CC3}) may change
$\{ b_i\}$. Therefore we can not assume both $b_2b_4c_2=b_1b_3d_2$ and $b_1b_4c_3=b_2b_3d_3$ at the same time.
\end{remark}

\subsection{Kummer quartic surfaces and curves of genus 2 in characteristic two}\label{Char2}

In this subsection, 
we recall fundamental results of singularities of Kummer quartic surfaces and curves of genus 2 
in characteristic two. 

Let $A$ be an abelian surface over $k$ and 
let $\iota$ be the inversion.
Then, for the singularities of $A/\langle \iota \rangle$, we have the
following theorem (cf. Katsura \cite{Ka}).

\def\MARU#1{\textcircled{\scriptsize#1}}

\begin{theorem}\label{Kummer-sigularity}
\begin{itemize}
\item[$({\rm i})$] $4$ rational double points of type $D_4$ if $A$ is ordinary.
\item[$({\rm ii})$] $2$ rational double points of type $D_8$ if the $2$-rank of $A$ is one.
\item[$({\rm iii})$] An elliptic double point of type $\mbox{\MARU{19}}_0$ if $A$ is superspecial.
\item[$({\rm iv})$] An elliptic double point of type $\mbox{\MARU{4}}_{0,1}^{1}$ if $A$ is supersingular and not superspecial.
\end{itemize}
\end{theorem}
Here a supersingular abelian surface is called {\it superspecial} if it is isomorphic to the product of two 
elliptic curves, and elliptic double points of type $\mbox{\MARU{19}}_0$, $\mbox{\MARU{4}}_{0,1}^{1}$
are in the sense of Wagreich \cite{W}.
The dual graphs of the minimal resolutions of these singularities are given in Figure \ref{ellipticSing}.
Each component is a non-singular rational curve, 
$-3$ or $-4$ is the self-intersection number of the component and other components are $(-2)$-curves. 
The central component has multiplicity 2.
In Figure \ref{ellipticSing}, an elliptic singularity of type $\mbox{\MARU{14}}_{0,0,1}^{(0),(1)}$ is also 
in the sense of Wagreich whose minimal resolution is obtained from the one of a singularity of type 
$\mbox{\MARU{4}}_{0,1}^{1}$ by blowing up a point on
the $(-3)$-curve.  This singularity appears in Theorem \ref{Int3quadricsSupersingular}.

\begin{figure}[htbp]
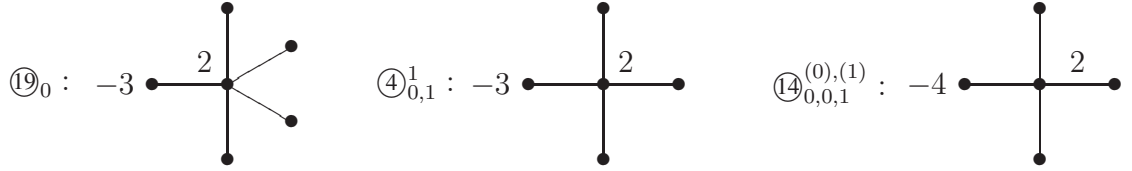

\begin{center}
\resizebox{!}{1.5cm}{
\xy
(0,5)*{};(10,5)*{}**\dir{-};
(10,5)*{};(10,15)*{}**\dir{-};
(10,5)*{};(10,-5)*{}**\dir{-};
(10,5)*{};(18.5,10)*{}**\dir{-};
(10,5)*{};(18.5,0)*{}**\dir{-};
(50,5)*{};(60,5)*{}**\dir{-};
(60,5)*{};(60,15)*{}**\dir{-};
(60,5)*{};(60,-5)*{}**\dir{-};
(60,5)*{};(70,5)*{}**\dir{-};
(108,5)*{};(118,5)*{}**\dir{-};
(118,5)*{};(118,15)*{}**\dir{-};
(118,5)*{};(118,-5)*{}**\dir{-};
(118,5)*{};(128,5)*{}**\dir{-};
@={(0,5),(10,5),(10,15),(18.5,10),
(18.5,0),(10,-5),(50,5),(60,5),(60,15),(60,-5),(70,5),
(108,5),(118,5),(118,15),(118,-5),(128,5)
}@@{*{\bullet}};
(6,2)*{-3};
(7,8)*{2};
(-15,5)*{\mbox{\MARU{19}}_{0}:};
(45,5)*{-3};
(63,8)*{2};
(35,5)*{\mbox{\MARU{4}}_{0,1}^{1}:};
(103,5)*{-4};
(121,8)*{2};
(90,5)*{\mbox{\MARU{14}}_{0,0,1}^{(0),(1)}:};
\endxy 
}
\caption{Dual graphs of minimal resolutions of singularities}
\label{ellipticSing}
\end{center}
\end{figure}

Note that in Cases (i) and (ii) the non-singular model of $A/\langle \iota \rangle$ is a $K3$ surface
and that in Cases (iii) and (iv) $A/\langle \iota \rangle$ is a rational surface 
(cf. Shioda \cite{Sh} and Katsura \cite{Ka}).

For a non-singular curve $C$ of genus $2$, we denote by $J(C)$ the Jacobian variety of $C$.
As for a {\it normal form} of $C$, by Igusa \cite{I} we have the following result.
\begin{equation}\label{g=2}
y^2 + y = 
\left\{ \, 
\begin{array}{ll}
 \alpha x + \beta x^{-1} + \gamma (x-1)^{-1} & \mbox{if}~J(C)~\mbox{is ordinary}, \\
 x^3 +\alpha x + \beta x^{-1} & \mbox{if}~J(C)~\mbox{is of}~2\mbox{-rank}~1,\\
 x^5 + \alpha x^3 & \mbox{if}~J(C)~\mbox{is supersingular
},
\end{array}
\right.
\end{equation}
where $\alpha, \beta, \gamma \in k$ and $\alpha, \beta, \gamma$ in the first case and
$\beta$ in the second case are different from $0$.  
Note that in characteristic 2 there exists no curves $C$ of genus 2 such that $J(C)$ is superspecial
(cf. Ibukiyama-Katsura-Oort \cite{IKO}), and hence only an elliptic singularity of type 
$\mbox{\MARU{4}}_{0,1}^{1}$ appears in the quotient surface $J(C)/\langle \iota\rangle$.

\section{The intersection of two quadrics}\label{sec3-0}

We use the same notation as in Subsection \ref{Pencils}.  
In this section we present a necessary and sufficient condition  for 
the intersection of two quadrics in ${\bf P}^5$ being non-singular.

Recall that  $\calP$, $\calP_1$ or $\calP_0$ is the pencil of quadrics in ${\bf P}^5$ given 
by (\ref{(a)}), (\ref{(b)}) or (\ref{(c)}) in Proposition \ref{CanFormPencilQuadrics}, respectively.    
Let $\calX, \calX_1$ or  $\calX_0$ be the base locus of $\calP$, $\calP_1$ or 
$\calP_0$, respectively.

\begin{lemma}\label{3-fold}\
{\rm (a)} $\calX$ is singular if and only if $\prod_i c_id_i= 0$.

{\rm (b)} $\calX_1$ is singular if and only if 
$c_1d_1=0$ or $b_2b_4c_2 + b_1b_4c_3 + b_1b_3d_2 + b_2b_3d_3= 0$.

{\rm (c)} $\calX_0$ is singular if and only if 
$c_1(b_1b_8+b_4b_6)(b_2b_6+b_3b_8)+c_3(b_1b_5+b_4b_7)(b_1b_8+b_4b_6)$
$+d_1(b_1b_5+b_4b_7)(b_2b_7+b_3b_5)+d_3(b_2b_6+b_3b_8)(b_2b_7+b_3b_5)=0.$

\end{lemma}
\begin{proof}
Let $J$ be the Jacobian matrix of $\calX, \calX_1$ or  $\calX_0$
with respect to the homogeneous coordinates $(X_1, Y_1, X_2, Y_2, X_3, Y_3)$.
We denote by $\Delta_{ij}$ the determinant of the matrix consisting of 
$i$-th and $j$-th columns.
Assume $\rank (J)\leq 1$.

(a) \ In this case, $J$ is given by
$$
\left(
\begin{array}{cccccc}
 Y_1 & X_1 & Y_2 & X_2 & Y_3 & X_3\\
a_1Y_1 & a_1X_1 & a_2Y_2 & a_2X_2 & a_3Y_3 & a_3X_3
\end{array}
\right).
$$
Note that 5 columns of $J$ vanish by the relation $\sum X_iY_i=0$.  Hence, if $X_i\ne 0$,
then $c_i=0$ by the equation $Q=0$.  
Conversely if $c_1=0$, then $(X_1,Y_1,X_2,Y_2,X_3,Y_3)=(1,0,0,0,0,0)$ is a singular point.

\smallskip
(b)\ In this case we use the equations of the pencil given in (\ref{p-rank1Pencil}) instead of 
(\ref{(b)}).  Moreover we use the pair $Q_1$ and $Q_2+a_2Q_1$.  Then $J$ is given by
$$
\left(
\begin{array}{cccccc}
 Y_1 & X_1 & Y_2 & X_2 & Y_3 & X_3\\
(a_1+a_2)Y_1 &  (a_1+a_2)X_1 & 0 & X_3 & Y_2 & 0
\end{array}
\right).
$$
Obviously $X_3=Y_2=0$ and then $J$ is given by
$$
\left(
\begin{array}{cccccc}
0 & 0 & 0 & X_2 & Y_3 & 0\\
(a_1+a_2)Y_1 &  (a_1+a_2)X_1 & 0 & 0 & 0 & 0
\end{array}
\right).
$$
Thus ${\rm rank}(J)\leq 1$ if and only if $X_2=Y_2=X_3=Y_3=0$ or $X_1=Y_1=X_3=Y_2=0$.

First assume that $X_2=Y_2=X_3=Y_3=0$.  Then, by changing the coordinates (\ref{CC}), the new coordinates
should satisfy 
$x_1y_1=0$ and $a_1x_1y_1+c_1x_1^2+d_1y_1^2=0$, that is, $c_1d_1=0$.

Next assume that $X_1=Y_1=X_3=Y_2=0$.  Then, by changing the coordinates (\ref{CC}), we have
$$\alpha_2\gamma_2X_2^2 +\beta_3\delta_3Y_3^2=0, \quad (c_2\alpha_2^2+d_2\gamma_2^2)X_2^2 + (c_3\beta_3^2+d_3\delta_3^2)Y_3^2=0.$$
Thus $\calX_1$ has a singularity if and only if 
$$\alpha_2\gamma_2(c_3\beta_3^2+d_3\delta_3^2) +\beta_3\delta_3(c_2\alpha_2^2+d_2\gamma_2^2)=0,$$
by using the relation between $\{b_i\}$ and $\{\alpha_i, \beta_i, \gamma_i, \delta_i\}$,
which is equivalent to
$$b_2b_4c_2 + b_1b_4c_3 + b_1b_3d_2 + b_2b_3d_3= 0.$$

\smallskip
(c)\ In this case we also use the equations of the pencil given in (\ref{SupersingularPencil}) instead of 
(\ref{(c)}).  Moreover we use the pair $Q_1$ and $Q_2+aQ_1$.  Then $J$ is given by
$$
\left(
\begin{array}{cccccc}
 Y_1 & X_1 & Y_2 & X_2 & Y_3 & X_3\\
0 &  X_2 & Y_1 & X_3 & Y_2 & 0
\end{array}
\right).
$$
Obviously $X_3=Y_1=0$ and hence $J$ is
$$
\left(
\begin{array}{cccccc}
0 & X_1 & Y_2 & X_2 & Y_3 & 0\\
0 &  X_2 & 0 & 0 & Y_2 & 0
\end{array}
\right).
$$
Thus ${\rm rank}(J)\leq 1$ if and only if $X_2=X_3=Y_1=Y_2=0$.
By changing the coordinates (\ref{CC}), we have
$$\alpha_1\gamma_1X_1^2 +\beta_3\delta_3Y_3^2=0, \quad (c_1\alpha_1^2+d_1\gamma_1^2)X_1^2 + (c_3\beta_3^2+d_3\delta_3^2)Y_3^2=0.$$
Thus $\calX_0$ has a singularity if and only if 
$$\alpha_1\gamma_1(c_3\beta_3^2+d_3\delta_3^2) +\beta_3\delta_3(c_1\alpha_1^2+d_1\gamma_1^2)=0.$$
By using the relation between $\{b_i\}$ and $\{\alpha_i, \beta_i, \gamma_i, \delta_i\}$,
we have obtained the assertion.
\end{proof}

\subsection*{Assumption:}
In the following of this paper we assume that $\calX$, $\calX_1$, $\calX_0$ are non-singular.

\begin{remark}\label{coefficients}
Under the assumption, 
the cases $b_1=b_2=0$ and $b_3=b_4=0$ do not occur in case (b) by the condition
$b_3(b_1d_2 + b_2d_3) + b_4(b_2c_2 + b_1c_3)\ne 0$.  
Only the case $b_1=b_3=0$, $b_1=b_4=0$,
$b_2=b_3=0$ or $b_2=b_4=0$ is possible.  In case (c), 
$$(b_1b_5+b_4b_7, b_2b_6+b_3b_8)\ne (0,0),\quad  (b_1b_8+b_4b_6, b_2b_7+b_3b_5) \ne (0,0).$$
\end{remark}

The following Lemma \ref{Invariant} will be used to study singularities of quartic surfaces 
and their blowing-ups in Theorems \ref{ssKummerQ}, \ref{Int3quadricsSupersingular}.
\begin{lemma}\label{Invariant}
Let $\xi_1=\sqrt{b_2b_4+b_5b_8}, \ \xi_2=\sqrt{b_1b_3+b_6b_7}$.  Then $(\xi_1, \xi_2)\ne (0,0)$.
\end{lemma}
\begin{proof}
Assume that $\xi_1^2= \alpha_2^2\beta_3\delta_3+\beta_2^2\alpha_1\gamma_1=0$ and
$\xi_2^2=\gamma_2^2\beta_3\delta_3+\delta_2^2\alpha_1\gamma_1=0$.  Then it follows that
$\alpha_2^2\delta_2^2\alpha_1\gamma_1\beta_3\delta_3=\gamma_2^2\beta_2^2\alpha_1\gamma_1\beta_3\delta_3$,
that is, $\alpha_1\gamma_1\beta_3\delta_3(\alpha_2\delta_2+\beta_2\gamma_2)^2 =0$.
Since $\alpha_2\delta_2+\beta_2\gamma_2=1$, we have $\alpha_1\gamma_1\beta_3\delta_3=0$.
Again by $\alpha_i\delta_i+\beta_i\gamma_i=1$, we can easily see that 
$\alpha_1=\beta_3=0$, $\alpha_1=\delta_3=0$, $\gamma_1=\beta_3=0$ or $\gamma_1=\delta_3=0$.
These imply that $b_1=b_4=b_6=b_8=0$, $b_2=b_3=b_6=b_8=0$, $b_1=b_4=b_5=b_7=0$ or 
$b_2=b_3=b_5=b_7=0$.  In any case, this contradicts to the condition (c) in Lemma \ref{3-fold}.
\end{proof}

\section{The quartic surfaces}\label{sec4}

In this section we discuss the quartic surfaces $S, S_1, S_0$ which are defined by the set of foci of 
the intersection of two quadrics $\calX$, $\calX_1, \calX_0$ (see the definition (\ref{KummerQuartic1})).
First of all, we give a proof of Theorem \ref{KummerQuartic}, and then study the singularities and tropes
of the surfaces.

\subsection*{Proof of Theorem \ref{KummerQuartic}} 

We consider the pencil $\calP_0$ of quadrics, but not assuming $a_1=a_2=a_3$.  
Recall that $G$ is the grassmannian $G(2,4)$.
Consider three points 
$$(X_1,X_2,X_3,Y_1,Y_2,Y_3)=(x, 0, 0, 0, y, z),\ (0, x, 0, y, 0, t),\ (0, 0, x, z, t, 0)$$ 
in ${\bf P}^5$ and the plane 
$$\Pi= \{(\alpha x, \beta x, \gamma x, \beta y+\gamma z, \alpha y + \gamma t , \alpha z + \beta t) \ :  
\ (\alpha,  \beta , \gamma)\in {\bf P}^2\}$$ 
generated by these points, where $(x, y, z, t)\in {\bf P}^3$.
The plane $\Pi$ is a generic member of an irreducible family of planes on $G$.
The conic $\Pi \cap Q_2$ on $\Pi$ is given by

\begin{equation}\label{ordinaryConic}
\begin{array}{l}
(c_1x^2+d_2y^2+d_3z^2+b_4yz+b_5xy)\alpha^2\\
\\
 + (c_2x^2+d_1y^2+d_3t^2+b_1xt+b_6xy)\beta^2\\
\\
+(c_3x^2+d_1z^2+d_2t^2+b_2xt+b_8zt)\gamma^2\\
\\
+((a_1+a_2)xy+b_1xz+b_4yt+b_7x^2+b_8y^2)\alpha\beta \\
\\
+((a_2+a_3)xt+b_6xz+b_8yt+b_3x^2+b_4t^2)\beta\gamma \\
\\
 + ((a_1+a_3)xz+b_2xy+b_4zt+b_5xt+b_8yz)\gamma\alpha =0.
\end{array}
\end{equation}
Then this conic has a singularity if and only if 
$$
\alpha = (a_2+a_3)xt+b_6xz+b_8yt+b_3x^2+b_4t^2,
$$
$$
\beta= (a_1+a_3)xz+b_2xy+b_4zt+b_5xt+b_8yz,
$$
$$
\gamma = (a_1+a_2)xy+b_1xz+b_4yt+b_7x^2+b_8y^2.
$$
By combining this and (\ref{ordinaryConic}), we obtain the sextic equation.  The terms of  degree
$\leq 1$ in $x$ of this sextic are identically 0, and hence by dividing by $x^2$ we have the following equation:
$$(b_3^2c_1+b_7^2c_3)x^4 + (b_2^2d_1 +b_8^2c_3)y^4 + (b_1^2d_1+b_6^2d_3)z^4 + 
(b_4^2c_1+b_5^2d_3)t^4$$
$$+b_5(b_1b_5+b_4b_7)xt^3 + b_7(b_2b_7+b_3b_5)x^3t + b_2(b_2b_6 + b_3b_8)xy^3 
+ b_8(b_2b_6+b_3b_8)y^3z$$
$$+b_3(b_2b_7+b_3b_5)x^3y + b_4(b_1b_5+b_4b_7)zt^3 + b_6(b_1b_8 + b_4b_6)yz^3 
+ b_1(b_1b_8+b_4b_6)z^3t$$ 
$$+(a_1+a_3)(b_3b_7x^3z +b_1b_6xz^3 + b_2b_8y^3t+b_4b_5yt^3)$$
$$+ (b_2^2c_2 + b_3^2d_2 + c_3(a_1+a_2)^2 + (a_1+a_2)b_2b_3)x^2y^2
+(b_5^2c_2+b_7^2d_2+ c_1(a_2+a_3)^2 + (a_2+a_3)b_5b_7)x^2t^2$$
$$+(b_1^2c_3+b_3^2d_3 + b_6^2c_1+b_7^2d_1+c_2(a_1+a_3)^2 + (a_1+a_3)(b_1b_3+b_6b_7))x^2z^2$$
$$+(b_6^2d_2+b_8^2c_2 + d_1(a_2+a_3)^2 + (a_2+a_3)b_6b_8)y^2z^2
+ (b_1^2d_2 + b_4^2c_2 + d_3(a_1+a_2)^2 + (a_1+a_2)b_1b_4)z^2t^2$$
$$+ (b_2^2d_3 + b_4^2c_3 + b_5^2d_1 + b_8^2c_1 + d_2(a_1+a_3)^2 + (a_1+a_3)(b_2b_4+b_5b_8))y^2t^2$$
$$+ (b_7(b_2b_6+b_3b_8)+ b_3(a_1+a_2)(a_1+a_3))x^2yz + (b_3(b_1b_5+b_4b_7)+b_7(a_1+a_3)(a_2+a_3))x^2zt$$
$$+ (b_8(b_2b_7+b_3b_5) + b_2(a_1+a_2)(a_1+a_3))xy^2t + (b_2(b_1b_8+b_4b_6)+b_8(a_1+a_3)(a_2+a_3))y^2zt$$
$$+(b_1(b_2b_6+b_3b_8)+b_6(a_1+a_3)(a_2+a_3))xyz^2 + (b_6(b_1b_5+b_4b_7)+b_1(a_1+a_2)(a_1+a_3))xz^2t$$
$$ + (b_4(b_2b_7+b_3b_5)+ b_5(a_1+a_3)(a_2+a_3))xyt^2 + (b_5(b_1b_8+b_4b_6)+ b_4(a_1+a_2)(a_1+a_3))yzt^2$$
$$+ (b_2b_7(a_2+a_3)+b_3b_5(a_1+a_2))x^2yt + (b_2b_6(a_1+a_2)+b_3b_8(a_2+a_3))xy^2z$$
$$+ (b_1b_5(a_2+a_3)+b_4b_7(a_1+a_2))xzt^2 + (b_1b_8(a_1+a_2)+b_4b_6(a_2+a_3))yz^2t$$
$$+ ((a_1+a_2)(a_2+a_3)(a_3+a_1)+(a_1+a_2)(b_5b_6+b_7b_8)+(a_2+a_3)(b_1b_2+b_3b_4))xyzt=0.$$
Now by putting $a_1=a_2=a_3$ we have the equation (c) of the quartic surface.  
The case (a) (resp. the case (b)) is obtained by putting $b_1=\cdots =b_8=0$ (resp. 
$a_2=a_3$ and $b_5=\cdots = b_8=0$).
\qed

\medskip

Next we study the quartic surfaces individually.  First we consider the case (a).

\begin{theorem}\label{KummerQuarticOrdinary}\
The quartic surface $S$ has exactly four rational double points 
$$P_1=(1, 0, 0, 0),\ P_1=(0, 1, 0, 0),\ P_3=(0, 0, 1, 0),\ P_4=(0, 0, 0, 1)$$ 
of type $D_4$
and contains four tropes:
\smallskip
$$\overline{\Theta}_1: \ x= (a_1+a_2)\sqrt{d_3}zt + (a_1+a_3)\sqrt{d_2}yt + (a_2+a_3)\sqrt{d_1}yz=0,$$
$$\overline{\Theta}_2:\ y= (a_1+a_2)\sqrt{d_3}zt + (a_1+a_3)\sqrt{c_2}xz + (a_2+a_3)\sqrt{c_1}xt =0,$$
$$\overline{\Theta}_3:\ z= (a_1+a_2)\sqrt{c_3}xy + (a_1+a_3)\sqrt{d_2}yt + (a_2+a_3)\sqrt{c_1}xt=0,$$
$$\overline{\Theta}_4:\ t= (a_1+a_2)\sqrt{c_3}xy + (a_1+a_3)\sqrt{c_2}xz + (a_2+a_3)\sqrt{d_1}yz=0.$$
The trope $\overline{\Theta}_i$ passes through three points $P_j\ (j\not=i)$ among the four points.
\end{theorem}
\begin{proof}
It follows from the Jacobian criterion that $S$ has exactly four singular points 
$P_1$$,\ldots,$ $P_4$.
By blowing up the singular points, we can see that $P_i$ is a rational double point
of type $D_4$ (or it follows from Remark \ref{CurveOrdinary2}, Corollary \ref{KummerQuarticCor}, 
Proposition \ref{Basic2} and Theorem \ref{Kummer-sigularity} that the singularities are of type $D_4$).
Each trope is defined by a hyperplane section, e.g.
$2\overline{\Theta}_1 = \{x=0\}$.
The last assertion is straightforward.
\end{proof}

The Cremona transformation
\begin{equation}\label{CremonaOrdinary}
c_r: (x, y, z, t)\to \left(\sqrt{d_1d_2d_3}/x, \sqrt{c_1c_2d_3}/y, \sqrt{c_1d_2c_3}/z, \sqrt{d_1c_2c_3}/t\right)
\end{equation}
preserves the equation of the quartic surface and
interchanges $P_i$ and $\overline{\Theta}_i$ $(1\leq i \leq 4)$.  
Also there are the following involutions of $S$ which generate $({\bf Z}/2{\bf Z})^2$:
$$\varphi_1: (x, y, z, t)\to \left(\sqrt{d_1d_2}y, \sqrt{c_1c_2}x, \sqrt{c_1d_2}t, \sqrt{d_1c_2}z\right)$$
$$\varphi_2: (x, y, z, t)\to \left(\sqrt{d_1d_3}z, \sqrt{c_1d_3}t, \sqrt{c_1c_3}x, \sqrt{d_1c_3}y\right).$$

\begin{remark}\label{LaszloRemark}
Laszlo and Pauly \cite[Proposition 4.1]{LP} gave the equation of the Kummer quartic surface
associated with the Jacobian $J(C)$ of an ordinary curve $C$ of genus 2 as follows:
$$\lambda_{10}^2(x_{00}^2x_{10}^2 + x_{01}^2x_{11}^2) + \lambda_{01}^2(x_{00}^2x_{01}^2 +
x_{10}^2x_{11}^2) + \lambda_{11}^2(x_{00}^2x_{11}^2+x_{01}^2x_{10}^2)$$
$$ + \lambda_{00}^{-1}\lambda_{10}\lambda_{01}\lambda_{11}x_{00}x_{01}x_{10}x_{11}=0.$$
Here $\{x_{ij}\}$ is a basis of $H^0(J(C), 2\Theta)$.
In the equation in Theorem \ref{KummerQuartic} (a), by putting 
$$X=x\sqrt[4]{c_1c_2c_3}, \ Y=y\sqrt[4]{d_1d_2c_3}, \ Z=z\sqrt[4]{d_1c_2d_3},\ T=t\sqrt[4]{c_1d_2d_3}$$
and then dividing by ${(a_1+a_2)(a_2+a_3)(a_3+a_1)\over \sqrt{c_1c_2c_3d_1d_2d_3}}$,
we obtain
\begin{equation}\label{LPeq}
\ 
\left\{ \begin{array}{l}
{\sqrt{c_3d_3}(a_1+a_2)\over (a_1+a_3)(a_2+a_3)}(X^2Y^2+Z^2T^2)+ {\sqrt{c_2d_2}(a_1+a_3)\over (a_1+a_2)(a_2+a_3)}(X^2Z^2+Y^2T^2)\\
\\ 
+{\sqrt{c_1d_1}(a_2+a_3)\over (a_1+a_2)(a_1+a_3)}(X^2T^2+Y^2Z^2)
+ XYZT=0
\end{array} \right.
\end{equation}
which is the same equation as Laszlo and Pauly's one above.
\end{remark}

\medskip

Next we consider the quartic surface in the case (b).

\begin{theorem}\label{KumQuarticRank1}
The surface $S_1$ has exactly two singular points 
$$P_1=\left(0,\sqrt{b_1}, \sqrt{b_2}, 0\right),\ \ P_2=\left(\sqrt{b_4}, 0, 0, \sqrt{b_3}\right)$$
of type $D_8$ and contains two tropes
$\overline{\Theta}_1$ and $\overline{\Theta}_2$ both of which are double conics cutting by
the hyperplane section $\sqrt{b_2}y + \sqrt{b_1}z=0$ and $\sqrt{b_3}x+\sqrt{b_4}t=0$, respectively.
Two tropes $\overline{\Theta}_1$ and $\overline{\Theta}_2$ meet at $P_1$ and $P_2$.
\end{theorem}
\begin{proof}
First note that $(b_1, b_2)\ne (0,0)$, $(b_3,b_4)\ne (0,0)$ (Remark \ref{coefficients}).
By the Jacobian criterion, we can see that $S_1$ has two singular points.
We will show that the Jacobian surface has 2-rank 1 (Remark \ref{CurveRank12}). 
It follows from Corollary \ref{KummerQuarticCor}, 
Proposition \ref{Basic2} and Theorem \ref{Kummer-sigularity} that
the singularities are of type $D_8$.
For the second assertion, for example, by putting the equation $\sqrt{b_2}y + \sqrt{b_1}z=0$ in 
the last term of the equation in Theorem \ref{KumQuarticRank1}, we have a decomposition
$$b_3x^2yz + b_2xy^2t + b_1xz^2t + b_4yzt^2= \sqrt{b_1/b_2}z^2\left(\sqrt{b_3}x + \sqrt{b_4}t\right)^2$$
by using $b_1b_2=b_3b_4$.  This implies the assertion.
\end{proof}

Recall that we may assume that $b_2b_4c_2=b_1b_3d_2$ for $b_1b_2b_3b_4\ne 0$ and 
$b_2^2c_2=b_3^2d_2$ for $b_1=b_4=0$ (Lemma \ref{prank1translation}, Remark \ref{prank1translationRemark}).

\begin{prop}\label{prank1translation2}
$S_1$ has an involution $\iota$ defined by
$$\iota : (x, y, z, t) \to \left(\sqrt[4]{d_1b_2^2/c_1b_3^2}y, \sqrt[4]{c_1b_3^2/d_1b_2^2}x, \sqrt[4]{c_1b_4^2/d_1b_1^2}t, \sqrt[4]{d_1b_1^2/c_1b_4^2}z\right)$$
for $b_1b_2b_3b_4\ne 0$ and 
$$\iota : (x, y, z, t) \to \left(\sqrt[4]{d_1b_2^2/c_1b_3^2}y, \sqrt[4]{c_1b_3^2/d_1b_2^2}x, \sqrt[4]{c_1b_2^2/d_1b_3^2}t, \sqrt[4]{d_1b_3^2/c_1b_2^2}z\right)$$
for $b_1=b_4=0$ which satisfies $\iota(\bar{\Theta}_1)=\bar{\Theta}_2$ and $\iota(P_1)=P_2$.
\end{prop}
\begin{proof}
The proof is straightforward.  
\end{proof}
We do not know an exact formula of a Cremona transformation interchanging $\{\bar{\Theta}_1, \bar{\Theta}_2\}$ and $\{P_1, P_2\}$
as the one given in (\ref{CremonaOrdinary}) for ordinary case.
\medskip

Finally we consider the case (c).  
Let $\xi_1= \sqrt{b_2b_4+ b_5b_8},\ \xi_2 = \sqrt{b_1b_3+b_6b_7}$.  Recall that 
$(b_1b_5+b_4b_7, b_2b_6+b_3b_8)\ne (0,0),\  (b_1b_8+b_4b_6, b_2b_7+b_3b_5) \ne (0,0)$
(Remark \ref{coefficients}) and $(\xi_1, \xi_2) \not=(0,0)$ (Lemma \ref{Invariant}). 
Let
$$x_0 =\sqrt{(b_1b_8+b_4b_6)\xi_1},\quad y_0 =\sqrt{(b_1b_5+b_4b_7)\xi_2},$$
$$z_0 =\sqrt{(b_2b_7+b_3b_5)\xi_1},\quad t_0 =\sqrt{(b_2b_6+b_3b_8)\xi_2},$$
$$x_0' =\sqrt{(b_1b_8+b_4b_6)\xi_2}, \quad y_0' =\sqrt{(b_1b_5+b_4b_7)\xi_1},$$
$$z_0' =\sqrt{(b_2b_7+b_3b_5)\xi_2}, \quad t_0' =\sqrt{(b_2b_6+b_3b_8)\xi_1}.$$

\begin{theorem}\label{ssKummerQ}
{\rm (1)}\ $S_0$ has a unique singular point 
$P_0= (x_0, y_0, z_0, t_0)$
of type $\mbox{\MARU{4}}_{0,1}^{1}$.

\smallskip
{\rm (2)}\ $S_0$ contains a trope $\bar{\Theta}$ by cutting by the hyperplane
$$z_0'x + t_0'y + x_0'z +y_0't=0.$$
\end{theorem}
\begin{proof}
(1)\ By the Jacobian criterion, we
can check that $P_0$ is a singular point of the surface by elementary but long calculation.
Later we show that $S_0$ is the quotient of the supersingular abelian surface (Remark \ref{CurveSupersingularIgusa}).  
It follows from Corollary \ref{KummerQuarticCor}, 
Proposition \ref{Basic2} and Theorem \ref{Kummer-sigularity} that
$S_0$ has a unique singularity of type $\mbox{\MARU{4}}_{0,1}^{1}$.
We remark that the point $P_0$ is nothing but the
base point of the linear system defined by six quadrics 
$$X_1=b_3x^2+b_4t^2+b_6xz+b_8yt,\ Y_1=b_2y^2+b_1z^2+b_7xz+b_5yt,$$
$$X_2=b_2xy+b_5xt+b_8yz+b_4zt,\ Y_2=b_3xy+b_7xt+b_6yz+b_1zt,$$
$$X_3= b_7x^2+b_8y^2+b_1xz+b_4yt,\ Y_3=b_6z^2+b_5t^2+b_3xz+b_2yt.$$
Here $X_i^2, Y_i^2$ are coefficients of $c_i, d_i$ when we consider the quartic equation of $S_0$
as a linear form of $c_i, d_i$. 

(2)\ Consider a hyperplane $t=\alpha x +\beta y + \gamma z$ $(\alpha, \beta, \gamma\in k)$.
By putting this into the equation of $S_0$, we have the equation 
$$f(x,y,z)^2 + xy L_1(x,y,z)^2 + yz L_2(x,y,z)^2+ zx L_3(x,y,z)^2=0$$
where $f$ is a quadric and $L_1, L_2, L_3$ are linear forms.  We can check that
 $L_1, L_2, L_3$ coincide up to constant by elementary but long calculation.
Then we choose $\alpha, \beta, \gamma$ satisfying $L_1 \equiv 0$ which gives  
the desired hyperplane.  Now a direct calculation shows that
the desired hyperplane is given in (2).
\end{proof}
We do not know an exact formula of a Cremona transformation interchanging $\bar{\Theta}$ and $P_0$
as the one given in (\ref{CremonaOrdinary}) for the ordinary case.

\section{The intersection of three quadrics}\label{sec3}

In this section, we study the sets of singular lines denoted by
$\Sigma$, $\Sigma_1$ and $\Sigma_0$ (see the definition (\ref{Kummer})).  
First we give a proof of Theorem \ref{Int3quadrics}.
Then we study singularities and lines on them, and give birational
maps from $S$ to $\Sigma$, $S_1$ to $\Sigma_1$ and $S_0$ to $\Sigma_0$, respectively.

\subsection*{Proof of Theorem \ref{Int3quadrics}}

We use Proposition \ref{singularlines}.
Consider the case (c) without assuming $a_1=a_2=a_3$.
For 
$$
x=(u_1, u_2, u_3, v_1, v_2, v_3) \in G,\ x'= (u_1', u_2', u_3', v_1', v_2', v_3')\in G\cap Q_2,
$$
the tangent space of $G$ at $x$ is given by 
$$
\sum_{i=1}^3 (v_iX_i + u_iY_i) =0,
$$
and that of $Q_2$ at $x'$ is given by
$$
(a_1v_1'+b_7u_2' +b_5v_2')X_1 + (a_1u_1'+b_6u_2'+b_8v_2')Y_1 + (a_2v_2'+b_7u_1'+b_3u_3'+ 
b_6v_1'+b_1v_3')X_2
$$
$$
 + (a_2u_2'+b_5u_1'+b_2u_3'+ b_8v_1'+b_4v_3')Y_2 + (a_3v_3'+b_3u_2'+b_2v_2')X_3 + (a_3u_3'+ b_1u_2'+b_4v_2')Y_3 =0.
$$
The condition that these two tangent spaces coincide implies that the coefficients of $X_i, Y_i$ coincide, and by putting these into $\sum_i u_iv_i=0$, we obtain the equation
$$(a_1^2 + b_5b_6+b_7b_8)u_1'v_1' + (a_2^2+ b_1b_2+b_3b_4+b_5b_6+b_7b_8)u_2'v_2' +(a_3^2 + b_1b_2+b_3b_4)u_3'v_3'$$
$$+ b_5b_7u_1'^2 + b_6b_8v_1'^2 + (b_1b_3+b_6b_7)u_2'^2 + (b_2b_4+b_5b_8)v_2'^2 + b_2b_3u_3'^2 + b_1b_4v_3'^2$$
$$+ (b_1b_5+b_4b_7)u_1'v_3' + (b_2b_6+b_3b_8)u_3'v_1' + (b_2b_7+b_3b_5)u_1'u_3' + (b_1b_8+b_4b_6)v_1'v_3' $$
$$+ (a_1+a_2)b_7u_1'u_2' + (a_1+a_2)b_5u_1'v_2' + (a_1+a_2)b_6u_2'v_1'+ (a_1+a_2)b_8v_1'v_2'$$
$$+ (a_2+a_3)b_3u_2'u_3' + (a_2+a_3)b_1u_2'v_3' + (a_2+a_3)b_2u_3'v_2'+ (a_2+a_3)b_4v_2'v_3'.$$
Now by using the equations $a_1=a_2=a_3$ and $b_1b_2=b_3b_4, \ b_5b_6=b_7b_8$, we obtain the equation (c).
By putting $b_1=\cdots = b_8=0$, we have the equation (a), and by putting $a_2=a_3$, $b_1b_2=b_3b_4$
and $b_5=\cdots = b_8=0$, we obtain the equation (b).
\qed

\medskip

\begin{remark}\label{originalequation}
Using (\ref{ordinaryPencil}), (\ref{p-rank1Pencil}) and (\ref{SupersingularPencil}), 
we see, in a similar way to the proof of 
Theorem \ref{Int3quadrics}, that the surfaces $\Sigma$, $\Sigma_1$ and $\Sigma_0$
are given as follows:

{\rm (a)} \
The surface $\Sigma$ is isomorphic to the surface defined by
$$
A_1+ \sum_{i=1}^3 (p_i^2X_i^2 + t_i^2Y_i^2)= 
\sum_{i=1}^3 (a_iX_iY_i + r_i^2X_i^2 + s_i^2Y_i^2).
$$
$$
= \sum_{i=1}^3 (a_i^2X_iY_i + a_i^2p_i^2X_i^2 + a_i^2t_i^2Y_i^2)=0.
$$

\smallskip
{\rm (b)} \ 
The surface $\Sigma_1$ is isomorphic to the surface defined by
$$
A_1 + \sum_{i=1}^3 (p_i^2X_i^2 + t_i^2Y_i^2)=
\sum_{i=1}^3 (a_iX_iY_i + r_i^2X_i^2 + s_i^2Y_i^2) + X_3Y_2.
$$
$$
= \sum_{i=1}^3 (a_i^2X_iY_i + a_i^2p_i^2X_i^2 + a_i^2t_i^2Y_i^2)+ p_2^2X_3^2 + t_3^2Y_2^2= 0
$$

\smallskip

{\rm (c)} \
The surface $\Sigma_0$ is isomorphic to the surface defined by
$$
A_1+ \sum_{i=1}^3 (p_i^2X_i^2 + t_i^2Y_i^2)=
\sum_{i=1}^3 (aX_iY_i + r_i^2X_i^2 + s_i^2Y_i^2) + X_2Y_1+X_3Y_2.
$$
$$
= \sum_{i=1}^3 (a^2X_iY_i + a^2p_i^2X_i^2 + a^2t_i^2Y_i^2)+ 
p_1^2X_2^2 +p_2^2X_3^2 + t_2^2Y_1^2 + t_3^2Y_2^2 + X_3Y_1= 0
$$

These equations are sometimes useful to examine the properties of $\Sigma$, $\Sigma_1$ and 
$\Sigma_0$, for example, to calculate the singularities.
\end{remark}

Next we study singularities and lines on $\Sigma$, $\Sigma_1$, $\Sigma_0$, and give birational
maps from $S$ to $\Sigma$, $S_1$ to $\Sigma_1$ and $S_0$ to $\Sigma_0$ individually.
First we consider the case (a).

\begin{theorem}\label{Int3quadricsOrdinary}\ 
\smallskip
{\rm (1)} \ The lines on $\Sigma$ are exactly the following eight ones$:$

\smallskip
$\widetilde{\Theta}_1 : X_1=X_2=X_3=\sum_i\sqrt{d_i}Y_i =0,$

$\widetilde{\Theta}_2 : Y_1=Y_2=X_3=\sqrt{c_1}X_1 + \sqrt{c_2}X_2 + \sqrt{d_3}Y_3 =0,$

$\widetilde{\Theta}_3 : Y_1=X_2=Y_3=\sqrt{c_1}X_1 + \sqrt{d_2}Y_2 + \sqrt{c_3}X_3 =0,$

$\widetilde{\Theta}_4 : X_1=Y_2=Y_3=\sqrt{d_1}Y_1 + \sqrt{c_2}X_2 + \sqrt{c_3}X_3 =0,$

$E_1 : Y_1=Y_2=Y_3=\sum_i\sqrt{c_i}X_i =0,$

$E_2 : X_1=X_2=Y_3=\sqrt{d_1}Y_1 + \sqrt{d_2}Y_2 + \sqrt{c_3}X_3 =0,$

$E_3 : X_1=Y_2=X_3=\sqrt{d_1}Y_1 + \sqrt{c_2}X_2 + \sqrt{d_3}Y_3 =0,$

$E_4 : Y_1=X_2=X_3=\sqrt{c_1}X_1 + \sqrt{d_2}Y_2 + \sqrt{d_3}Y_3 =0.$

\smallskip
{\rm (2)}\ The surface $\Sigma$ has exactly twelve nodes at the twelve intersection points of
$\widetilde{\Theta}_i$ and $E_j$.  In particular $\Sigma$ is a $K3$ surface with rational double points.
\end{theorem}
\begin{proof}
Since any line on $\Sigma$ is of the form $\sigma(p,h)$ cutting by planes $\sigma(p)$ and 
$\sigma(h)$, it corresponds to a singularity of the quartic surface $S$ or a trope.  Hence the number of lines is eight.   The assertion (2) follows from the Jacobian criterion and
the resolution of singularities.
\end{proof}

\noindent
The configuration of $\{\widetilde{\Theta}_i, E_j\}$ is given as in Figure \ref{EightLines}:

\begin{figure}[htbp]
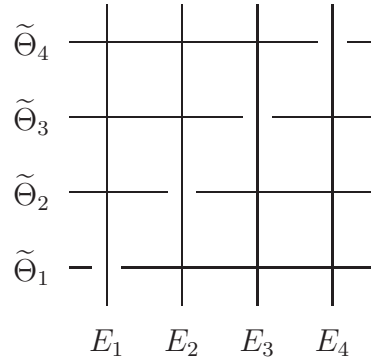

\begin{center}
\resizebox{!}{4.0cm}{
\xy
(10,0)*{};(10,40)*{}**\dir{-};
(20,0)*{};(20,40)*{}**\dir{-};
(30,0)*{};(30,40)*{}**\dir{-};
(40,0)*{};(40,40)*{}**\dir{-};
(5,5)*{};(8,5)*{}**\dir{-};
(12,5)*{};(45,5)*{}**\dir{-};
(5,15)*{};(18,15)*{}**\dir{-};
(22,15)*{};(45,15)*{}**\dir{-};
(5,25)*{};(28,25)*{}**\dir{-};
(32,25)*{};(45,25)*{}**\dir{-};
(5,35)*{};(38,35)*{}**\dir{-};
(42,35)*{};(45,35)*{}**\dir{-};
(10,-5)*{E_1};(20,-5)*{E_{2}};(30,-5)*{E_3};(40,-5)*{E_4};
(0,5)*{\widetilde{\Theta}_1};(0,15)*{\widetilde{\Theta}_2};(0,25)*{\widetilde{\Theta}_3};(0,35)*{\widetilde{\Theta}_4};
\endxy 
}
\caption{Eight lines on $\Sigma$}
\label{EightLines}
\end{center}
\end{figure}

Let $\widetilde{\Sigma}$ be the minimal resolution of $\Sigma$.
Then $\widetilde{\Sigma}$ is a $K3$ surface and contains twenty $(-2)$-curves (i.e. non-singular rational curves) which are the proper transforms of
$\widetilde{\Theta}_i$, $E_j$ and the twelve exceptional curves over the twelve nodes.
We denote the proper transforms by the same symbols $\widetilde{\Theta}_i$, $E_j$.
Then the dual graph of twenty $(-2)$-curves is given as in Figure \ref{20curves}:

\begin{figure}[htbp]
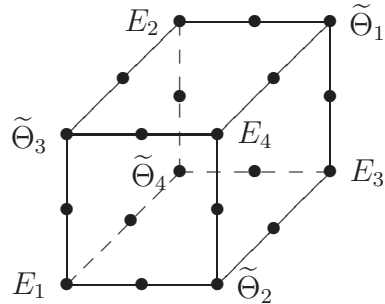

\begin{center}
\resizebox{!}{4.0cm}{
\xy
(10,5)*{};(10,25)*{}**\dir{-};
(25,20)*{};(25,40)*{}**\dir{--};
(30,5)*{};(30,25)*{}**\dir{-};
(45,20)*{};(45,40)*{}**\dir{-};
(10,5)*{};(30,5)*{}**\dir{-};
(10,25)*{};(25,40)*{}**\dir{-};
(25,20)*{};(45,20)*{}**\dir{--};
(30,25)*{};(45,40)*{}**\dir{-};
(10,25)*{};(30,25)*{}**\dir{-};
(10,5)*{};(25,20)*{}**\dir{--};
(30,5)*{};(45,20)*{}**\dir{-};
(25,40)*{};(45,40)*{}**\dir{-};
@={(10,5),(10,15),(10,25),(20,5),(20,25),(30,5),(30,15),(30,25),(25,20),(18.5,13.5),
(25,40),(45,20),(45,40),(37.5,12.5),(37.5, 32.5),(45,30),(17.5,32.5),(35,40),(35,20),(25,30)}@@{*{\bullet}};
(5,5)*{E_1};(5,25)*{\widetilde{\Theta}_{3}};(50,20)*{E_3};(35,5)*{\widetilde{\Theta}_2};
(50,40)*{\widetilde{\Theta}_1};(35,25)*{E_4};(21,20)*{\widetilde{\Theta}_4};(20,40)*{E_2};
\endxy 
}
\caption{The dual graph of 20 $(-2)$-curves on $\widetilde{\Sigma}$}
\label{20curves}
\end{center}
\end{figure}

\begin{remark}\label{Mukai-Peters}
Over the complex numbers, Peters and Stienstra \cite{PS} studied a 1-dimensional family of $K3$ surfaces
containing 20 $(-2)$-curves forming the dual graph in Figure \ref{20curves} and Mukai and Ohashi \cite{MO} also
studied quartic surfaces given in Remark \ref{LaszloRemark}.
\end{remark}

The projective transformations
$$\iota_1: (X_1, X_2, X_3, Y_1, Y_2, Y_3) \to \left(\sqrt{d_1/c_1}Y_1, X_2, X_3, \sqrt{c_1/d_1}X_1, Y_2, Y_3\right),$$
$$\iota_2: (X_1, X_2, X_3, Y_1, Y_2, Y_3) \to \left(X_1, \sqrt{d_2/c_2}Y_2, X_3, Y_1, \sqrt{c_2/d_2}X_2, Y_3\right),$$
$$\iota_3: (X_1, X_2, X_3, Y_1, Y_2, Y_3) \to \left(X_1, X_2, \sqrt{d_3/c_3}Y_3, Y_1, Y_2: \sqrt{c_3/d_3}X_3\right)$$
act on $\Sigma$ and hence induce an automorphism group of $\widetilde{\Sigma}$ isomorphic to $({\bf Z}/2{\bf Z})^3$.  The subgroup $({\bf Z}/2{\bf Z})^2$ generated by $\iota_1\circ \iota_2, \iota_1\circ \iota_3$ preserves 
$\{\widetilde{\Theta}_i\}$ and
the remaining involutions interchange $\{\widetilde{\Theta}_i\}$ and $\{E_j\}$.

\begin{prop}\label{OrdinaryStoQ}  The linear system defined by six quadrics 
$$X_1=(a_2+a_3)xt,\ X_2=(a_1+a_3)xz,\ X_3=(a_1+a_2)xy,$$
$$Y_1=(a_2+a_3)yz,\ Y_2=(a_1+a_3)yt, \ Y_3=(a_1+a_2)zt$$
gives a birational map $\mu$ from the quartic surface $S$ to $\Sigma$.  
\end{prop}
\begin{proof}
Note that $X_i^2, Y_i^2$ are coefficients of $c_i, d_i$ when we consider the quartic equation of $S$ as
a linear form of $c_i, d_i$. 
By definition of $\mu$, its image satisfies the equations $\sum_{i=1}^3 X_iY_i= \sum_{i=1}^3 a_i^2X_iY_i=0$, and
the relation $\sum_{i=1}^3 (a_iX_iY_i + c_iX_i^2 + d_iY_i^2)=0$ follows from the quartic equation of $S$.
The base locus of the linear system consists of four singular points of $S$.
Let $\tilde{S}$ be the blowing-up of the four singular points of $S$. Then $\mu$ induces a proper morphism from
$\tilde{S}$ to $\Sigma$.  Thus if we show that the inverse image $\mu^{-1}(P)$ of a point $P$ of $\Sigma$ consists of one point, then
the assertion follows.  Let $P$ be a general point of the line $\tilde{\Theta}_1$.  Then we can write
$P=(0, 0, 0, \alpha, \beta, \gamma)$ satisfying $\alpha, \beta, \gamma \in k^*$ and $\sqrt{d_1}\alpha + \sqrt{d_2}\beta +\sqrt{d_3}\gamma =0$.
Then the inverse image $\mu^{-1}(P)$ consists of only one point by solving the equation
$(xt, xz, xy, yz, yt, zt) = \left(0, 0, 0, \alpha/(a_2+a_3), \beta/(a_1+a_3), \gamma/(a_1+a_2)\right).$
\end{proof}

\begin{remark}
In Figure \ref{20curves}, by contracting sixteen $(-2)$-curves except $\widetilde{\Theta}_1,
\ldots, \widetilde{\Theta}_4$, we obtain the quartic surface $S$.  If we contract
sixteen $(-2)$-curves except $E_1,\ldots, E_4$, then we obtain the dual surface $S^\vee$.
\end{remark}

\medskip
Next we consider the case (b).

\begin{theorem}\label{Int3quadricsP-rank1}
{\rm (1)} \ The lines on $\Sigma_1$ are exactly the following four ones$:$
$$\widetilde{\Theta}_1 : Y_1=\sqrt{b_1}X_2 + \sqrt{b_2}X_3= \sqrt{b_2}Y_2 + \sqrt{b_1}Y_3 =0,$$
$$\widetilde{\Theta}_2 : X_1=\sqrt{b_3}X_2 + \sqrt{b_4}Y_3= \sqrt{b_3}X_3 + \sqrt{b_4}Y_2=0,$$
$$E_1 : X_1=\sqrt{b_1}X_2+\sqrt{b_2}X_3=\sqrt{b_2}Y_2+\sqrt{b_1}Y_3 =0,$$
$$E_2 : Y_1=\sqrt{b_3}X_2 + \sqrt{b_4}Y_3= \sqrt{b_3}X_3 +\sqrt{b_4}Y_2=0.$$
Here we give only the equation of a plane $\Pi$ such that $\Pi \cap \Sigma_1$ is a double line.
The dual graph of the four lines is a square, that is, $\widetilde{\Theta}_1 \cdot \widetilde{\Theta}_2 =
E_1\cdot E_2 =0$ and $\widetilde{\Theta}_i\cdot E_j=1, \ 1\leq i,j \leq 2$. 

\smallskip
{\rm (2)} \ The surface $\Sigma_1$ has exactly four singular points which are the intersection points of 
the four lines.
\end{theorem}
\begin{proof}
The proof of (1) is the same as that of Theorem \ref{Int3quadricsOrdinary}.  We use the condition of 
Lemma \ref{3-fold} (b) for the intersection multiplicities. 
The assertion (2) follows from the Jacobian criterion and resolutions of singularities.
\end{proof}

Consider the involutions $\iota_1$, $\iota_2$ of ${\bf P}^5$:
$$\iota_1: (X_1, X_2, X_3, Y_1, Y_2, Y_3) \to \left(\sqrt{d_1/c_1}Y_1, X_2, X_3, \sqrt{c_1/d_1}X_1, Y_2, Y_3\right);$$
In case $b_1b_2b_3b_4\ne 0$, 
$$\iota_2: (X_1, X_2, X_3, Y_1, Y_2, Y_3) \to \left(\sqrt{d_1/c_1}Y_1, {\sqrt{b_2b_4/b_1b_3}}Y_2, X_3, \sqrt{c_1/d_1}X_1, \sqrt{b_1b_3/b_2b_4}X_2, Y_3\right),$$
and in case $b_1b_2b_3b_4=0$, for example, $b_1=b_4=0$,
$$\iota_2: (X_1, X_2, X_3, Y_1, Y_2, Y_3) \to \left(\sqrt{d_1/c_1}Y_1, (b_2/b_3)Y_2, X_3, \sqrt{c_1/d_1}X_1, (b_3/b_2)X_2, Y_3\right).$$
Obviously $\iota_1$ preserves the equations of $\Sigma_1$ and hence induces an involution of $\Sigma_1$ and
$\iota_1(\widetilde{\Theta}_1) =E_1$ and $\iota_1(\widetilde{\Theta}_2)=E_2$.
Note that $\iota_2$ preserves $\Sigma_1$ if and only if $b_2b_4c_2=b_1b_3d_2$ for $b_1b_2b_3b_4\ne 0$ and 
$b_2^2c_2=b_3^2d_2$ for $b_1=b_4=0$.

\begin{prop}\label{prank1translation3}
After changing the coordinates, we may assume that $b_2b_4c_2=b_1b_3d_2$ for $b_1b_2b_3b_4\ne 0$ and 
$b_2^2c_2=b_3^2d_2$ for $b_1=b_4=0$, and hence  
the projective transformation $\iota_2$ preserves $\Sigma_1$ and $\iota_2(\widetilde{\Theta}_1)=\widetilde{\Theta}_2$ and $\iota_2(E_1)=E_2$.
\end{prop}
\begin{proof} The assertion follows from Lemma \ref{prank1translation} and Remark \ref{prank1translationRemark}.
\end{proof}

\begin{prop}\label{P-rank1StoQ}  The linear system defined by six quadrics 
$$X_1=b_3x^2+b_4t^2,\ X_2=b_2xy+b_4zt+(a_1+a_2)xz,\ X_3=b_1xz+b_4yt+(a_1+a_2)xy,$$
$$Y_1=b_2y^2+b_1z^2,\ Y_2=b_3xy+b_1zt+(a_1+a_2)yt, \ Y_3=b_3xz+b_2yt+(a_1+a_2)zt$$
gives a birational map $\mu_1$ from the quartic surface $S_1$ to $\Sigma_1$.  
\end{prop}
\begin{proof} We remark that $X_i^2, Y_i^2$ are the coefficients of $c_i, d_i$ when we consider the quartic equation of $S_1$ as a linear form of $c_i, d_i$.
The proof of the assertion is similar to the one of Proposition \ref{OrdinaryStoQ}.
In this case consider a general point $(\alpha, \sqrt{b_2}, \sqrt{b_1}, 0, \sqrt{b_1}, \sqrt{b_2})$
of the line $\tilde{\Theta}_1$\ $(\alpha \in k)$.  Then we can easily check that $(x,y,z,t)$ is uniquely determined by $\alpha$.
\end{proof}

The involution $\iota_2$ of $\Sigma_1$ corresponds to the involution $\iota$ of $S_1$ defined in
Proposition \ref{prank1translation2} under the birational map $\mu_1$.

\medskip
Finally we consider the case (c).
Let $\iota$ be the projective transformation of ${\bf P}^5$ defined by
$$\iota(X_1)= \sqrt{b_6b_8/b_5b_7}Y_1, \ \iota(Y_1)= \sqrt{b_5b_7/b_6b_8}X_1,\ \iota(X_i)=X_i,\ \iota(Y_i)=Y_i \ (i=2,3)$$
if $b_5b_6b_7b_8\ne 0$, 
$$\iota(X_1)= (b_6/b_7)Y_1,\ \iota(Y_1)= (b_7/b_6)X_1,\ \iota(X_i)=X_i,\ \iota(Y_i)=Y_i\ (i=2,3)$$
if, for example,  $b_5=b_8=0$.
By Lemma \ref{prank1translation}, Remark \ref{prank1translationRemark}, we may assume that 
$b_6b_8c_1=b_5b_7d_1$ in the first case and $b_6^2c_1=b_7^2d_1$ in the second case. 

\begin{lemma}\label{ssAut}
The involution $\iota$
acts on $\Sigma_0$ as an automorphism.
\end{lemma}
\begin{proof}
This is straightforward.
\end{proof}

\begin{theorem}\label{Int3quadricsSupersingular}
{\rm (1)}\ 
The surface $\Sigma_0$ has only one singular point $P$, which is the intersection point of the two conics
defined by $b_7X_1 + b_6Y_1 = 0$ and $b_2X_3 + b_4Y_3 = 0$ or by $b_5X_1 + b_8Y_1 = 0$ and $b_3X_3 + b_1Y_3 = 0$.
The singular point $P$ is an elliptic double point of 
type $\mbox{\MARU{14}}_{0,0,1}^{(0),(1)}$ 
{\rm (Figure \ref{ellipticSing})}.
Let $\xi_1 = \sqrt{b_2b_4 + b_5b_8}$ and 
$\xi_2 = \sqrt{b_1b_3 + b_6b_7}$.
Then, the singular point $P= (X_1, X_2, X_3, Y_1, Y_2, Y_3)$ is concretely given by
$$
\begin{array}{l}
X_1 = b_8(\xi_1\sqrt{b_1b_3c_2}+ \xi_2\sqrt{b_1b_3d_2} + 
b_1\sqrt{\xi_1\xi_2 c_3} + b_3\sqrt{\xi_1\xi_2 d_3}), \\
X_2 = \xi_1(b_8\sqrt{b_1b_3c_1} + b_5\sqrt{b_1b_3d_1} + 
b_1\sqrt{b_5b_8c_3} + b_3\sqrt{b_5b_8d_3}), \\
X_3 = b_1(b_8\sqrt{\xi_1\xi_2 c_1} + b_5\sqrt{\xi_1\xi_2 d_1}
+\xi_1 \sqrt{b_5b_8c_2} + \xi_2 \sqrt{b_5b_8d_2}), \\
Y_1 = b_5(\xi_1 \sqrt{b_1b_3c_2} + \xi_2 \sqrt{b_1b_3d_2} + 
b_1\sqrt{\xi_1\xi_2 c_3} + b_3\sqrt{\xi_1\xi_2 d_3}), \\
Y_2 = \xi_2(b_8 \sqrt{b_1b_3c_1} + b_5 \sqrt{b_1b_3d_1} + 
b_1\sqrt{b_5b_8c_3} + b_3\sqrt{b_5b_8d_3}), \\ 
Y_3 = b_3(b_8\sqrt{\xi_1\xi_2 c_1} + b_5\sqrt{\xi_1\xi_2 d_1}
+\xi_1 \sqrt{b_5b_8c_2} +\xi_2 \sqrt{b_5b_8d_2}).
\end{array}
$$
{\rm (2)}\ Let 
$x_0, y_0, z_0, t_0, x_0', y_0', z_0', t_0'$ be 
as in Theorem {\rm \ref{ssKummerQ}}.
The lines on $\Sigma_0$ are exactly the following two$:$
$$\widetilde{E}: y_0X_1 +t_0X_3+ x_0Y_2= y_0X_2+z_0X_3+x_0Y_1=z_0X_1+t_0X_2+x_0Y_3$$
$$=\left(\sqrt{c_1}x_0 +\sqrt{d_2}y_0 + \sqrt{d_3}z_0+\sqrt{b_4y_0z_0+b_5x_0y_0}\right)X_1$$
$$+\left(\sqrt{c_2}x_0 +\sqrt{d_1}y_0 + \sqrt{d_3}t_0+\sqrt{b_1x_0t_0+b_6x_0y_0}\right)X_2$$
$$+\left(\sqrt{c_3}x_0 +\sqrt{d_1}z_0 + \sqrt{d_2}t_0+\sqrt{b_2x_0t_0+b_8z_0t_0}\right)X_3=0$$
and
$$\widetilde{\Theta}: x_0'Y_1 +z_0'X_3+ y_0'Y_2= x_0'X_2+t_0'X_3+y_0'X_1=t_0'Y_1+z_0'X_2+y_0'Y_3$$
$$=\left(\sqrt{d_1}y_0' +\sqrt{d_2}x_0' + \sqrt{d_3}t_0'+\sqrt{b_4x_0't_0'+b_8x_0'y_0'}\right)Y_1$$
$$+\left(\sqrt{c_2}y_0' +\sqrt{c_1}x_0' + \sqrt{d_3}z_0'+\sqrt{b_1y_0'z_0'+b_7x_0'y_0'}\right)X_2$$
$$+\left(\sqrt{c_3}y_0' +\sqrt{c_1}t_0' + \sqrt{d_2}z_0'+\sqrt{b_2y_0'z_0'+b_5z_0't_0'}\right)X_3=0,$$
respectively.  
Moreover
$\widetilde{\Theta} = \iota(\widetilde{E})$,
and $\widetilde{E}$ and $\widetilde{\Theta}$ meet at $P$.
\end{theorem}
\begin{proof}
Note that by Lemma \ref{Invariant} we have $(\xi_1, \xi_2) \neq (0, 0)$.
To calculate the singularities of $\Sigma_0$, we use the equations 
in Remark \ref{originalequation} {\rm (c)}.
Subtracting suitable constant multiples of the first equation
from the second and the third equations, 
the surface $\Sigma_0$ is  isomorphic to the surface defined by the following three equations:
$$
\begin{array}{l}
(1)~ \sum_{i=1}^3 (X_iY_i + p_i^2X_i^2 + t_i^2Y_i^2)= 0,\\
(2) ~\sum_{i=1}^3 \{(r_i^2 + ap_i^2)X_i^2 + (s_i^2 + at_i^2)Y_i^2\} + X_2Y_1+X_3Y_2= 0,\\
(3)~ p_1^2X_2^2 +p_2^2X_3^2 + t_2^2Y_1^2 + t_3^2Y_2^2 + X_3Y_1= 0
\end{array}
$$
The Jacobian matrix with respect to the coordinates $(X_1, X_2, X_3, Y_1, Y_2, Y_3)$
is given by
$$
\left(
\begin{array}{cccccc}
 Y_1 & Y_2 & Y_3 & X_1 & X_2 & X_3\\
 0  & Y_1  & Y_2 & X_2 & X_3  & 0\\
 0  & 0 &   Y_1  & X_3  &  0  & 0
\end{array}
\right).
$$
In this matrix, if $Y_1 \neq 0$, then the first, the second and the third rows are linearly
independent. If $X_3 \neq 0$, then the 4th, the 5th and the 6th rows are linearly
independent. Therefore, singular points must satisfy the equation $Y_1 = X_3 = 0$.
In the change of coordinates (\ref{CC}), we have 
$$
     Y_1 = \gamma_1x_1 + \alpha_1y_1, X_3 = \delta_3 x_3 + \beta_3 y_3.
$$
Therefore, for the equations which define $\Sigma_0$ in Theorem \ref{Int3quadrics},
the singularities are defined by the following two equations:
\begin{equation}\label{1}
       \gamma_1X_1 + \alpha_1Y_1 = 0,
\end{equation}
\begin{equation}\label{2}
       \delta_3 X_3 + \beta_3 Y_3 = 0.
\end{equation}
Note that either $\alpha_2\delta_2 \neq 0$ or $\beta_2\gamma_2\neq 0$ holds
by $\alpha_2\delta_2 + \beta_2\gamma_2 = 1$.
If $\alpha_2\delta_2 \neq 0$, then multiplying (\ref{1}) by $\delta_2$ and 
(\ref{2}) by $\alpha_2$, we have $b_7X_1 + b_6Y_1 = 0$ and $b_2X_3 + b_4Y_3 = 0$.
If $\beta_2\gamma_2 \neq 0$, multiplying (\ref{1}) by $\beta_2$ and 
(\ref{2}) by $\gamma_2$, we have  $b_5X_1 + b_8Y_1 = 0$ and $b_3X_3 + b_1Y_3 = 0$.
Since we are in charactersitic 2, using these equations, we have 3 linear equations
from the defining equations of $\Sigma_0$.
Solving the equations, we get the concrete coordinates of the only one singular point $P$.

Let $\Pi$ be the plane defined by
$$y_0X_1 +t_0X_3+ x_0Y_2= y_0X_2+z_0X_3+x_0Y_1=z_0X_1+t_0X_2+x_0Y_3=0.$$
In Theorem \ref{ssKummerQ}, we showed that the quartic surface $S_0$
has a singular point $P_0=(x_0, y_0, z_0, t_0)$.  The plane $\Pi$ is nothing but the one
consisting of lines in ${\bf P}^3$ passing through $P_0$.  We can also check that 
$\Pi$ cuts $\Sigma_0$ along the double line $\widetilde{E}$, 
$\widetilde{E}$ contains the singular point $P$ and $\widetilde{\Theta} = \iota(\widetilde{E})$ 
by direct but long calculation.  
Since $P$ is the unique singular point, it is fixed by $\iota$ and hence $\iota(\widetilde{E})$ also
contains $P$.  Therefore $\widetilde{E}$ and $\widetilde{\Theta}$ meet at $P$.

Recall that there exists a canonical morphism 
$\pi: \Sigma_0 \to S_0$ by sending $x$ to the focus of $\ell_x$ (see \S \ref{ClassicalQLC})
which is nothing but the one sending $\tilde{E}$ to $P_0$.
It is easy to see that $\pi$ is the blowing-up of $P_0 \in S_0$.
On the other hand, it follows from the canonical resolution given in Katsura \cite[\S6]{Ka} that
the minimal resolution of a singularity of type $\mbox{\MARU{14}}_{0,0,1}^{(0),(1)}$ 
is obtained from the one of a singularity of type $\mbox{\MARU{4}}_{0,1}^{1}$ by blowing-up a point
on the $(-3)$-curve (Figure \ref{ellipticSing}).   Since $P_0$ is of type $\mbox{\MARU{4}}_{0,1}^{1}$ (Theorem \ref{ssKummerQ}),
$P$ is of type $\mbox{\MARU{14}}_{0,0,1}^{(0),(1)}$.
\end{proof}

\begin{remark}\label{}
In Theorem \ref{Int3quadricsSupersingular} (2), if $x_0=0$, then the equations 
$$y_0X_1 +t_0X_3+ x_0Y_2= y_0X_2+z_0X_3+x_0Y_1=z_0X_1+t_0X_2+x_0Y_3=0$$ 
define a 3-dimensional subspace instead of a plane.
However, together with the relation $\sum X_iY_i=0$, it defines a unique line. The same thing holds if
$y_0'=0$. 
\end{remark}

\begin{prop}\label{SupersingularStoQ}  The linear system defined by six quadrics 
$$X_1=b_3x^2+b_4t^2+b_6xz+b_8yt,\ Y_1=b_2y^2+b_1z^2+b_7xz+b_5yt,$$
$$X_2=b_2xy+b_5xt+b_8yz+b_4zt,\ Y_2=b_3xy+b_7xt+b_6yz+b_1zt,$$
$$X_3= b_7x^2+b_8y^2+b_1xz+b_4yt,\ Y_3=b_6z^2+b_5t^2+b_3xz+b_2yt$$
gives a birational map $\mu_0$ from the quartic surface $S_0$ to $\Sigma_0$.  
\end{prop}
\begin{proof}
Note that $X_i^2, Y_i^2$ are coefficients of $c_i, d_i$ when we consider the quartic equation of $S_0$
as a linear form of $c_i, d_i$.
The proof is similar to the one of Proposition \ref{OrdinaryStoQ}.
By using the relations $b_1b_2=b_3b_4$, $b_5b_6=b_7b_8$, we have
$$\sqrt{b_2b_6+b_3b_8}y + \sqrt{b_1b_5+b_4b_7}t = \sqrt{b_7X_1+b_6Y_1+b_3X_3+b_1Y_3},$$
$$\sqrt{b_2b_7+b_3b_5}x +\sqrt{b_1b_8+b_4b_6}z =\sqrt{b_5X_1+b_8Y_1+b_2X_3+b_4Y_3},$$
$$(b_2b_6+b_3b_8)xy+(b_1b_8+b_4b_6)zt=b_6X_2+b_8Y_2.$$
It follows that $(x, y, z, t)$ is uniquely determined by $(X_1, X_2, X_3, Y_1, Y_2, Y_3)$.
\end{proof}

\section{Curves of genus 2}\label{sec5}

In this section, we prove Theorem \ref{Curve} and give Igusa's normal forms of curves of genus two 
(see (\ref{g=2})).  
First we prepare a general theory and then apply it to
ordinary, 2-rank 1 and supersingular case individually.

\subsection{General theory}
Let $f_i(t)$ and $g_i(t)$ ($i = 1, 2, 3$) be non-zero rational functions,
and let $a_i$ ($i = 1, 2, 3$) be elements of $k$. Let
\begin{equation}\label{normalpencil}
\begin{array}{rl}
    Q(t) &= f_1(t)X_1^2 + (t + a_1)X_1Y_1 + g_1(t)Y_1^2 \\
    &\quad + f_2(t)X_2^2 + (t + a_2)X_2Y_2 + g_2(t)Y_2^2 \\
    & \quad + f_3(t)X_3^2 + (t + a_3)X_3Y_3 + g_3(t)Y_3^2 = 0
\end{array}
\end{equation}
be a pencil of quadrics. Let $G(3, 6)$ be the Grassmannian of 3-dimensional subspaces
of $k^6$ and let $Z$ be the correspondence variety in ${\bf P}^1 \times G(3, 6)$ defined by
$$
      Z = \{(t, W)\mid W \mbox{is totally singular for}~ Q(t)\}.
$$

For a general point $t \in {\bf A}^1 \subset {\bf P}^1$, 
the fiber of the morphism $f : Z \longrightarrow {\bf P}^1$ has 
two connected components. We set $C = {\rm Spec}(f_{*}{\mathcal O}_Z)$.
Then, $C$ is a double cover of ${\bf P}^1$ (cf. Bhosle \cite[Proposition 2.16]{B}). 
We calculate the concrete equation of $C$.

For this purpose, we factorize polynomials $f_i(t)X_i^2 + (t + a_i)X_iY_i + g_i(t)Y_i^2$
($i = 1, 2, 3$) over the algebraic closure of $k(t)$
as
$$
  f_i(t)X_i^2 + (t + a_i)X_iY_i + g_i(t)Y_i^2 
  = \{\sqrt{f_i(t)}(X_i + \alpha_i Y_i)\}\{\sqrt{f_i(t)}(X_i + \alpha'_i Y_i)\}.
$$
Here, we have
$$
 \alpha_i + \alpha'_i = \frac{t + a_i}{f_i(t)}~\mbox{and}~\alpha_i \alpha'_i = \frac{g_i(t)}{f_i(t)} \quad (i = 1, 2, 3).
$$
We have 3 vectors defined by the equations
$$
\left\{
\begin{array}{l}
X_1 + \alpha_1Y_1 = \frac{1}{\sqrt{f_1(t)}},~X_1 + \alpha'_1Y_1 = 0, \\
X_2 + \alpha_2Y_2 = 0,~X_2 + \alpha'_2Y_2 = \frac{a}{\sqrt{f_2(t)}}, \\
X_3 + \alpha_3Y_3 = 0,~X_3 + \alpha'_3Y_3 = \frac{b}{\sqrt{f_3(t)}},
\end{array}
\right.
$$
$$
\left\{
\begin{array}{l}
X_1 + \alpha_1Y_1 = 0,~X_1 + \alpha'_1Y_1 = \frac{a}{\sqrt{f_1(t)}}, \\
X_2 + \alpha_2Y_2 = \frac{1}{\sqrt{f_2(t)}},~X_2 + \alpha'_2Y_2 = 0, \\
X_3 + \alpha_3Y_3 = 0,~X_3 + \alpha'_3Y_3 = \frac{c}{\sqrt{f_3(t)}},
\end{array}
\right.
$$
$$
\left\{
\begin{array}{l}
X_1 + \alpha_1Y_1 = 0,~X_1 + \alpha'_1Y_1 = \frac{b}{\sqrt{f_1(t)}}, \\
X_2 + \alpha_2Y_2 = 0,~X_2 + \alpha'_2Y_2 = \frac{c}{\sqrt{f_2(t)}}, \\
X_3 + \alpha_3Y_3 = \frac{1}{\sqrt{f_3(t)}},~X_3 + \alpha'_3Y_3 = 0,
\end{array}
\right.
$$
which are a basis of a 3-dimensional vector space $V$ corresponding with a point of $Z$ over 
a general point $t \in {\bf A}^1 \subset {\bf P}^1$. Here, $a$, $b$ and $c$ are arbitrary elements of $k$.
Solving each system of equations, we have 3 vectors ${\bf u}_1$, 
${\bf u}_2$ and ${\bf u}_3$ whose coordinates are arranged as $(X_1,X_2, X_3, Y_1, Y_2, Y_3)$:
$$
\left\{
\begin{array}{l}
{\bf u}_1 = \left(\frac{\sqrt{f_1(t)}\alpha'_1}{t + a_1}, \frac{\sqrt{f_2(t)}a\alpha_2}{t + a_2}, \frac{\sqrt{f_3(t)}b\alpha_3}{t + a_3}, \frac{\sqrt{f_1(t)}}{t + a_1}, \frac{\sqrt{f_2(t)}a}{t + a_2}, \frac{\sqrt{f_3(t)}b}{t + a_3}\right),
\\
{\bf u}_2 = \left(\frac{\sqrt{f_1(t)}a\alpha_1}{t + a_1}, \frac{\sqrt{f_2(t)}\alpha'_2}{t + a_2}, \frac{\sqrt{f_3(t)}c\alpha_3}{t + a_3}, \frac{\sqrt{f_1(t)}a}{t + a_1}, \frac{\sqrt{f_2(t)}}{t + a_2}, \frac{\sqrt{f_3(t)}c}{t + a_3}\right),
\\
{\bf u}_3 = \left(\frac{\sqrt{f_1(t)}b\alpha_1}{t + a_1}, \frac{\sqrt{f_2(t)}c\alpha_2}{t + a_2}, 
\frac{\sqrt{f_3(t)}\alpha'_3}{t + a_3}, \frac{\sqrt{f_1(t)}b}{t + a_1}, \frac{\sqrt{f_2(t)}c}{t + a_2}, 
\frac{\sqrt{f_3(t)}}{t + a_3}\right).
\end{array}
\right.
$$
We set 
$$
\left\{
\begin{array}{l}
{\bf v}_ 1=  \frac{t + a_1}{\sqrt{f_1(t)}}{\bf u}_1,\\ 
{\bf v}_ 2=  \sqrt{f_2(t)}({\bf u}_2 + {\bf u}_1),\\
{\bf v}_ 3=  \sqrt{f_3(t)}({\bf u}_3 + {\bf u}_1),
\end{array}
\right.
$$
Again we set
$$
\left\{
\begin{array}{l}
{\bf w}_ 1=  {\bf v}_1 + \sqrt{\frac{f_2(t)}{f_1(t)}}\frac{t + a_1}{t + a_2}\alpha_2{\bf v}_2 + \sqrt{\frac{f_3(t)}{f_1(t)}}\frac{t + a_1}{t + a_3}\alpha_3{\bf v}_3,\\ 
{\bf w}_ 2=  {\bf v}_2, \\
{\bf w}_ 3=  {\bf v}_3.
\end{array}
\right.
$$
Then, $\{{\bf w}_1, {\bf w}_2, {\bf w}_3\}$ is also a basis of $V$ and the matrix
$$
\left(
\begin{array}{c}
{\bf w}_1 \\
{\bf w}_2 \\
 {\bf w}_3
\end{array}
\right)
$$
is the homogeneous coordinates for the point in $G(3, 6)$ which corresponds with $V$.
Putting $a= b = c= 1$ in the matrix, we have a multi-section of $f : Z \longrightarrow {\bf P}^1$.
The homogeneous coordinates are given by 
$$
\left(
\begin{array}{cccccc}
\alpha_{1}' + \frac{t + a_1}{t + a_2}\frac{f_2(t)\alpha_2}{f_1(t)} 
+ \frac{t + a_1}{t + a_3}\frac{f_3(t)\alpha_3}{f_1(t)}
& 0&
0 & 1 & \sqrt{\frac{f_2(t)}{f_1(t)}}\frac{t + a_1}{t + a_2}& \sqrt{\frac{f_3(t)}{f_1(t)}}\frac{t + a_1}{t + a_3} \\
\sqrt{\frac{f_2(t)}{f_1(t)}} & 1 & 0 & 0 & 0 & 0 \\
\sqrt{\frac{f_3(t)}{f_1(t)}} & 0 & 1 & 0 & 0 & 0
\end{array}
\right).
$$
and certain affine coordinates for the point in $G(3, 6)$ which corresponds with $V$ is given by
$$
\left(
\begin{array}{ccc}
\alpha_{1}' + \frac{t + a_1}{t + a_2}\frac{f_2(t)\alpha_2}{f_1(t)} 
+ \frac{t + a_1}{t + a_3}\frac{f_3(t)\alpha_3}{f_1(t)}
& \sqrt{\frac{f_2(t)}{f_1(t)}}\frac{t + a_1}{t + a_2}& \sqrt{\frac{f_3(t)}{f_1(t)}}\frac{t + a_1}{t + a_3} \\
\sqrt{\frac{f_2(t)}{f_1(t)}}  & 0 & 0 \\
\sqrt{\frac{f_3(t)}{f_1(t)}}  & 0 & 0
\end{array}
\right).
$$
Using $(1, 1)$ component of this matrix, we set
$$
z = (t + a_2)(t + a_3)f_1(t)\left(\alpha'_1 + \frac{t + a_1}{t + a_2}\frac{f_2(t)\alpha_2}{f_1(t)} 
+ \frac{t + a_1}{t + a_3}\frac{f_3(t)\alpha_3}{f_1(t)}\right).
$$
Then, we know
$$
 z = (t + a_2)(t + a_3)f_1(t)\alpha'_1 
 + (t + a_1)(t + a_3)f_2(t)\alpha_2
  + (t + a_1)(t + a_2)f_3(t)\alpha_3
$$
and we have
\begin{equation}\label{standardcurve}
\begin{array}{l}
 z^2 + (t + a_1)(t + a_2) (t + a_3)z \\
 =(t + a_2)^2(t + a_3)^2f_1(t)\{f_1(t){\alpha'_1}^2 +(t + a_1)\alpha'_1\}\\
\quad  + (t + a_1)^2(t + a_3)^2f_2(t)\{f_2(t)\alpha_2^2 +(t + a_2)\alpha_2\}\\
\quad  +(t + a_1)^2(t + a_2)^2f_3(t)\{f_3(t)\alpha_3^2 +(t + a_3)\alpha_3\}\\
 =(t + a_2)^2(t + a_3)^2f_1(t)g_1(t) + 
 (t + a_1)^2(t + a_3)^2f_2(t)g_2(t)  \\
\quad + (t + a_1)^2(t + a_2)^2f_3(t)g_3(t). 
 \end{array}
\end{equation}
We denote by $C$ the curve defined by this equation (\ref{standardcurve}).

Now, we consider 3 cases (a), (b) and (c) in Proposition \ref{CanFormPencilQuadrics}.
By the explicit calculation below, we see that the multi-section which we constructed
becomes a section in Case $({\rm a})$, and a 2-section (purely inseparable of degree 2
over the base curve) in Cases $({\rm b})$ and $({\rm c})$,
after the base change using the curve $C$.

\subsection*{Proof of Theorem \ref{Curve} (a)}

In this case the pencil $\calP$ of quadrics is given by
$$
      \sum_{i = 1}^{3}\left\{c_iX_i^2 + (t + a_i)X_iY_i + d_iY_i^2\right\} = 0.
$$
Therefore, we set $f_{i}(t) = c_i$ and $g_i(t) = d_i$. Thus we have proved the theorem.
\qed

\begin{remark}\label{CurveOrdinary2}
The equation of Igusa's normal form of $C$ is given by
$$
y^2+ y = {\sqrt{c_1d_1}(a_2+a_3)\over (a_1+a_2)(a_1+a_3)}x + {\sqrt{c_2d_2}(a_1+a_3)\over (a_1+a_2)(a_2+a_3)x} + {\sqrt{c_3d_3}(a_1+a_2)\over (a_1+a_3)(a_2+a_3)(x+1)}.
$$
The Jacobian variety $J(C)$ is ordinay, that is, of 2-rank 2.
Recall that the coefficients of Igusa's canonical model are different from 0, that is, $\prod c_id_i\ne 0$ (see (\ref{g=2})) which
coincides with the condition of the smoothness of $\calX$ (Lemma \ref{3-fold}(a)).
Combining this equation and Laszlo and Pauly \cite[Proposition 3.1]{LP2}, we obtain the equation
(\ref{LPeq}) of the Kummer quartic surface.  It follows from Proposition \ref{Basic2} that 
$S$ is a Kummer quartic surface $J(C)/\la \iota\ra$.
\end{remark}

\medskip

\subsection*{Proof of Theorem \ref{Curve} (b)}

In this case the pencil of quadrics is given by
$$
      \sum_{i = 1}^{3}\left\{c_iX_i^2 + (t + a_i)X_iY_i + d_iY_i^2\right\} + B = 0,
$$
$$
 B = b_1X_2Y_3 + b_2X_3Y_2 + b_3X_2X_3 + b_4Y_2Y_3, \ \mbox{with}~b_1b_2=b_3b_4
$$
as in (\ref{(b)}).  By the condition $b_2b_4c_2 + b_1b_4c_3 + b_1b_3d_2 + b_2b_3d_3\ne 0$ 
for the smoothness of $\calX_1$ (Lemma \ref{3-fold}), we may assume $b_4\ne 0$.
Since $B = b_4\left(Y_2 + \frac{b_1}{b_4} X_2\right)\left(Y_3 + \frac{b_2}{b_4} X_3\right)$, 
we set
$$
   Z_2 = Y_2 + \frac{b_1}{b_4} X_2,~ Z_3 = Y_3 + \frac{b_2}{b_4} X_3.
$$
Moreover, we set
$$
   W_3 = X_3 + \frac{b_4}{t + a_3}Z_2.
$$
Then, our pencil becomes
\begin{equation}
\begin{array}{l} 
c_1X_1^2 + (t + a_1)X_1Y_1 + d_1Y_1^2 \\
+\left(c_2 + \frac{b_1(t + a_2)}{b_4} + \frac{b_1^2d_2}{b_4^2}\right)X_2^2 + (t + a_2)X_2Z_2
+ \left(\frac{b_2^2 d_3 + b_4^2c_3}{(t + a_3)^2} + \frac{b_2b_4}{t + a_3} 
 + d_2\right)Z_2^2\\
+\left(c_3 + \frac{b_2(t + a_3)}{b_4} + \frac{b_2^2d_3}{b_4^2}\right)W_3^2 + (t + a_3)W_3Z_3 + d_3Z_3^2= 0.
\end{array}
\end{equation}
Comparing this equation with (\ref{normalpencil}), we see that the curve $C_1$ is
given by the equation
$$
\begin{array}{l}
 z^2 + (t + a_1)(t + a_2) (t + a_3)z  \\
 = (t + a_2)^2(t + a_3)^2c_1d_1 \\
 +(t + a_1)^2(t + a_3)^2\left(c_2 + \frac{b_1(t + a_2)}{b_4} 
 + \frac{b_1^2d_2}{b_4^2}\right)\left(\frac{b_2^2 d_3 + b_4^2c_3}{(t + a_3)^2}
  + \frac{b_2b_4}{t + a_3} + d_2\right)\\
 + (t + a_1)^2(t + a_2)^2\left(c_3 + \frac{b_2(t + a_3)}{b_4} + \frac{b_2^2d_3}{b_4^2}\right)d_3.
\end{array}
$$
Using $a_2= a_3$ and $b_1b_2 = b_3b_4$, we have
$$
\begin{array}{l} 
z^2 + (t + a_1)(t + a_2)^2z \\
= (t + a_2)^4c_1d_1 
+ (t + a_1)^2(t + a_2)^2\left\{b_1b_2 + c_2d_2 + c_3d_3 + (\frac{b_1d_2 + b_2d_3}{b_4})^2\right\}\\
+ (t + a_1)^2(t + a_2)^3(\frac{b_1d_2 + b_2d_3}{b_4}) + (t + a_1)^2(t + a_2)\{b_3(b_1d_2 + b_2d_3)+ b_4(b_2c_2 + b_1c_3)\}\\
+ (t + a_1)^2\{b_1^2c_3d_2 + b_2^2c_2d_3 + b_3^2d_2d_3 + b_4^2c_2c_3\}.
\end{array}
$$
Setting
$$
z = v + (t + a_1)(b_1\sqrt{c_3d_2} + b_2\sqrt{c_2d_3} + b_3\sqrt{d_2d_3} + b_4\sqrt{c_2c_3}) +
(t+a_1)(t + a_2)\frac{b_1d_2 + b_2d_3}{b_4},
$$
we obtain
$$
\begin{array}{l}
v^2 +(t + a_1)(t + a_2)^2v = (t +a_2)^4c_1d_1\\
 +(t+a_1)^2(t+a_2)^2\left(b_1b_2 + c_2d_2+c_3d_3 + b_1\sqrt{c_3d_2} + b_2\sqrt{c_2d_3} + b_3\sqrt{d_2d_3} + b_4\sqrt{c_2c_3}\right)\\
+(t + a_1)^2(t + a_2)\{b_3(b_1d_2 + b_2d_3) + b_4(b_2c_2 + b_1c_3)\},
\end{array}
$$
where $c_1d_1\ne 0, b_3(b_1d_2 + b_2d_3) + b_4(b_2c_2 + b_1c_3)\ne 0$.
If $c_1d_1= 0$ or $b_3(b_1d_2 + b_2d_3) + b_4(b_2c_2 + b_1c_3)= 0$, then $C_1$ is singular.
Thus we have proved the theorem.
\qed

\medskip

\begin{remark}\label{CurveRank12}
The equation of Igusa's normal form of $C_1$ is given by
$$
y^2+ y = x^3 + \alpha x + \beta x^{-1}
$$
with
$$
\begin{array}{l}
\alpha = \frac{\sqrt[6]{b_4(b_2c_2+b_1c_3) +b_3(b_1d_2 + b_2d_3)}}{\sqrt{a_1} + \sqrt{a_2}}
+\frac{\sqrt{c_2d_2 + c_3d_3 + b_1b_2 + b_4\sqrt{c_2c_3} + b_2\sqrt{c_2d_3} +
b_1\sqrt{c_3d_2} + b_3\sqrt{d_2d_3}}}{\sqrt[3]{b_4(b_2c_2+b_1c_3) +b_3(b_1d_2 + b_2d_3)}}\\
\quad \quad + \frac{\sqrt[3]{\{b_4(b_2c_2+b_1c_3) +b_3(b_1d_2 + b_2d_3)\}^2}}{a_1^2 + a_2^2}\\
\beta = \frac{\sqrt{c_1d_1}\sqrt[3]{b_4(b_2c_2+b_1c_3) +b_3(b_1d_2 + b_2d_3)}}{a_1^2 + a_2^2}.
\end{array}
$$
The Jacobian variety $J(C_1)$ is of 2-rank 1.
Recall that $\beta\ne 0$ (see (\ref{g=2})), that is, 
$c_1d_1\ne 0$ and $b_4(b_2c_2+b_1c_3) +b_3(b_1d_2 + b_2d_3)\ne 0$ which
coincides with the condition of the smoothness of $\calX_1$ (Lemma \ref{3-fold}(b)).
It follows from Proposition \ref{Basic2} that 
$S_1$ is a Kummer quartic surface $J(C_1)/\la \iota\ra$.
\end{remark}

\medskip

\subsection*{Proof of Theorem \ref{Curve} (c)}

In this case the pencil of quadrics is given by
$$
      \sum_{i = 1}^{3}\left\{c_iX_i^2 + (t + a_i)X_iY_i + d_iY_i^2\right\} + C = 0,
$$
$$
\begin{array}{l}
 C= b_1X_2Y_3 + b_2X_3Y_2 + b_3X_2X_3 +b_4Y_2Y_3 + b_5X_1Y_2+b_6X_2Y_1+b_7X_1X_2+
b_8Y_1Y_2, \\\

\end{array}
$$ 
with $b_1b_2=b_3b_4$ and $b_5b_6=b_7b_8$ as in (\ref{(c)}).
By the condition for the smoothness of $\calX_0$ (Lemma \ref{3-fold}), 
the cases $b_1=b_2=b_3=b_4=0$ and $b_5=b_6=b_7=b_8=0$ do not occur.
So we may assume that $b_4\ne 0$ and $b_8\ne0$.  In this case $b_1b_8+b_4b_6\ne 0$.
In fact if $b_1b_8+b_4b_6= 0$, then we have $b_2b_6+b_3b_8=b_1b_5+b_4b_7=0$ by the relations 
$b_2(b_1b_8+b_4b_6)= b_4(b_2b_6+b_3b_8)$ and $b_5(b_1b_8+b_4b_6)= b_8(b_1b_5+b_4b_7)$,
which contradicts the smoothness of $\calX_0$ (Lemma \ref{3-fold}).
Thus in the following we may assume that $b_4\ne 0, b_8\ne 0$ and $b_1b_8+b_4b_6\ne 0$.

Since $b_1b_2=b_3b_4$ and $b_5b_6=b_7b_8$, we have
$$
C = b_4\left(\frac{b_1}{b_4}X_2 + Y_2\right)\left(\frac{b_2}{b_4}X_3 + Y_3\right) + b_8\left(\frac{b_5}{b_8}X_1 + Y_1\right)\left(\frac{b_6}{b_8}X_2 + Y_2\right).
$$
We set 
$$
\begin{array}{l}
Z_2 = \frac{b_1}{b_4}X_2 + Y_2,~Z_3 = \frac{b_2}{b_4}X_3 + Y_3,\\
Z_1 = \frac{b_5}{b_8}X_1 + Y_1, ~ W_2 = \frac{b_6}{b_8}X_2 + Y_2.
\end{array}
$$
Then, we have
$$
\begin{array}{l}
X_2 = \frac{b_4b_8}{b_1b_8 + b_4b_6}(Z_2 + W_2),\\
Y_2 = \frac{b_4b_6}{b_1b_8 + b_4b_6}Z_2 + \frac{b_1b_8}{b_1b_8 + b_4b_6}W_2,
\end{array}
$$
and our pencil becomes
$$
\begin{array}{l} 
\left(c_1 + \frac{b_5^2d_1}{b_8^2} + (t +a_1)\frac{b_5}{b_8}\right)X_1^2 + (t + a_1)X_1Z_1 
+ d_1Z_1^2\\
+ \frac{b_4^2b_8^2c_2 + b_1^2b_8^2d_2 + (t + a_2)b_1b_4b_8^2}{(b_1b_8 + b_4b_6)^2}W_2^2
+ (t+a_2)\frac{b_4b_8}{b_1b_8+b_4b_6}W_2Z_2 
+ \frac{b_4^2b_8^2c_2 + b_4^2b_6^2d_2 + (t + a_2)b_4^2b_6b_8}{(b_1b_8 + b_4b_6)^2}Z_2^2 \\
+ \left(c_3 + \frac{b_2^2d_3}{b_4^2} + (t+a_3)\frac{b_2}{b_4}\right)X_3^2
  + (t+a_3)X_3Z_3 
+ d_3Z_3^2
+b_4Z_2Z_3 + b_8Z_1W_2= 0.
\end{array}
$$
We set
$$
W_1 = X_1 + \frac{b_8}{t + a_1}W_2,~ W_3 = X_3 + \frac{b_4}{t + a_3}Z_2.
$$
Then, we have
$$
\begin{array}{l} 
\left(c_1 + \frac{b_5^2d_1}{b_8^2} + (t +a_1)\frac{b_5}{b_8}\right)W_1^2 + (t + a_1)W_1Z_1 
 + d_1Z_1^2 \\
 + \left\{\frac{b_8^2c_1 + b_5^2d_1 + (t+a_1)b_5b_8}{(t+a_1)^2} +\frac{b_4^2b_8^2c_2 + b_1^2b_8^2d_2 + (t + a_2)b_1b_4b_8^2}{(b_1b_8 + b_4b_6)^2}\right\}W_2^2 \\ 
 + (t+a_2)\frac{b_4b_8}{b_1b_8+b_4b_6}W_2Z_2 
 + \left\{\frac{b_4^2c_3 + b_2^2d_3 + (t + a_3)b_2b_4}{(t + a_3)^2} 
 + \frac{b_4^2b_8^2c_2 + b_4^2b_6^2d_2 +(t + a_2)b_4^2b_6b_8}{(b_1b_8 + b_4b_6)^2}\right\}Z_2^2\\
+ \left(c_3 + \frac{b_2^2d_3}{b_4^2} + (t+a_3)\frac{b_2}{b_4}\right)W_3^2
+ (t+a_3)W_3Z_3 + 
d_3Z_3^2 = 0.
\end{array}
$$
Setting
$$
\tilde{W}_2 = \frac{\sqrt{b_4b_8}}{\sqrt{b_1b_8+b_4b_6}}W_2 , ~\tilde{Z}_2 =  \frac{\sqrt{b_4b_8}}{\sqrt{b_1b_8+b_4b_6}}Z_2,
$$
we have
$$
\begin{array}{l} 
\left(c_1 + \frac{b_5^2d_1}{b_8^2} + (t +a_1)\frac{b_5}{b_8}\right)W_1^2 
+ (t + a_1)W_1Z_1 + d_1Z_1^2 \\
+
\left\{\frac{(b_8^2c_1 + b_5^2d_1 + (t+a_1)b_5b_8)(b_1b_8+b_4b_6)}{(t+a_1)^2b_4b_8} +\frac{b_4^2b_8c_2 + b_1^2b_8d_2 + (t + a_2)b_1b_4b_8}{(b_1b_8 + b_4b_6)b_4}\right\}\tilde{W}_2^2\\
 + (t+a_2)\tilde{W}_2\tilde{Z}_2
 + 
\left\{\frac{(b_4^2c_3 + b_2^2d_3 + (t + a_3)b_2b_4)(b_1b_8+b_4b_6)}{(t + a_3)^2b_4b_8} + \frac{b_4b_8^2c_2 + b_4b_6^2d_2 + (t + a_2)b_4b_6b_8}{(b_1b_8 + b_4b_6)b_8}\right\}\tilde{Z}_2^2\\
+ \left(c_3 + \frac{b_2^2d_3}{b_4^2} + (t+a_3)\frac{b_2}{b_4}\right)W_3^2
 + (t+a_3)W_3Z_3 
+ d_3Z_3^2 = 0.
\end{array}
$$
Comparing this equation with (\ref{normalpencil}), we see that the curve $C_0$ is
given by the equation
$$
\begin{array}{l}
 z^2 + (t + a_1)(t + a_2) (t + a_3)z  \\
 = (t + a_2)^2(t + a_3)^2\left(c_1 + \frac{b_5^2d_1}{b_8^2} + (t +a_1)\frac{b_5}{b_8}\right)d_1 \\
 +(t + a_1)^2(t + a_3)^2\left\{\frac{(b_8^2c_1 + b_5^2d_1 + (t+a_1)b_5b_8)(b_1b_8+b_4b_6)}{(t+a_1)^2b_4b_8} +\frac{b_4^2b_8c_2 + b_1^2b_8d_2 + (t + a_2)b_1b_4b_8}{(b_1b_8 + b_4b_6)b_4}\right\}\\
 \times \left\{\frac{(b_4^2c_3 + b_2^2d_3 + (t + a_3)b_2b_4)(b_1b_8+b_4b_6)}{(t + a_3)^2b_4b_8} + \frac{b_4b_8^2c_2 + b_4b_6^2d_2 + (t + a_2)b_4b_6b_8}{(b_1b_8 + b_4b_6)b_8}\right\}\\
 + (t + a_1)^2(t + a_2)^2\left(c_3 + \frac{b_2^2d_3}{b_4^2} + (t+a_3)\frac{b_2}{b_4}\right)d_3.
\end{array}
$$
Using $a_1 = a_2 = a_3 = a$, we have
$$
\begin{array}{l}
 z^2 + (t + a)^3z  \\
 = (t + a)^6\frac{b_1b_4b_6b_8}{(b_1b_8+b_4b_6)^2}
 + (t + a)^5\left(\frac{b_5}{b_8}d_1 + \frac{b_2}{b_4}d_3 
 + \frac{b_4b_8c_2 +b_1b_6d_2}{b_1b_8 + b_4b_6}\right)\\
 + (t+a)^4\left\{\left(c_1 +\frac{b_5^2d_1}{b_8^2}\right)d_1 + \left(c_3 + \frac{b_2^2d_3}{b_4^2}\right)d_3 + 
 \frac{(b_4^2c_2 + b_1^2d_2)(b_8^2c_2 + b_6^2d_2)}{(b_1b_8 + b_4b_6)^2} +b_1b_2 + b_5b_6\right\}\\
+ (t +a)^3\left\{\frac{(b_8^2c_1+b_5^2d_1)b_6+ (b_8^2c_2 + b_6^2d_2)b_5}{b_8} + \frac{(b_4^2c_3 + b_2^2d_3)b_1+ (b_4^2c_2 + b_1^2d_2)b_2}{b_4}\right\} \\
+ (t+a)^2\left\{\frac{(b_8^2c_1 + b_5^2d_1)(b_8^2c_2+b_6^2d_2)}{b_8^2} + \frac{(b_4^2c_3 + b_2^2d_3)(b_4^2c_2 + b_1^2d_2)}{b_4^2} +
\frac{(b_1b_8 +b_4b_6)^2b_2b_5}{b_4b_8}\right\}\\
+(t+a)(b_1b_8 + b_4b_6)^2\left\{\frac{(b_8^2c_1 + b_5^2d_1)b_2}{b_4b_8^2} 
+ \frac{(b_4^2c_3 + b_2^2d_3)b_5}{b_4^2b_8}\right\}
+ \frac{(b_1b_8 + b_4b_6)^2(b_8^2c_1 + b_5^2d_1)(b_4^2c_3 + b_2^2d_3)}{b_4^2b_8^2}.
\end{array}
$$
Let $\alpha$ be a root of the equation $x^2 + x + \frac{b_1b_4b_5b_6}{(b_1b_8 + b_4b_6)^2} = 0$.
Replacing $z$ by
$$
z + (t+ a)^3 \alpha + (t+a)^2\left(\frac{b_5}{b_8}d_1 + \frac{b_2}{b_4}d_3 
+ \frac{b_4b_8c_2 + b_1b_6d_2}{b_1b_8 + b_4b_6}\right),
$$
we have proved the theorem.
\qed

\medskip

\begin{remark}\label{CurveSupersingularIgusa}
The equation of Igusa's normal form of $C_0$ is given by
$$
y^2+ y = x^5 + \alpha x^3
$$
with
$$
\begin{array}{l}
\alpha = \frac{\alpha_1}{\sqrt[5]{\alpha_2^3}},\\
\alpha_1 = b_6b_8c_1 +(b_5b_8 +b_2b_4)c_2 + b_1b_4c_3 +b_5b_7d_1 + (b_1b_3 +b_6b_7)d_2 +b_2b_3d_3\\
 +  (b_1b_8 +b_4b_6)\sqrt{c_1c_3} +(b_3b_8 +b_2b_6)\sqrt{c_1d_3} +(b_1b_5 + b_4 b_7)\sqrt{c_3d_1} + (b_3b_5 +b_2b_7)\sqrt{d_1d_3},\\ 
\alpha_2 = (b_1b_3b_8^2 +b_2b_4b_6^2)c_1 +(b_1b_3b_5^2 +b_2b_4b_7^2)d_1 +(b_1^2b_5b_8 +b_4^2b_6b_7)c_3 +(b_3^2b_5b_8 +b_2^2b_6b_7)d_3.
\end{array}
$$
The Jacobian variety $J(C_0)$ is supersingular, that is, of 2-rank 0.
The necessary and sufficient condition for the curve $C_0$ to be non-singular 
is $\alpha_2\neq 0$, that is, 
$(b_1b_3b_8^2 +b_2b_4b_6^2)c_1 +(b_1b_3b_5^2 +b_2b_4b_7^2)d_1 +(b_1^2b_5b_8 +b_4^2b_6b_7)c_3 +(b_3^2b_5b_8 +b_2^2b_6b_7)d_3 \neq 0$,
which
coincides with the condition of the smoothness of $\calX_0$ (Lemma \ref{3-fold}(c))
by $b_1b_2 =b_3b_4$ and $b_5b_6 = b_7b_8$.
It follows from Proposition \ref{Basic2} that 
$S_0$ is a Kummer quartic surface $J(C_0)/\la \iota\ra$.
\end{remark}

\section{Multi-linear systems}\label{sec6}
In the following, let $k$ be an algebraically closed field of characteristic 2, $C$
a curve of genus two over $k$, and $J(C)$ the Jacobian variety of $C$.
Throughout this section except the last subsection \ref{Duquesne1},
we assume that $J(C)$ is ordinary.
In this section, we study linear systems on $J(C)$ which defines a (rational) map from $J(C)$ to a quartic surface $J(C)/\la \iota\ra$ in ${\bf P}^3$ or an intersection of three quadrics in ${\bf P}^5$.  

The curve $C$ gives a principal
polarization of $J(C)$ which we call the {\it theta divisor}. 
We may assume $C$ contains the zero point $O$ of $J(C)$ and is symmetric
with respect to the inversion $\iota$, that is, $\iota^*C = C$.  Since $J(C)$ is ordinary,
the $2$-rank of $J(C)$ is equal to two and 
denoting by $J(C)[2]$ the group scheme of 2-torsion points of $J(C)$, 
we have $J(C)[2]_{red}\cong {\bf Z}/2{\bf Z}\oplus {\bf Z}/2{\bf Z}$.
We denote by $a_i$ ($i = 1, 2, 3$) the 2-torsion points of $J(C)$. 
We may assume that $a_1$ and $a_2$ are contained in $C$ and 
that the points $O$, $a_1$ and $a_2$ give the ramification points of 
the double covering $\pi : C \longrightarrow {\bf P}^1$.
In this case, we have $a_3\not\in C$. We sometimes set $a_0 = O$.

\begin{table}[h]
\begin{tabular}{|c|c|}\hline
curve & elements of $J(C)[2]_{red}$\\
\hline
$C$ &    $O$, $a_1$, $a_2$ \\
\hline
$T_{a_1}^*C$ &    $a_1$, $O$, $a_3$ \\
\hline
$T_{a_2}^*C$ &    $a_2$, $a_3$, $O$ \\
\hline
$T_{a_3}^*C$ &    $a_3$, $a_2$, $a_1$ \\
\hline
\end{tabular}
\caption{The elements of $J(C)[2]_{red}$ which are contained in the curve}
\end{table}

For a point $x \in J(C)$, we denote by $T_x$ the translation by the point $x$
and denote by $\hat{J}(C)$ the dual abelian surface of $J(C)$.
For a divisor $D$, we have a homomorphism
$$
\begin{array}{rccc}
   \Phi_D: & J(C) & \longrightarrow & \hat{J}(C)\\
         &  x  &  \mapsto   &  T_x^*(D) - D
\end{array}
$$
If $D$ is ample, then $\Phi_D$ is an isogeny, and if $D$ is a principal polarization,
then $\Phi_D$ is an isomorphism (cf. Mumford \cite{M1})

\subsection{Points on the theta divisor $C$}
Torsion points of $J(C)$ play an important role to examine linear systems
$\vert nC\vert$ with positive integer $n$. In this subsection we investigate
the relation between the theta divisor $C$ 
and some torsion points of $J(C)$.

\begin{lemma}\label{4-torsion} 
The theta divisor $C$ contains no $4$-torsion point of $J(C)$.
\end{lemma}
\begin{proof}
Suppose there exists an element $x \in C$ of order 4.
Then, $2x = a$ is a 2-torsion point of $J(C)$. 
Assume $a \in C$. Then,
since $C = \iota^*C \ni -x$, 
$T_x^*C$ contains $O$, $-2x=2x = a$ and $O-x = -x$.
Therefore, $C \cap T_x^*C$ contains three different points $O$, $a$ and $-x$. 
If $T_x^*C\neq C$, then we have $3 \leq (T_x^*C\cdot C) = C^2 = 2$, a contradiction.
Therefore, we have $T_x^*C = C$ in this case.
Therefore, we have $x \in \Ker~ \Phi_C$.   However, since $C$ is a principal polarization,
we have $\Ker ~\Phi_C = \{O\}$, a contradiction. 
Hence we have $2x = a = a_3 \not\in C$. 

In this case, $T_x^*C$ contains $x-x=O$, $a_1-x$ and $-x-x=-2x =-a_3= a_3$.
$T_{a_1}^*C$ contains $a_1-a_1=O$, $-x-a_1 = -x+a_1$ and $a_2 - a_1 = a_3$.
If $T_x^*C\neq T_{a_1}^*C$, then we have $3 \leq (T_x^*C\cdot T_{a_1}^*C) = C^2 = 2$, 
a contradiction.
Therefore, we have $T_x^*C = T_{a_1}^*C$. Therefore, we have $T_{x-a_1}^*C = C$.
This means the non-zero element $x-a_1$ is contained in $\Ker~ \Phi_C$.
However, since $C$ is a principal polarization,
we have $\Ker~ \Phi_C = \{O\}$, and we have a contradiction again. 
Hence, $C$ contains no 4-torsion point of $J(C)$.
\end{proof}

\begin{corollary}\label{4-translation} 
Let $x$ be a $4$-torsion point of $J(C)$. Then, $T_x^*C$
contains no point of $J(C)[2]$.
\end{corollary}
\begin{proof}
This corollary follows from Lemma \ref{4-torsion}.
\end{proof}

\begin{lemma}\label{general-transformation}
Let $x$ be a general point of $J(C)$. Then, $T_x^*C$
contains no point of $J(C)[2]$.
\end{lemma}
\begin{proof}
There exists a point $x\in J(C)$ which is not contained
in $\cup_{i=0}^{3}T_{a_i}^*C$. Then, $T_x^*C$
contains no point of $J(C)[2]$.
\end{proof}

The following lemma is well-known.
\begin{lemma}\label{point-divisor}
Let $x_i$ {\rm (}$i = 1,2, \ldots, n${\rm )} be points of $J(C)$. Then,
$$
\sum_{i = 1}^{n}T_{x_i}^*C -nC \sim 0\quad \mbox{if and only if}\quad \sum_{i= 1}^nx_i = O.
$$
\end{lemma}
\begin{proof}
Let $O_{\hat{J}(C)}$ be the zero point of $\hat{J}(C)$.
Since
$\Phi_C(\sum_{i= 1}^nx_i)= \sum_{i= 1}^n\Phi_C(x_i)= \sum_{i = 1}^{n}T_{x_i}^*C -nC$,
we see $\sum_{i = 1}^{n}T_{x_i}^*C -nC\sim 0$ if and only if 
$\Phi_C(\sum_{i= 1}^nx_i) = O_{\hat{J}(C)}$.
Since $C$ is a principal polarization, we see $\Phi_C(\sum_{i= 1}^nx_i) = O_{\hat{J}(C)}$
if and only if $\sum_{i= 1}^nx_i=O$.
\end{proof}

\begin{remark}\label{systemofparameter} 
If $i \neq j$, $T_{a_i}^*C \cap T_{a_j}^*C$ consists of 
different two 2-torsion points. Since $(T_{a_i}^*C \cdot T_{a_j}^*C) = C^2 = 2$,
$T_{a_i}^*C$ intersects $T_{a_j}^*C$ at each intersection point transversely.
Therefore, the defining equations for $T_{a_i}^*C$ and $T_{a_j}^*C$
make the system of parameters at each intersection point.
\end{remark}

\subsection{The linear system $\vert 2C\vert$}\label{7.2}
In this subsection, we examine the linear system $\vert 2C\vert$, and
for our future use we
explain the concrete equation of the image of $\varphi_{\vert 2C\vert}$
in ${\bf P}^3$ which was given by Laszlo-Pauly \cite{LP}.

Let $L(2C)$ be the vector space given by the divisor $2C$.
By the Riemann-Roch theorem, we have $\dim L(2C) = 4$.
By Lemma \ref{point-divisor}, we have
$$
     2(T_{a_i}^*(C) - C) \sim 0 \quad (i = 0, 1, 2, 3).
$$
Therefore, there exists rational functions $f_i$ ($i = 0, 1, 2, 3$) such that
$$
       2(T_{a_i}^*(C) - C) = (f_{3-i}).
$$
We take a natural isomorphism
\begin{equation}\label{equation}
        \rho : L(2C) \cong {\rm H}^0(J(C), \mathcal{O}_{J(C)}(2C))
\end{equation}
and we set $\tilde{f}_i = \rho (f_i)$.

The following Lemmas \ref{basis} and \ref{iota-action} are known (see Laszlo-Pauly
\cite{LP}). But we give their proofs for the sake of completeness.
\begin{lemma}\label{basis}
$f_i$ ($i = 0, 1, 2, 3$) gives a basis of $L(2C)$.
\end{lemma}
\begin{proof} Since $\dim L(2C) = 4$ and (\ref{equation}), it suffices
to show that $\tilde{f}_{i}$ ($i = 0, 1, 2, 3$) are linearly independent over $k$.
We consider a linear relation
$$
  \alpha_0\tilde{f}_0+\alpha_1\tilde{f}_1+\alpha_2\tilde{f}_2+\alpha_3\tilde{f}_3 = 0 
  \quad(\alpha_0, \alpha_1, \alpha_2, \alpha_3 \in k).
$$
Note that $\tilde{f}_i$ has 3 zero points in $J(C)[2]$ and $\tilde{f}_i(a_i) \neq 0$.
Therefore, we have $\alpha_i\tilde{f}_i(a_i) =0$ and $\tilde{f}_i(a_i) \neq 0$.
Therefore, we have $\alpha_i = 0$.
\end{proof}

\begin{remark}
Let $F: J(C) \longrightarrow J(C)^{(2)}$ be the relative Frobenius morphism
and we denote by $\Theta$ the descent divisor of $2C$, i.e., $\rho^{*}(\Theta) \sim 2C$.
We may assume $\Theta$ is an effective divisor. 
Since $\dim {\rm H}^0(J(C)^{(2)}, \mathcal{O}_{J(C)^{(2)}}(\Theta))$ is
one-dimensional, we take a basis $\theta$. Then, we may assume $F^*(\theta) = \tilde{f}_3$,
and we may assume $\tilde{f}_{3-i} = T_{a_i}^*\tilde{f}_3$. Then,
$f_i$ ($i = 0, 1, 2, 3$) is our basis and 
$\tilde{f}_{i}$ ($i = 0, 1, 2, 3$) gives the canonical basis in Laszlo and Pauly \cite{LP}.
From here on, we take our basis $f_i$ ($i = 0, 1, 2, 3$) of $L(2C)$
like this.
\end{remark}

\begin{lemma}\label{iota-action}
The inversion $\iota$ acts on $L(2C)$ identically.
\end{lemma}
\begin{proof}
Since $C$ is symmetric, we have
$$
 (\iota^*f_{3-i}) = 2\iota^*(T_{a_i}^*C - C) = 2(T_{a_i}^*C - C) = (f_{3-i}).
$$
Therefore, there exists $\alpha \in k$, $\alpha \neq 0$ such that 
$\iota^*f_{3-i} = \alpha f_{3-i}$. Since $\iota^*$ is of order 2 and
the characteristic $p = 2$, we have $\alpha = 1$. Therefore, by Lemma \ref{basis}
we complete our proof.
\end{proof}

\begin{remark}\label{quartic}
We consider the map
$$
\begin{array}{rccc}
\varphi_{\vert 2C \vert}: & J(C) & \longrightarrow & {\bf P}^3 \\
        &   P   & \mapsto & (f_3(P), f_2(P), f_1(P), f_0(P))
\end{array}
$$
The divisor $2C$ is base-point-free (cf. Mumford \cite{M1}, for instance). 
Therefore, $\varphi_{\vert 2C \vert}$ is a morphism.
We set $x_{00}=f_3$, $x_{10}= f_2$, $x_{01}=f_1$ and $x_{11} = f_0$.
Then,
by Laszlo and Pauly \cite{LP} the image $\varphi_{\vert 2C \vert}(J(C))$
is given by the equation
$$
\begin{array}{c}
\lambda_{10}^2(x_{00}^2x_{10}^2 + x_{01}^2x_{11}^2) 
+ \lambda_{01}^2(x_{00}^2x_{01}^2 + x_{10}^2x_{11}^2) 
+ \lambda_{11}^2(x_{00}^2x_{11}^2 + x_{01}^2x_{10}^2) \\
 +\frac{\lambda_{10}\lambda_{01}\lambda_{11}}{\lambda_{00}}x_{00}x_{10}x_{01}x_{11} =0
\end{array}
$$
with certain constants $\lambda_{ij}$ $(\lambda_{10}\lambda_{01}\lambda_{11}\lambda_{00} \neq 0)$,
and the image $\varphi_{\vert 2C \vert}(J(C))$ is isomorphic to 
the Kummer quartic surface $J(C)/\langle \iota \rangle$.
\end{remark}

\subsection{The linear system $\vert 4C\vert$}
In Griffths-Harris \cite{GH}, they considered the linear system $\vert 4C\vert$ in characteristic 0
and showed that $\vert 4C - \sum_{i=1}^{16}p_i\vert$ ($p_i$'s are 2-torsion points of $J(C)$)
gives a rational map from $J(C)$ to ${\bf P}^5$. Moreover, they showed that
the image is a complete intersection of three quadrics. In this subsection
we show a similar result in characteristic two. 

We denote by $V$ the subspace of $L(4C)$ generated by $f_if_j$ ($i, j = 0, 1, 2, 3$):
$$
      V = \langle f_if_j ~(i,j = 0, 1, 2, 3) \rangle,
$$
and we set $\tilde{V} = \rho (V)$. 
\begin{lemma}\label{10} 
$\dim V = \dim \tilde{V} = 10$.
\end{lemma}
\begin{proof}
By Remark \ref{quartic}, $f_i$ ($i = 0, 1, 2, 3$) has a quartic relation
and there exist no quadric relations. Therefore,
$f_if_j \in L(4C)$ $(i,j = 0, 1, 2, 3)$ are linearly independent over $k$.
Therefore, we have $\dim V = \dim \tilde{V} = 10$.
\end{proof}

\begin{corollary}\label{6}
$$
\dim \tilde{V} \cap\ {\rm H}^0(J(C), {\mathcal O}_{J(C)}(4C - 2\sum_{i=0}^3a_i)) = 6.
$$
In particular, 
$$
\dim {\rm H}^0(J(C), {\mathcal O}_{J(C)}(4C - 2\sum_{i=0}^3a_i)) \geq 6.
$$
\end{corollary}
\begin{proof}
$\tilde{f}_i\tilde{f}_j$ ($i, j = 0, 1, 2, 3; i\neq j$) are contained in 
${\rm H}^0(J(C), {\mathcal O}_{J(C)}(4C - 2\sum_{i=0}^3a_i))$.
Therefore, we have 
$\dim \tilde{V} \cap {\rm H}^0(J(C), {\mathcal O}_{J(C)}(4C - 2\sum_{i=0}^3a_i)) \geq 6$.
On the other hand, $\tilde{f}_i^2$ ($i = 0, 1, 2, 3$) has 3 zeros in $J(C)[2]$
and $\tilde{f}_i^2(a_i) \neq 0$. Therefore, we have 
$\langle \tilde{f}_i^2 ~(i = 0, 1, 2, 3) \rangle \cap 
{\rm H}^0(J(C), {\mathcal O}_{J(C)}(4C - 2\sum_{i=0}^3a_i)) =\{0\}$, 
from which the result follows.
\end{proof}

We denote by $L(4C)^{\langle \iota^*\rangle}$ the $\iota^*$-invariant subspace of $L(4C)$.

\begin{corollary}\label{morethan10}
$V \subset L(4C)^{\langle \iota^*\rangle}$. 
In particular, $\dim L(4C)^{\langle \iota^*\rangle}\geq 10$.
\end{corollary}
\begin{proof}
This follows from Lemmas \ref{iota-action} and \ref{10}.
\end{proof}

We take a 4-torsion point $x \in J(C)[4]\setminus J(C)[2]$.
Then, by Lemma \ref{point-divisor}, there exists a rational function $\varphi_x$
such that $4(T_x^*C - C) = (\varphi_x)$. Considering the action of $\iota$,
we have
$$
  (\iota^*\varphi_{x}) = 4(T_{-x}^*C - C) = (\varphi_{-x}).
$$
Considering the constant multiple, we can choose 
$\varphi_{-x}$ as $\varphi_{-x} = \iota^*\varphi_{x}$.
For an element $x\in J(C)[2]$, we have already $2(T_x^*C - C) = (f_x)$.
Therefore, we have $4(T_x^*C - C) = (f_x^2)$. In this case we choose
$\varphi_x$ as $\varphi_x = f_x^2$. We denote by $U$ the subspace of $L(4C)$
generated by 16 functions $\varphi_x$. 
Since we work in characteristic 2, the theta group $\mathcal{G}(4C)$ acts on
$U$ (cf. Mumford \cite{M1}). The action of the closed points is given by
$$
(x,\varphi_x)\circ\varphi_y = \varphi_x\circ T_x^*\varphi_y = \mbox{constant}\cdot \varphi_{x+y}
$$
(cf. Mumford \cite{M1}, \cite{M3}).
By Mumford \cite{M2} (also see Sekiguchi \cite{S}) the representation of $\mathcal{G}(4C)$
on $L(4C)$ is irreducible. Therefore, we have $U = L(4C)$.
Hence, we have the following lemma.
\begin{lemma}
$\varphi_x$ $(x\in J(C)[4])$ are a basis of $L(4C)$.
\end{lemma}

\newfont{\bg}{cmr10 scaled\magstep4}
\newcommand{\bigzerol}{\smash{\hbox{\bg 0}}}
\newcommand{\bigzerou}{\smash{\lower1.7ex\hbox{\bg 0}}}

\begin{proposition}
$\dim L(4C)^{\langle \iota^*\rangle} = 10$. In particular, $V= L(4C)^{\langle \iota^*\rangle}$.
\end{proposition}
\begin{proof}
We consider the representation of $\iota^*$ with respect to the basis $\{\varphi_x\}$.
We arrange $\varphi_x$ $(x \in J(C)[2])$ first, and then for a 4-torsion point $x$
we arrange $\varphi_{-x}$ next to $\varphi_x$. Then, the representation matrix
is given by
$$
\left(
\begin{array}{cccccccccc}
1  & & & & & & & & &\bigzerou \\
& 1 & & & & & & & &  \\
& & 1 & & & & & & & \\
& & & 1 & & & & & & \\
& & & & 0 & 1 & & & & \\
& & & & 1 & 0 & & & & \\
& & & &   &  &  \ddots & & &  \\
& & & & & & &  & 0 & 1 \\
\bigzerol & & & & & & & & 1 & 0
\end{array}
\right)
$$
Therefore, we have $\dim L(4C)^{\langle \iota^*\rangle} = 10$.
The latter part follows from Corollary \ref{morethan10}.
\end{proof}

\begin{lemma}\label{16}
$\dim {\rm H}^0(J(C), \mathcal{O}_{J(C)}(4C - \sum_{i=0}^3a_i)) =12$.
\end{lemma}
\begin{proof}
By the Riemann-Roch theorem, we have $\dim {\rm H}^0(J(C), \mathcal{O}_{J(C)}(4C)) = 16$.
We have $\tilde{f}_i^2(a_i) \neq 0$ $(i = 0, 1, 2, 3)$ 
and $\tilde{f}_i^2(a_j) = 0$ for $j\neq i$.
Subtracting a suitable linear combination of these four functions, 
any element of ${\rm H}^0(J(C), \mathcal{O}_{J(C)}(4C))$
becomes zero at all points of $J(C)[2]$. Therefore, we have
$\dim {\rm H}^0(J(C), \mathcal{O}_{J(C)}(4C - \sum_{i=0}^3a_i)) =12$.
\end{proof}
 
Let $O_{a_i}$ be the local ring at the point $a_i$ and $m_{a_i}$ the maximal ideal
of $O_{a_i}$. The cotangent space at the point $a_i$ of $J(C)$ is isomorphic to
$m_{a_i}/m_{a_i}^2$ and it is 2-dimensional. We have a natural exact sequence
$$
0 \longrightarrow {\mathcal O}_{J(C)}(-2\sum_{i=0}^3a_i) \longrightarrow 
{\mathcal O}_{J(C)}(-\sum_{i=0}^3a_i)\longrightarrow 
\oplus_{i = 0}^3m_{a_i}/m_{a_i}^2
\longrightarrow 0.
$$
Tensoring ${\mathcal O}_{J(C)}(4C)$, we have an exact sequence
$$
0 \longrightarrow {\mathcal O}_{J(C)}(4C -2\sum_{i=0}^3a_i) \longrightarrow 
{\mathcal O}_{J(C)}(4C -\sum_{i=0}^3a_i)\longrightarrow 
\oplus_{i = 0}^3m_{a_i}/m_{a_i}^2
\longrightarrow 0.
$$
Therefore, we have a long exact sequence
\begin{equation}\label{exact}
\begin{array}{c}
0 \longrightarrow {\rm H}^0(J(C), {\mathcal O}_{J(C)}(4C -2\sum_{i=0}^3a_i)) \longrightarrow 
{\rm H}^0(J(C), {\mathcal O}_{J(C)}(4C -\sum_{i=0}^3a_i)) \\
\stackrel{\psi}{\longrightarrow}
\oplus_{i = 0}^3m_{a_i}/m_{a_i}^2.
\end{array}
\end{equation}

To calculate the dimension of $\Im ~\psi$,
we construct some elements of ${\rm H}^0(J(C), {\mathcal O}_{J(C)}(4C -\sum_{i=0}^3a_i))$.
For this purpose, we choose different two elements $a_i$, $a_j$ in $J(C)[2]$.
Then, there exist an element $x_k$ of order 4 in $J(C)$ such that $2x_k = a_i + a_j$.
By Lemma \ref{point-divisor}, we see
$$
2T_{x_k}^*C + T_{a_i}^*C + T_{a_j}^*C - 4C \sim 0.
$$
Therefore, there exists a rational function $\varphi_{ij}$ such that
$$
2T_{x_k}^*C + T_{a_i}^*C + T_{a_j}^*C - 4C =(\varphi_{ij}).
$$
We set $\tilde{\varphi}_{ij} = \rho (\varphi_{ij})$.
Note that $T_{x_k}^*C$ contains no element of $J(C)[2]$ by Lemma \ref{4-translation}.
By Table 3 the situation of $\tilde{\varphi}_{ij}$ at the points of $J(C)[2]$ 
is listed as in Table 4.

\begin{table}[h]
\begin{tabular}{|c|c|c|}\hline
 &  simple zero point in $J(C)[2]$ &  double zero point in $J(C)[2]$\\
\hline
$\tilde{\varphi}_{01}$ &    $a_2$, $a_3$ & $a_0$, $a_1$ \\
\hline
$\tilde{\varphi}_{02}$ &    $a_1$, $a_3$ & $a_0$, $a_2$ \\
\hline
$\tilde{\varphi}_{03}$ &    $a_0$, $a_3$ & $a_1$, $a_2$ \\
\hline
$\tilde{\varphi}_{12}$ &    $a_1$, $a_2$ & $a_0$, $a_3$ \\
\hline
$\tilde{\varphi}_{13}$ &    $a_0$, $a_2$ & $a_1$, $a_3$ \\
\hline
$\tilde{\varphi}_{23}$ &    $a_0$, $a_1$ & $a_2$, $a_3$ \\
\hline
\end{tabular}
\caption{Zero points of $\tilde{\varphi}_{ij}$}
\end{table}

Now, let $z$ be a general point of $C$. Since $a_1 + a_2 + a_3 = a_0 = O$,
by Lemma \ref{point-divisor} there exist rational functions $g_0$ and $h_0$
such that
$$
\begin{array}{l}
T_{z+a_1}^*C + T_{-z}^*C + T_{a_2}^*C + T_{a_3}^*C -4C =(g_0),\\
T_{z+a_2}^*C + T_{-z}^*C + T_{a_1}^*C + T_{a_3}^*C -4C =(h_0).
\end{array}
$$
Note that $T_{-z}^*C$ contains no point in $J(C)[2]$
and that $T_{z+ a_1}^*C$ contains only $a_1$ among the points of $J(C)[2]$,
and that $T_{z+ a_2}^*C$ contains only $a_2$ among the points of $J(C)[2]$.
Therefore, the situation of $\tilde{g}_0= \rho (g_0)$ and $\tilde{h}_0 = \rho (h_0)$ at the points of $J(C)[2]$ is as follows.

\begin{table}[h]
\begin{tabular}{|c|c|c|}\hline
 &  simple zero point in $J(C)[2]$ &  double zero point in $J(C)[2]$\\
\hline
$\tilde{g}_0$ &    $a_0$  & $a_1$, $a_2$, $a_3$ \\
\hline
$\tilde{h}_0$ &    $a_0$  & $a_1$, $a_2$, $a_3$ \\
\hline
\end{tabular}
\caption{Zero points of $\tilde{g}_0$, $\tilde{h}_0$}
\end{table}

\begin{lemma}\label{m_{a_0}}
$\tilde{g}_0$ and $\tilde{h}_0$ make a basis of the cotangent space $m_{a_0}/m_{a_0}^2$ at the point $O =a_0$.
\end{lemma}
\begin{proof}
We have $a_0 = O \in T_{a_1}^*C \cap T_{a_2}^*C$. Therefore, this lemma
follows from Remark \ref{systemofparameter}.
\end{proof}

By Table 4, $\tilde{\varphi}_{ij}$ gives non-zero vectors at cotangent spaces of two
points of $J(C)[2]$ and becomes 0-vector at cotangent spaces of the other two
points of $J(C)[2]$. We denote by $v^{(k)}_{ij}$ the non-zero vector given by
$\tilde{\varphi}_{ij}$ at the cotangent space of the point $a_k$.
By Table 3 and Lemma \ref{m_{a_0}}, $\tilde{g}_0$ and $\tilde{h}_0$ make a basis, say  
$\langle v_1, v_2\rangle$, of the cotangent
space $m_{a_0}/m_{a_0}^2$ at $a_0$, and are 0-vectors of the cotangent spaces
at the 2-torsion points.
In Table 5, we summarize the situation of these functions at the cotangent spaces of the points 
of $J(C)[2]$.

\begin{table}[h]
\begin{tabular}{|c|c|c|c|c|}\hline
 & $m_{a_0}/m_{a_0}^2$  &  $m_{a_1}/m_{a_1}^2$ & $m_{a_2}/m_{a_2}^2$ & $m_{a_3}/m_{a_3}^2$\\
\hline
$\tilde{\varphi}_{01}$ &   $0$ & $0$ & $v_{01}^{(2)}$ &  $v_{01}^{(3)}$ \\
\hline
$\tilde{\varphi}_{02}$ &  $0$ & $v_{02}^{(1)}$ & $0$ &  $v_{02}^{(3)}$ \\
\hline
$\tilde{\varphi}_{03}$ &  $v_{03}^{(0)}$ & $0$ & $0$ &  $v_{03}^{(3)}$ \\
\hline
$\tilde{\varphi}_{12}$ &  $0$ & $v_{12}^{(1)}$ & $v_{12}^{(2)}$ &  0  \\
\hline
$\tilde{g}_0$ &  $v_1$ & $0$ & $0$ &  $0$ \\
\hline
$\tilde{h}_0$ &  $v_2$ & $0$ & $0$ &  $0$ \\
\hline
\end{tabular}
\caption{Cotangent vectors}
\end{table}

\begin{lemma}\label{morethan6}
$\dim {\rm Im}~\psi \geq 6$.
\end{lemma}
\begin{proof}
By Remark \ref{systemofparameter}, $T_{a_i}^*C$ intersects $T_{a_j}^*C$ ($j\neq i$)
at different two points and the intersection is transversal.
Therefore, the defining equations of $T_{a_i}^*C$ and $T_{a_j}^*C$
give a basis at the cotangent spaces of the two points.
Therefore, in Table 6, two or two of three non-zero vectors
at the cotangent spaces are linearly independent over $k$.

The 6 functions in Table 6 are contained 
in ${\rm H}^0(J(C), {\mathcal O}_{J(C)}(4C - \sum_{i=0}^3a_i))$.
Suppose that the images of 6 functions in Table 6 by $\psi$ are linearly dependent.
Then, there exist $b_i \in k$ ($i = 1, 2, \ldots, 6$) such that
$$
b_1\psi(\tilde{\varphi}_{01}) + b_2\psi(\tilde{\varphi}_{02}) + b_3\psi(\tilde{\varphi}_{03}) + b_4\psi(\tilde{\varphi}_{12}) + 
b_5\psi(\tilde{g}_0) + b_6\psi(\tilde{h}_0)  = 0
$$
Considering the component of the cotangent space $m_{a_1}/m_{a_1}^2$, we have
$b_2v_{02}^{(1)} + b_4v_{12}^{(1)} = 0$. Since $v_{02}^{(1)}$ and $v_{12}^{(1)}$ are
linearly independent over $k$ as we explained above, we have $b_2 = b_4 = 0$.
Considering the components of the cotangent spaces $m_{a_3}/m_{a_3}^2$ and
$m_{a_0}/m_{a_0}^2$ successively, 
by similar arguments we have $b_i = 0$ for all $i = 1, 2, \cdots, 6$.
\end{proof}

\begin{proposition}
$\dim {\rm H}^0(J(C), {\mathcal O}_{J(C)}(4C - 2\sum_{i=0}^3a_i)) = 6$ and $\dim {\rm Im}~\psi = 6$.
\end{proposition}
\begin{proof}
By Lemmas \ref{16}, \ref{morethan6} and (\ref{exact}),
we have 
$\dim {\rm H}^0(J(C), {\mathcal O}_{J(C)}(4C - 2\sum_{i=0}^3a_i)) \leq 6$.
The former part follows from  Corollary \ref{6}. The latter part follows
from the exact sequence (\ref{exact}).
\end{proof}

We denote by $\tilde{W}$ the subspace of $\tilde{V}$ which is generated by $\tilde{f}_i\tilde{f}_j$ $(i, j = 0, 1, 2, 3; i\neq j)$.

\begin{theorem}\label{4C-2points}
$\tilde{W} = {\rm H}^0(J(C), {\mathcal O}_{J(C)}(4C - 2\sum_{i=0}^3a_i))$.
\end{theorem}
\begin{proof}
We have $\tilde{W} \subset {\rm H}^0(J(C), {\mathcal O}_{J(C)}(4C - 2\sum_{i=0}^3a_i))$.
Since both sides are 6-dimensional, we get our result.
\end{proof}

Using ${\rm H}^0(J(C), {\mathcal O}_{J(C)}(4C - 2\sum_{i=0}^3a_i))$, we consider
the following rational map.
$$
\begin{array}{rccl}
    f :& J(C)  & \longrightarrow & {\bf P}^5 \\
      &P & \mapsto & (X_0, X_1, X_2, X_3, X_4, X_5) 
      = (f_3f_2, f_3f_1, f_3f_0, f_2f_1, f_2f_0, f_1f_0).
\end{array}
$$
Then, by Remark \ref{quartic} we have the relation
\begin{equation}\label{quadratic1}
\lambda_{10}^2(X_0^2 + X_5^2) 
+ \lambda_{01}^2(X_1^2 + X_4^2) 
+ \lambda_{11}^2(X_2^2 + X_3^2)
 +\frac{\lambda_{10}\lambda_{01}\lambda_{11}}{\lambda_{00}}X_0X_5 =0 \quad 
 (\lambda_{10}\lambda_{01}\lambda_{11}\lambda_{00} \neq 0)
\end{equation}

We also have two trivial equations
\begin{equation}\label{quadratic2}
X_0X_5 + X_1X_4 = 0, \quad X_0X_5 + X_2X_3 = 0.
\end{equation}

\begin{theorem}
Let $\Sigma$ be a surface  in ${\bf P}^5$ defined by the equations $(\ref{quadratic1})$,
$(\ref{quadratic2})$.
Then, $\Sigma$ is a K3 surface
with 12 $A_1$-rational double points.
\end{theorem}
\begin{proof} 
By a direct calculation, we have 12 singularities
at such as a point $(\lambda_{01}, \lambda_{10}, 0, 0, 0, 0)$.
By a blowing-up, the singular points can be resolved and the types are
$A_1$-rational double.
The dualizing sheaf of $\Sigma$ is isomorphic to 
${\mathcal O}_{\Sigma}(-6 + 2 + 2+ 2) \cong {\mathcal O}_{\Sigma}$.
Since the singularities are rational, the minimal resolution of $\Sigma$ is a K3 surface.
\end{proof}

\begin{remark}\label{position}
The 12 singular points of $\Sigma$ are given as follows:
$$
\begin{array}{lll}
P_{34}= (0, 0, 0, \lambda_{01}, \lambda_{11}, 0),~ 
P_{24}= (0, 0, \lambda_{01}, 0, \lambda_{11}, 0),\\
P_{13}= (0, \lambda_{11}, 0, \lambda_{01}, 0, 0), ~
P_{12}= (0, \lambda_{11}, \lambda_{01}, 0, 0, 0),\\
P_{35}= (0, 0, 0, \lambda_{10}, 0, \lambda_{11}),~ 
P_{25}= (0, 0, \lambda_{10},0, 0, \lambda_{11}),\\ 
P_{03}= (\lambda_{11}, 0, 0, \lambda_{10}, 0, 0), ~
P_{02}= (\lambda_{11}, 0, \lambda_{10}, 0, 0, 0),\\
P_{45}= (0, 0, 0, 0, \lambda_{10}, \lambda_{01}),~ 
P_{15}= (0, \lambda_{10}, 0, 0, 0, \lambda_{01}),\\
P_{04}= (\lambda_{01}, 0, 0, 0, \lambda_{10}, 0), ~
P_{01}= (\lambda_{01}, \lambda_{10}, 0, 0, 0, 0).
\end{array}
$$
\end{remark}

We have a commutative diagram of rational maps.
$$
\begin{array}{ccl}
J(C) & \stackrel{\varphi_{\vert 2C \vert}}{\longrightarrow} & {\bf P}^3 \ni (f_3, f_2, f_1, f_0)\\
  & f \searrow  \quad   &  \downarrow \varphi   \\
    &     & {\bf P}^5 \ni (f_3f_2, f_3f_1, f_3f_0, f_2f_1, f_2f_0, f_1f_0).
\end{array}
$$
In this diagram, $\varphi_{\vert 2C \vert}$ is a morphism and
${\rm Im}~\varphi_{\vert 2C \vert}$ is isomorphic to $J(C)/\langle \iota \rangle$.

\begin{theorem}
$f$ is a rational map whose base points consist of the points in $J(C)[2]$.
$\varphi$ is a biratinal map whose base points consist of the singular points of
$J(C)/\langle \iota \rangle$. The inverse map $\varphi^{-1}$ is a blow-down morphism
and $\Sigma$ has four exceptional curves with respect to $\varphi^{-1}$. 
Each exceptional curve contains three singular points of $\Sigma$.
\end{theorem}
\begin{proof}
Since $\tilde{f}_i\tilde{f}_j \in {\rm H}^0(J(C), {\mathcal O}_{J(C)}(4C - 2\sum_{i=0}^3a_i))$ 
$(i\neq j)$,
the points in $J(C)[2]$ are base points of $f$. Since some $\tilde{f}_i\tilde{f}_j$ is not zero
outside of $J(C)[2]$, we have the first statement.
By a general theory of Abelian variety, $|2C|$ is base-point-free (cf. Mumford \cite{M1}).
Therefore, we get the second statement. 

The inverse map $\varphi^{-1}$ is given by
$$
(X_0, X_1, X_2, X_3, X_4, X_5) \mapsto (X_0X_1, X_0X_3, X_1X_3, X_0X_5).
$$
If only one component of coordinates of a point $P$ on ${\bf P}^5$ is not zero, then
$P$ is not a point on $\Sigma$. Therefore, for a point $P$, there exist at least two non-zero 
components of coordinates of $P$, say $i$, $j$. Since we can express $\varphi^{-1}$
as a map which includes $X_iX_j$ as a coordinate, $\varphi^{-1}$ is a morphism.
For example, if $i = 3$ and $j = 5$, then we can express $\varphi^{-1}$ as
$$
(X_0, X_1, X_2, X_3, X_4, X_5) \mapsto (X_0X_5, X_3X_4, X_3X_5, X_4X_5),
$$
which coincides with the original one by $X_0X_5 = X_1X_4 = X_2X_3$ (cf. (\ref{quadratic2})).
Therefore, $\varphi^{-1}$ is a morphism. We can show the other statements 
by direct calculations.
\end{proof}

\begin{remark}
We list up the 4 exceptional curves $\ell_i$ $(i = 1, 2, 3, 4)$
for $\varphi^{-1}$ and the singular points on them.
$$
\begin{array}{rc}
   \ell_1 : &X_0 = X_2 = X_4 =0,  \lambda_{10}X_5 + \lambda_{01}X_1 + \lambda_{11}X_3 = 0,\\
        &     \mbox{the singular points on}~\ell_1:\quad P_{12},~P_{15},~P_{35}, \\
   \ell_2 : & X_1 = X_2 = X_5 =0,  \lambda_{10}X_0 + \lambda_{01}X_4 + \lambda_{11}X_3 = 0,\\
        &     \mbox{the singular points on}~\ell_2:\quad P_{04},~P_{03},~P_{34}, \\
   \ell_3 : & X_3 = X_4 = X_5 =0,  \lambda_{10}X_0 + \lambda_{01}X_1 + \lambda_{11}X_2 = 0,\\
        &     \mbox{the singular points on}~\ell_3:\quad P_{01},~P_{02},~P_{12}, \\
   \ell_4 : & X_0 = X_1 = X_3 =0,  \lambda_{10}X_5 + \lambda_{01}X_4 + \lambda_{11}X_2 = 0,\\
        &     \mbox{the singular points on}~\ell_4:\quad P_{24},~P_{25},~P_{45}.
\end{array}
$$
\end{remark}

\subsection{Relations to a paper by Duquesne}\label{Duquesne1}
In his paper \cite{Du}, Duquesne examined the linear system $\vert 2C\vert$ in characteristic 2
and showed that the image of the rational map $\varphi_{\vert 2C \vert}$
is isomorphic to the Kummer surface. In this subsection, we examine
the relation between our theory and his results.

Let $C$ be a non-singular curve of genus 2 and we consider the symmetric product $Sym^{2}(C)$
of two $C$'s. Then, as is well-known, we have morphisms
$$
\begin{array}{rccccc}
 \tilde{\phi} :& C \times C & \longrightarrow & Sym^2(C) & \stackrel{\nu}{\longrightarrow}  & J(C) \\
   & (P_1, P_2) & \mapsto &   P_1 + P_2    & \mapsto  & P_1 + P_2 - K_C.
\end{array}
$$
Here, $\nu$ is the blowing-up at the zero point $O$, and $K_C$ is a canonical divisor of $C$.
We have an inclusion morphism
$$
\phi : C\times \{\infty \} \hookrightarrow C\times C \longrightarrow Sym^2(C)  \stackrel{\nu}{\longrightarrow}   J(C).
$$
Here, $\infty$ is a point of a ramification point of the hyperelliptic structure 
$C \longrightarrow {\bf P}^1$.
We denote by $C_{\infty}$ the image of this inclusion morphism. Then, $C_{\infty}$
gives the principal polarization of $J(C)$.

(a) The ordinary case.

Igusa's normal form of the curve $C$ of genus 2 such that the Jacobian variety 
$J(C)$ is ordinary is given by (\ref{g=2}). It is easy to see that this curve is isomorphic 
to the curve given by
$$
  y^2 +(x^2 + x)y = \alpha x^5 + (\alpha + \beta + \gamma)x^3 + \gamma x^2 + \beta x.
$$
We denote this affine curve by $C_{aff}$ and the point at infinity by $\infty$.
Then, we have
$$
       C = C_{aff} \cup \{\infty\}.
$$
We consider points $P_1 = (1, 0)$ and $P_2 =(0, 0)$, and set
$$
 a_0 = \phi ((\infty, \infty)), a_1= \phi ((P_1, \infty)),
  a_2= \phi((P_2, \infty)),
 a_3= \tilde{\phi} ((P_1, P_2)).
$$
Then, $a_0$ is the zero point of $J(C)$, and $a_i$ ($i = 1, 2, 3$) are the 2-torsion points
of $J(C)$. Note $a_0, a_1, a_2 \in C_{\infty}$ and $a_3 \not\in C_{\infty}$.

Using the notation in Duquesne \cite{Du}, we consider symmetric functions
$$
\begin{array}{l}
k_1= 1, k_2 = x_1 + x_2, k_3 = x_1x_2,\\
k_4 = \frac{(x_1+ x_2)(\alpha x_1^2x_2^2 + (\alpha + \beta + \gamma)x_1x_2 + \beta) + (x_2^2 + x_2)y_1 + (x_1^2 + x_1)y_2}{(x_1 + x_2)^2}.
\end{array}
$$
Then, these four functions give a basis of $L(2C_{\infty})$ and we have a morphism
$$
\begin{array}{rccc}
 \varphi_{\vert 2C_{\infty} \vert} : & J(C) & \longrightarrow & {\bf P}^3 \\
      &   P & \mapsto & (k_1(P), k_2(P), k_3(P), k_4(P)).
\end{array}
$$
The image of 
$\varphi_{\vert 2C_{\infty} \vert}$ is given by the equation
\begin{equation}\label{ordinary-case}
\begin{array}{l}
k_2^2k_4^2 + k_1^2k_3k_4 + k_1k_3^2k_4 + k_1k_2k_3k_4 +
\beta^2k_1^4 \beta k_1^3k_3 +\beta k_1^2k_2k_3 \\
+ 
(\alpha^2 + \beta^2 + \gamma^2 + \alpha + \beta) k_1^2 k_3^2
+ \alpha k_1k_2k_3^2 + \alpha k_1k_3^3 + \alpha^2 k_3^4 = 0
\end{array}
\end{equation}
(cf. Duquesne \cite[Sections 2 and 3]{Du}) and
$\varphi_{\vert 2C_{\infty} \vert}(C_{\infty})$ is a trope defined by $k_1 = 0$.

The singularities of this surface are
$$
\begin{array}{l}
\varphi_{\vert 2C_{\infty} \vert}(a_0)=(0,0,0,1),\quad \varphi_{\vert 2C_{\infty} \vert}(a_1)=(0,1,1,\alpha), \\
\varphi_{\vert 2C_{\infty} \vert}(a_2)=(0,1,0,0),\quad \varphi_{\vert 2C_{\infty} \vert}(a_3)=(1,1, 0, \beta).
\end{array}
$$
They are rational double points of type $D_4$.
We consider the change of coodinates 
$$
   X_1 = k_1, X_2 = k_2 + k_1 + k_3, X_3 = k_3, X_4 =k_4 + \beta k_1 + \alpha k_3.
$$
Using the notation of Subsection \ref{7.2}, we set $C = C_{\infty}$ and $\tilde{X}_i = \rho(X_i)$
($i = 1,2,3,4$). Then, the zero points and the non-zero points of $\tilde{X}_i$
in $J(C)[2]$ are listed as in Table 7.  

\begin{table}[h]
\begin{tabular}{|c|c|c|}\hline
 &  zero points in $J(C)[2]$ &  non-zero points in $J(C)[2]$\\
\hline
$\tilde{X}_{1}$ &    $a_0$, $a_1$, $a_2$     & $a_3$ \\
\hline
$\tilde{X}_{2}$ &    $a_0$,  $a_1$, $a_3$    & $a_2$ \\
\hline
$\tilde{X}_{3}$ &    $a_0$, $a_2$, $a_3$     & $a_1$ \\
\hline
$\tilde{X}_{4}$ &    $a_1$, $a_2$ , $a_3$    & $a_0$\\
\hline
\end{tabular}
\caption{Zero points of $\tilde{X}_i$}
\end{table}

\begin{proposition}
      $X_i$ coincides with $f_{4-i}$ given in Lemma \ref{basis} up to constant.
\end{proposition}
\begin{proof}
Since $f_0$, $f_1$, $f_2$ and $f_3$ make a basis of $L(2C_{\infty})$, there exist elements
$b_j \in k$ ($j = 0,1,2,3$) such that
$X_i = b_0 f_0 + b_1 f_1 + b_2 f_2 + b_3 f_3$. Therefore, we have 
$$
    \tilde{X}_i = b_0\tilde{f}_0 + b_1 \tilde{f}_1 + b_2 \tilde{f}_2 + b_3 \tilde{f}_3.
$$
Taking values at the points $a_0$, $a_1$, $a_2$ and $a_3$, we get the result.
\end{proof}

Using this change of coordinates, the quartic surface (\ref{ordinary-case}) is isomorphic to
the surface defined by
$$
(X_1^2X_4^2 + \alpha^2X_2^2X_3^2) + (X_3^2 X_4^2 + \beta^2X_1^2X_2^2) + (X_2^2X_4^2 + \gamma^2X_1^2X_3^2) + X_1X_2X_3X_4 = 0.
$$
Using the theory in Theorem \ref{4C-2points}, we have a rational map
$$
\begin{array}{rccl}
    f :& J(C)  & \longrightarrow & {\bf P}^5 \\
      &P & \mapsto & (Y_0, Y_1, Y_2, Y_3, Y_4, Y_5) 
      = (X_1X_2, X_1X_3, X_1X_4, X_2X_3, X_2X_4, X_3X_4).
\end{array}
$$
and the image is given by
$$Y_0Y_5 + Y_1Y_4 = Y_0Y_5 + Y_2Y_3 = 
Y_2^2 + \alpha Y_3^2 + Y_5^2 + \beta^2Y_0^2 + Y_4^2 + \gamma^2Y_1^2 + Y_0Y_5 = 0.
$$
This surface is a K3 surface with 12 rational double points of type $A_1$, as we already examined.

(b) 2-rank 1 case.

Igusa's normal form of the curve $C$ of gneus 2 such that the the Jacobian variety 
$J(C)$ is of 2-rank 1 is given by (\ref{g=2}). It is easy to see that this curve is isomorphic 
to the curve defined by
$$
  y^2 +  xy = x^5 + \alpha x^3 +  \beta x.
$$
Using the notation in Duquesne \cite{Du}, we consider symmetric functions
$$
\begin{array}{l}
k_1= 1, k_2 = x_1 + x_2, k_3 = x_1x_2,\\
k_4 = \frac{(x_1+ x_2)(x_1^2x_2^2 + \alpha x_1x_2 + \beta) + x_2y_1 + x_1y_2}{(x_1 + x_2)^2}.
\end{array}
$$
Then, these four functions give a basis of $L(2C_{\infty})$, and the image of 
$\varphi_{\vert 2C_{\infty} \vert}$ is given by the equation
$$
k_2^2k_4^2 + k_1^2k_3k_4 + \beta^2k_1^4 + \alpha^2 k_1^2k_3^2 + k_1k_2k_3^2 + k_3^4 = 0
$$
(cf. Duquesne \cite[Sections 2 and 3]{Du}).
The singularities are
$$
(0, 0, 0, 1),\quad (0, 1, 0, 0),
$$
which are rational double points of type $D_8$.
We set
$$
Y_0 = k_2k_4, Y_1 = k_1^2, Y_2= k_3^2, Y_3 = k_1k_3, Y_4 = k_1k_4, Y_5 = k_2k_3.
$$
Then we have a surface in ${\bf P}^5$ defined by
$$
Y_1Y_2 + Y_3^2 = Y_0Y_3 + Y_4Y_5 = 
Y_0^2 + \beta^2Y_1^2 + Y_2^2 + \alpha^2 Y_3^2  + Y_3Y_4 + Y_3Y_5 = 0.
$$
This surface has 4 singular points. They are rational double points of type $A_{3}$ at
$(0,0,0,0,0,1)$ and $(0,0,0,0,1,0)$, and rational double points of type $D_{4}$ at 
$(1,0,1,0,0,0)$ and $(\beta, 1,0,0,0,0)$.

(c) Supersingular case.

As we stated in (\ref{g=2})  Igusa's normal form of the curve $C$ of genus 2 
such that the Jacobian variety
$J(C)$ is supersingular is given by 
$$
         y^2 + y = x^5 + \alpha x^3.
$$
Using the notation in Duquesne \cite{Du}, we consider symmetric functions
$$
\begin{array}{l}
k_1= 1, k_2 = x_1 + x_2, k_3 = x_1x_2,\\
k_4 = \frac{(x_1+ x_2)(x_1^2x_2^2 + \alpha x_1x_2) + y_1 + y_2}{(x_1 + x_2)^2}.
\end{array}
$$
Then, these four functions give a basis of $L(2C_{\infty})$, and the image of 
$\varphi_{\vert 2C_{\infty} \vert}$ is given by the equation
$$
k_2^2k_4^2 + k_1^3k_4 + \alpha k_1^3k_2 + k_1^2k_2k_3 + \alpha^2k_1^2k_3^2 + k_1k_2^3 + k_3^4 = 0
$$
(cf. Duquesne \cite[Sections 2 and 3]{Du}).
This surface has only one singular point at $(0, 0, 0, 1)$, 
which is an elliptic double point of type $\mbox{\MARU{4}}_{0,1}^{1}$.
We don't know how to connect this surface with a $(2, 2, 2)$-surface in ${\bf P}^5$.


\begin{thebibliography}{99}

\bibitem{B} U.\ Bhosle, \textit{Pencils of quadrics and hyperelliptic curves in characteristic two}, J. reine angew. Math., {\bf 407} (1990), 75--98.

\bibitem{CF} J.W.S.\ Cassels and E.V.\ Flynn, \textit{Prolegomena to a middlebrow arithmetic of curves of genus $2$}, London Math. Soc. Lect. Note Series {\bf 230}, Cambridge Univ. Press 1996.

\bibitem{Die} J.\ Dieudonn\'e, \textit{La G\'om\'etrie des groupes classiques}, Springer-Verlag 1963.

\bibitem{DoC} I.\ Dolgachev, \textit{Classical Algebraic Geometry}, Cambridge Univ. Press 2012.

\bibitem{Do} I.\ Dolgachev, \textit{Kummer surfaces$: 200$ years of study}, Notices of  the American Math. Soc., November 2020, 1527--1533.

\bibitem{Do24} I.\ Dolgachev, \textit{$K3$ surfaces of Kummer type in characteristic two}, Bull. Lond. Math. Soc., {\bf 56} (2024), 1903--1919.

\bibitem{DD} I.\ Dolgachev and A. Duncan, \textit{Regular pairs of quadratic forms on odd-dimensional 
spaces in characteristic} 2, Algebra and Number Theory {\bf 12} (2018), 99--130.

\bibitem{Duc} L.\ Ducrohet, \textit{The action of the Frobenius map on rank $2$ vector bundles over a supersingular genus $2$ curve in characteristic $2$}, Math. Z., {\bf 258} (2008), 477--492.

\bibitem{Du} S.\ Duquesne, \textit{Traces of the group law on the Kummer surface of a curve 
of genus $2$ in characteristic $2$}, Math. Computer Science 3 (2010), 173--183.

\bibitem{GH} P.\ Griffiths and J.\ Harris, \textit{Principles of algebraic geometry}, 
John Wiley and Suns, New York 1978.

\bibitem{H} R.W.H.T.\ Hudson, \textit{Kummer's quartic surface}, Cambridge Univ. Press 1905.

\bibitem{IKO} T.\ Ibukiyama, T.\ Katsura and F. Oort, \textit{Supersingular curves of genus two 
and class numbers}, Compositio Math., {\bf 57} (1986), 127--152.

\bibitem{I} J.\ Igusa, \textit{Arithmetic variety of moduli for genus two}, Ann. Math., {\bf 72} (1960), 612--649.

\bibitem{J} C.\ M.\ Jessop, \textit{A treatise on the line complex}, Cambridge Univ. Press, 1903.

\bibitem{Ka} T.\ Katsura, \textit{On Kummer surfaces in characteristic $2$}, 
M.\ Nagata (ed.), Proceedings of the international symposium on  algebraic geometry, 525--542, Kinokuniya Book Store, Tokyo 1978. 

\bibitem{Ka2} T.\ Katsura, \textit{Generalized Kummer surfaces and their unirationality in characteristic $p$}, J. Fac. Sci. Univ. Tokyo, Sect. IA Math. {\bf 34} (1987), 1--41.

\bibitem{Klein} F.\ Klein, \textit{Zur Theorie der Liniencomplexe des ersten und zweiten Grades}, Math. Ann., {\bf 2} (1870), 198--226.

\bibitem{Kli} W.\ Klingenberg, \textit{Paare symmetrischer und alternierender Formen zweiten Grades}, Abhandlungen aus dem Math. Seminar der Universit\"at Hamburg {\bf 19} (1954), 78--93.

\bibitem{LP} Y.\ Laszlo and C.\ Pauly, \textit{The action of the Frobenius map on rank $2$ vector bundles in characteristic $2$}, J. Algebraic Geometry {\bf 11} (2002), 219--243.

\bibitem{LP2} Y.\ Laszlo and C.\ Pauly, \textit{The Frobenius map, rank $2$ vector bundles and Kummer's quartic surface in characteristic $2$ and $3$}, Adv. Math., {\bf 185} (2004), 246--269.

\bibitem{MO} S.\ Mukai and H. Ohashi, \textit{The automorphism groups of Enriques surfaces covered by symmetric quartic surfaces}, Recent advances in algebraic geometry, 307--320, London Math. Soc. Lect. Note Ser. {\bf 417}, 2015. 
\bibitem{M1} D.\ Mumford, \textit{Abelian Varieties,} \newblock{Oxford Univ. Press, London, 1970.} 

\bibitem{M2} D.\ Mumford, \textit{Varieties defined by quadratic equations}, in Questions on Algebraic Varieties, C.I.M.E., 1969, publ. by Edition Cremonese, 1970.

\bibitem{M3} D.\ Mumford, \textit{Tata Lectures on Theta} I, Progress in Math. 28,
Birkh${\rm \ddot{a}}$user, Boston Basel Stuttgart, 1983.

\bibitem{N} P.\ E.\ Newstead, \textit{Stable bundles of rank 2 and odd degree over a curve of genus} 2, Topology {\bf 7} (1968), 205--215.

\bibitem{NR} M.S.\ Narasimhan and S.\ Ramanan, \textit{Moduli of vector bundles on a compact Riemann surface}, Ann. Math., {\bf 89} (1969), 14--51.

\bibitem{PS} C.\ Peters and J. Stienstra, \textit{A pencil of K3-surfacces related to Ap\'ery's recurrence for $\zeta (3)$ and Fermi surfaces for potential zero}, in: Arithmetic of Complex Manifolds (Erlangen, 1988). 
Lect. Notes in Math., {\bf 1399} (1991), 110--127, Springer.

\bibitem{S} T.\ Sekiguchi, \textit{On projective normality of Abelian varieties} II, J. Math. Soc. Japan 29 (1977), 707--727.

\bibitem{Sh} T.\ Shioda, \textit{Kummer surfaces in characteristic $2$}, 
Proc.\ Japan Acad., {\bf 50} (1974), 718--722.

\bibitem{W} P.\ Wagreich, \textit{Elliptic singularities of surfaces}, 
Amer. J. Math., {\bf 92} (1970), 419--454.

\end{thebibliography}
\end{document}